%% file: thesis.tex
\documentclass[12pt]{report}
\usepackage[dvips]{graphics}
\usepackage{amsmath,amssymb,amsthm,graphicx,color}
\usepackage{cite,multirow}
\usepackage[british]{babel}
\usepackage{enumerate}
\usepackage[linktocpage=true]{hyperref}
\usepackage{longtable}

\newtheorem{theorem}{Theorem}[chapter]
\newtheorem{lemma}[theorem]{Lemma}

\newtheorem{corollary}[theorem]{Corollary}
\newtheorem{proposition}[theorem]{Proposition}

\theoremstyle{remark}

\newtheorem{remark}{Remark}[chapter]
\newtheorem{remarks}[remark]{Remarks}

\allowdisplaybreaks

\numberwithin{equation}{chapter}

\newcommand{\Z}{\mathbb{Z}}
\newcommand{\ZP}{\mathbb{Z}_+}
\newcommand{\R}{\mathbb{R}}
\newcommand{\RP}{\mathbb{R}_+}
\newcommand{\N}{\mathbb{N}}
\renewcommand{\SS}{\mathbb{S}}

\newcommand{\ud}{\textup{d}}
\newcommand{\re}{{\mathrm{e}}}
\newcommand{\rc}{{\mathrm{c}}}

\newcommand{\eps}{\varepsilon}

\newcommand{\ubar}[1]{\underline{#1\mkern-4mu}\mkern4mu }

\renewcommand{\Pr}{\mathbb{P}}
\newcommand{\Exp}{\mathbb{E}\,}
\newcommand{\Var}{\mathbb{V} {\rm ar}}
\newcommand{\hull}{\mathop \mathrm{hull}}
\newcommand{\as}{{\ \mathrm{a.s.}}}
\newcommand{\As}{{\ \mathrm{as}\ }}
\newcommand{\ext}{\mathop \mathrm{ext}}
\newcommand{\Int}{\mathop \mathrm{int}}

\newcommand{\cF}{{\mathcal{F}}}
\newcommand{\cH}{{\mathcal{H}}}
\newcommand{\cA}{{\mathcal{A}}}
\newcommand{\cL}{{\mathcal{L}}}
\newcommand{\cC}{{\mathcal{C}}}
\newcommand{\cK}{{\mathcal{K}}}

\newcommand{\cS}{{\mathcal{S}}}
\newcommand{\cN}{{\mathcal{N}}}

\newcommand{\NN}{{\mathcal{N}}}

\newcommand{\HH}{{\mathcal{H}}}
\newcommand{\FF}{{\mathcal{F}}}
\newcommand{\CC}{{\mathcal{C}}}

\newcommand{\V}{{\mathcal{V}}}

\newcommand{\Sp}{\mathbb{S}}

\newcommand{\leb}{\cA}
\newcommand{\sperp}{\sigma^2_{\mu_\per}}
\newcommand{\spara}{\sigma^2_{\mu}}
\newcommand{\blob}{\mkern 1.5mu \raisebox{1.7pt}{\scalebox{0.4}{$\bullet$}} \mkern 1.5mu}

\DeclareMathOperator*{\argmin}{arg \mkern 1mu \min}
\DeclareMathOperator*{\argmax}{arg \mkern 1mu \max}
 
\DeclareMathOperator{\trace}{tr} 
\newcommand{\tra}{{\scalebox{0.6}{$\top$}}}
\newcommand{\per}{{\mkern -1mu \scalebox{0.5}{$\perp$}}}

\newcommand{\1}{{\,\bf 1}}
\newcommand{\be}{{\bf e}}
\newcommand{\by}{{\bf y}}
\newcommand{\bx}{{\bf x}}
\newcommand{\bz}{{\bf z}}
\newcommand{\0}{{\bf 0}}
\newcommand{\bu}{{\bf u}}
\newcommand{\bv}{{\bf v}}
\newcommand{\bX}{{\bf X}}

\newcommand{\tod}{\stackrel{d}{\longrightarrow}}
\newcommand{\topr}{\stackrel{p}{\longrightarrow}}
\newcommand{\toas}{\stackrel{a.s.}{\longrightarrow}}
\newcommand{\tolp}{\stackrel{L^p}{\longrightarrow}}

\makeatletter
\def\namedlabel#1#2{\begingroup  
    (#2)%
    \def\@currentlabel{#2}%
    \phantomsection\label{#1}\endgroup
}
\makeatother

\begin{document}

\setcounter{secnumdepth}{5}
\setcounter{tocdepth}{3}

\include{titlepage}


\setcounter{page}{1}
\pagenumbering{roman}

\include{acknowledgements}

\include{abstract}

\tableofcontents

\newpage

\include{notation}

\pagenumbering{arabic}

\include{chapter1}

\include{chapter2}

\include{chapter3}

\include{chapter4}

\include{chapter5}

\include{chapter6}

\include{chapter7}


\end{document}

%% file: titlepage.tex
\thispagestyle{empty}

\begin{figure}
\centering \scalebox{1.5}{\includegraphics{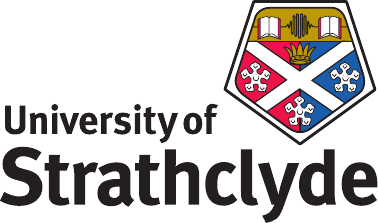}}
\end{figure}
\vspace*{3cm}
\begin{center}
\LARGE{Convex Hulls of Planar Random Walks}

\vspace{3cm}
{\large \date*{\today}}
\bigskip

{\Large{Chang Xu \\
Department of Mathematics and Statistics\\
University of Strathclyde\\
Glasgow, UK}\\
}
\end{center}
\bigskip
\begin{center}\normalsize{This thesis is submitted to the University of Strathclyde for the\\
 degree of Doctor of Philosophy in the Faculty of Science.}
\end{center}

\newpage

\thispagestyle{empty} \noindent The copyright of this thesis belongs
to the author under the terms of the United Kingdom Copyright Acts
as qualified by University of Strathclyde Regulation 3.50. Due
acknowledgement must always be made of the use of any material in,
or derived from, this thesis.

%% file: acknowledgements.tex
\chapter*{Acknowledgements}

My deepest gratitude goes to my supervisor Doctor Andrew Wade and Professor Xuerong Mao. 
This thesis would not exist without their constant support.

I also need to thank Prof. Mikhail Menshikov at Durham University, my current and former colleagues Andreas Wachtel, 
Dr. Hongrui Wang, Dr. Wei Liu, Dr. Jiafeng Pan et al. for plenty of valuable advice and help during my PhD study.

Last, I would like to express my thanks to my family and the University of Strathclyde for the financial support.

%% file: abstract.tex
\chapter*{Abstract}

For the perimeter length $L_n$ and the area $A_n$ of the
convex hull of the first $n$ steps of a planar random walk,
this thesis study $n \to \infty$ mean and variance asymptotics and establish distributional limits. 
The results apply to random walks both with drift (the mean of random walk increments) and with no drift 
under mild moments assumptions on the increments.

Assuming increments of the random walk have finite second moment and non-zero mean,
Snyder and Steele
showed that $n^{-1} L_n$ converges almost surely to a deterministic limit, and proved an upper bound on the variance $\Var [ L_n] = O(n)$.
We show that $n^{-1} \Var [L_n]$ converges and give a simple expression for the limit,
which is non-zero for walks outside a certain degenerate class.
This answers a question of Snyder and Steele.
Furthermore, we prove a central limit theorem for $L_n$ in the non-degenerate case.

Then we focus on the perimeter length with no drift and area with both drift and zero-drift cases. These results complement
and contrast with previous work and establish non-Gaussian distributional limits.
We deduce these results from weak convergence statements for the convex hulls of random walks
to scaling limits defined in terms of convex hulls of certain Brownian motions. We give bounds
that confirm that the limiting variances in our results are non-zero.

%% file: notation.tex
\section*{Notations}

\begin{longtable}[!h]{lllr}
 & & & \quad Page\\
$S_n$, $Z_i$ & : & random walk with location $S_n$ and increments $Z_i$ & \pageref{S_n,Z_i} \\
$\hull ( S_0, \ldots, S_n )$  & : & the convex hull of random walk $S_n$ & \pageref{hull S} \\
$\cS_n$ & : & the random walk $\{ S_0, S_1, \ldots, S_n \}$ & \pageref{cS_n} \\
$L_n$ & : & the perimeter length of $\hull ( S_0, \ldots, S_n )$ & \pageref{L_n} \\
$A_n$ & : & the area of $\hull ( S_0, \ldots, S_n )$ & \pageref{A_n} \\
$\| \blob \|$ & : & the Euclidean norm & \pageref{2-norm} \\
$\mu$ & : & the mean drift vector & \pageref{mu} \\
$\Sigma$ & : & the covariance matrix associated with $Z_i$ & \pageref{Sigma} \\
$\Sigma^{1/2}$ & : & the matrix square-root of $\Sigma$ & \pageref{Sigma^1/2} \\
$\sigma^2$ & : & $= \trace \Sigma$ & \pageref{sigma^2} \\
$\hat \mu$ & : & $= \| \mu \|^{-1} \mu$ for $\mu \neq 0$ & \pageref{hat mu} \\
$\spara$ & : & $= \Exp \left[ \left( ( Z_1 - \mu) \cdot \hat \mu \right)^2 \right]$ & \pageref{spara} \\
$\sperp$ & : & $= \sigma^2 - \spara$ & \pageref{sperp} \\
$\cC ( [0,T] ; \R^d )$ & : & the class of continuous functions from $[0,T]$ to $\R^d$ & \pageref{cC} \\
$\cC^0 ( [0,T] ; \R^d )$ & : & $= \{ f \in \cC ( [0,T] ; \R^d ) : f(0) = \0 \}$ & \pageref{cC^0} \\
$\rho_\infty (\blob\, , \blob)$ & : & the supremum metric & \pageref{rho_infty} \\
$\rho(\bx,A)$ & : & $= \inf_{\by \in A} \rho(\bx,\by)$ for  $A \subseteq \R^d$ and a point $\bx \in \R^d$ & \pageref{rho(x,A)} \\
$\cC_d$ & : & $= \cC ( [0,1] ; \R^d )$  & \pageref{cC_d} \\
$\cC_d^0$ & : & $= \{ f \in \cC_d : f(0) = \0 \}$ & \pageref{cC_d^0} \\
$\SS_{d-1}$ & : & $= \{ \bu \in \R^d : \| \bu \| = 1 \}$, the unit sphere in $\R^d$ & \pageref{cS_d-1} \\
$\cK_d$ & : & the collection of convex compact sets in $\R^d$ & \pageref{cK_d} \\
$\cK_d^0$ & : & $= \{ A \in \cK_d : \0 \in A \}$ & \pageref{cK_d^0} \\
$\rho_H ( \blob\, , \blob )$ & : & the Hausdorff metric & \pageref{rho_H} \\
$\pi_r (\blob)$ & : & the parallel body at distance $r$ & \pageref{pi_r} \\
$b$ & : & $=( b(s) )_{s \in [0,1]}$, standard Brownian motion in $\R^d$ & \pageref{b} \\
$\cA(\blob)$ & : & the area of convex compact sets in the plane & \pageref{cA} \\
$\cL(\blob)$ & : & the perimeter length of convex compact sets & \pageref{cL} \\
$h_A (\blob)$ & : & the support function of $A \in \cK^0_d$ & \pageref{h_A()} \\
$H(f)$ & : & $= \hull \left( f [ 0,1 ] \right)$ for $f \in \cC_d$ & \pageref{H()} \\
$h_t$ & : & the convex hull of the Brownian path up to time $t$ & \pageref{h_t} \\
$\ell_t$ & : & $= \cL ( h_t )$ & \pageref{eqn:def of Lt At for BM} \\
$a_t$ & : & $= \cA ( h_t )$ & \pageref{eqn:def of Lt At for BM} \\
$\Rightarrow$ & : & weak convergence & \pageref{thm:donsker} \\
$w$ & : & $=( w(s) )_{s \in [0,1]}$, standard Brownian motion in $\R$  & \pageref{w} \\
$\tilde b (s)$ & : & $=  (  s , w(s) )$, for $s \in [0,1]$  & \pageref{tilde b} \\
$\tilde h_t$ & : & $= \hull \tilde b [0,t] \in \cK_2^0$  & \pageref{tilde h_t} \\
$\tilde a_t$ & : & $= \cA ( \tilde h_t )$  & \pageref{tilde a_t} \\
$\! \1\{\blob \}$ & : & the indicator function  & \pageref{kac} \\
$x^+$ & : & $= x \1\{ x>0 \}$  & \pageref{x^+} \\
$x^+$ & : & $= -x \1\{ x<0 \}$  & \pageref{x^-} \\
$u_0 ( \Sigma )$ & : & $= \Var \cL ( \Sigma^{1/2} h_1 )$  & \pageref{eq:var_constants} \\
$v_+$, $v_0$ & : & defined in equation \eqref{eq:two_vars10}  & \pageref{eq:two_vars10} \\
\end{longtable}

%% file: chapter1.tex
\pagestyle{myheadings} \markright{\sc Chapter 1}

\chapter{Introduction}
\label{chapter1}

\section{Background on Random Walk}

Let $Z_1, Z_2, \dots$ \label{S_n,Z_i} be independent identically distributed (i.i.d.) random variables taking values in $\R^d$ and
let $S_n = \sum_{i=1}^n Z_i$. $S_n$ is a \emph{random walk} \cite[p.\,88]{gut}.

Random walk theory is a classical and well-studied topic in probability theory. In 1905, Albert Einstein studied the Brownian motion in his paper 
\emph{``On the Movement of Small Particles Suspended in a Stationary Liquid Demanded by the Molecular-Kinetic Theory of Heat''}. 
Brownian motion is the random motion of particles in a fluid which is found by the botanist Robert Brown in 1827 \cite[Sec. 2.1]{hughes}. He noted that the pollen grains in water kept moved through 
randomly. Einstein explained in details how the motion that Brown had observed was a result of the pollen being moved by individual water molecules. 

Scientists then gave the mathematical formalisation for the Brownian motion and its generalisation: random walk. The term \emph{random walk} was first used by Karl Pearson in 1905.
In a letter to Nature, he gave a simple model to describe a mosquito infestation in a forest. At each time step, a single mosquito moves a fixed length in a randomly chosen direction. Pearson wanted to know the distribution of the mosquitoes after many steps had been taken. The letter was answered by Lord Rayleigh, who had already solved a more general form of
this problem in 1880, in the context of sound waves in heterogeneous materials. 
Modelling a sound wave travelling through the material can be thought of as summing up a sequence of random wave-vectors
 of constant amplitude but random phase since sound waves in the material have roughly
constant wavelength, but their directions are altered at scattering sites within the material.

There are some classical results we need to bear in mind when we study random walks. First we need to introduce the concepts of recurrence and transience.
A random walk $S_n$ taking values in $\R^d$ is called \emph{point-recurrent} if
$$\Pr(S_n = 0 \text{ infinitely often}) = 1$$
and \emph{point-transient} if
$$\Pr(S_n = 0 \text{ infinitely often}) = 0.$$
If the random walk is not discrete then these definitions are not very useful. Instead we say that the random walk is
\emph{neighbourhood-recurrent} if for some $\eps > 0$,
$$\Pr(|S_n| < \eps \text{ infinitely often}) = 1 $$
and \emph{neighbourhood-transient} if
$$\Pr(|S_n| < \eps \text{ infinitely often}) = 0 .$$

In the discrete case, for a simple random walk we have the P\'olya's theorem \cite{polya}. A random walk $S_n = \sum_{i=1}^n Z_i$ on $\Z^d$ is \emph{simple} if 
for any $i \in \N$,
$$\Pr(Z_i = e) = \begin{cases}\phantom{2} (2d)^{-1} & \text{if } e \in \Z^d \text{ and } \|e\|=1 , \\
\quad 0 & \text{otherwise} . 
\end{cases}$$

\begin{theorem}[P\'olya]
A simple random walk $S_n = \sum_{i=1}^n Z_i$ in $\Z^d$ is recurrent for $d = 1$ or $d = 2$ and transient for $d \geq 3$. 
\end{theorem}

This theorem was generalised by Chung and Fuchs \cite{chung-fuchs} in 1951.
\begin{theorem}[Chung--Fuchs]
Let $S_n$ be a random walk in $\R^d$. Then,
\begin{enumerate}[(i)]
\item If $d = 1$ and $n^{-1}S_n \to 0$ in probability, then $S_n$ is neighbourhood-recurrent. 
\item If $d = 2$ and $n^{-1/2}S_n$ converges in distribution to a centred normal distribution, then $S_n$ is neighbourhood-recurrent. 
\item If $d \geq 3$ and the random walk is not contained in a lower-dimensional subspace, then it is neighbourhood-transient.
\end{enumerate}
\end{theorem}

\section{Background on geometric probability}

A central theme of classical geometric probability or stochastic geometry concerns the study of the properties of random point sets in Euclidean space and associated structures. For example, a large literature is devoted to study of the lengths of graphs on random vertex sets in Euclidean space $\R^d$, $d \geq 2$. 
The interests are primarily in the lengths of those graphs representing the solutions to problems in Euclidean combinatorial optimization (see \cite{steele2} or \cite{yukich}). 
In the classical setting, the random point sets are generated by i.i.d. random variables.
Some typical problems involve the construction of the shortest possible network of some kind:

Let $X_0, X_1, \dots, X_n$ be i.i.d. random points with common distribution on $\R^d$ and $V=\{X_i\}_{i=0}^n$.
\begin{enumerate}[(i)]
\item Travelling salesman problem. Find the length of shortest closed path traversing each vertex in $V$ exactly once.
\item Minimal spanning tree. Find the minimal total edge length of a spanning tree through $V$.
\item Minimal Euclidean matching. Find the minimal total edge length of a Euclidean matching of points in $V$. 
\end{enumerate}

Many of the questions of geometric probability or stochastic geometry are equally valid for point sets generated by random walk trajectories.

\section{Random convex hulls}

We first define the convex hull here. A set $C$ in $\R^d$ is \emph{convex} if it has the following property \cite[p.\,42]{gruber}:
$$(1 - \lambda)x + \lambda y \in C \ \text{for any } x, y \in C, 0 \leq \lambda \leq 1. $$
Given a set $A$ in $\R^d$, its \emph{convex hull} is the intersection of all convex sets
in $\R^d$ which contain $A$. Since the intersection of convex sets is always convex,
the convex hull of $A$ is convex and it is the smallest convex set in $\R^d$ with respect to set inclusion, which contains $A$. 

One of the motivations to study the convex hulls is to find the extreme values in the random points. For the 1-dimensional case, the extreme values are just the maximum and minimum values. For higher dimensional cases, the extreme values could be determined by the convex hulls.

However, the extreme values have different meanings in these two different main settings of classical stochastic geometry. 
For the setting of i.i.d. random points, one important concern is the outlier detection in random sample. 
For the setting of trajectories of stochastic processes, extremes are important for study of record values. 
It gives two related but different streams of research, with different underlying probabilistic models and different motivating questions, 
though generally the motivations are all comes from multidimensional theory of extremes. See for example \cite{barnett1}, \cite{barnett2}, \cite{barnett-lewis} and \cite{nevzorov}.

\subsection{i.i.d. random points}

Convex hulls of iid. random points, also known as \emph{random polytopes}, were first studied by Geffroy \cite{geffroy} (1961), R\'enyi and Sulanke \cite{renyi-sulanke} (1963), and Efron \cite{efron} (1965). In the case where the points are normally distributed, the resulting convex hulls are known as \emph{Gaussian polytopes}. See Reitzner \cite[Random polytopes, pp.\,45-76]{reitzner} (2010) and Hug \cite{hug} (2013) for recent surveys.

Motivation arises in statistics (multivariate extremes) and convex geometry (approximation of convex sets), and there are connection to the isotropic constant in functional analysis: see Reitzner \cite{reitzner}. He also listed some other applications including to the analysis of algorithms and optimization.

For the multivariate extremes, let $X_0, X_1, \dots, X_n$ be the iid. random points with common distribution on $\R^d$ and $V=\{X_i\}_{i=0}^n$.
In the case of $d=1$, iid. points extremes are used in outlier detection in statistics.
In the case of $d \geq 2$, Green \cite{green} describes the peeling algorithm for detection of multivariate outliers via the iterated removal of points on the boundary of the convex hulls.

For the approximation of convex sets, Reitzner \cite{reitzner} insulates the algorithms to efficiently compute convex hull for large point set in $\R^d$.

\subsection{Trajectories of stochastic process}

Before the study of random polytopes, L\'evy \cite{levybook} had considered the convex hull of planar Brownian motion. The study of convex hull of random walk goes back to Spitzer and Widom \cite{sw}. Generally, the convex hull of a stochastic process is an interesting geometrical object, related to extremes of the stochastic processes, giving a multivariate analogue of \emph{record values}. 

In one dimension, a value of a process is a record value if it is either less than all previous values (a lower record) or greater than all previous values (an upper record). In higher dimensions, a natural definition of ``record'' is then a point that lies outside the convex hull of all previous values.

More recent work on convex hull of Brownian motion includes Burdzy \cite{burdzy} (1985), Cranston, Hsu and March \cite{chm} (1989), Eldan \cite{eldan} (2014), Evans \cite{evans} (1985), Pitman and Ross \cite{pitman-ross} (2012). 

For general stochastic processes, convex hulls and related convex \emph{minorants} or \emph{majorants}, are studied by Bass \cite{bass} (1982) and Sinai \cite{sinai} (1998).

\section{Applications for convex hulls of random walks}

In recent studies of random walks, attention has focussed on various geometrical aspects of random walk trajectories.
Many of the questions of stochastic geometry, traditionally concerned with functionals of independent random points, are also of interest for point sets generated by random walks. 

Study of the convex hull of planar random walk goes back to Spitzer and Widom \cite{sw}
 and the continuum analogue, convex hull of planar Brownian motion,
to L\'evy \cite[\S52.6, pp.~254--256]{levybook}; both have received renewed interest recently, in part motivated by
applications arising for example in modelling the `home range' of animals.
Random walks have  been extensively used to model the movement of animals; Karl Pearson's original
motivation for the random walk problem originated with modelling the migration of animal
species such as mosquitoes, and subsequently random walks have been used to model the locomotion
of microbes: see~\cite{codling,smouse} for surveys. If the trajectory of the random walker
represents the locations visited by a roaming animal,
then the convex hull is a natural estimate of the `home range' of the animal~\cite{worton1,worton2}.
Natural properties of interest are the perimeter length and area of the convex hull.
See \cite{mcr} for a recent survey
of motivation and previous work. 
The method of Chapter \ref{chapter3} in part relies on an analysis of \emph{scaling limits},
and thus links the discrete and continuum settings.

\section{Introduction of the model}
\label{sec:intro}

On each unsteady step, a drunken gardener deposits one of $n$ seeds. Once the flowers have bloomed, what is the minimum length of fencing required to enclose the garden?

Let $Z_1, Z_2, \ldots$ be a sequence of independent, identically distributed (i.i.d.)
random vectors on $\R^2$. Write $\0$ for the origin in $\R^2$.
Define the random walk $(S_n; n \in \ZP)$ by $S_0 := \0$
and for $n \geq 1$, $S_n := \sum_{i=1}^n Z_i$.
Let
$\hull ( S_0, \ldots, S_n )$  \label{hull S} 
be the convex hull of positions of the walk up to and including the $n$th step, which is the smallest convex set that contains $S_0, S_1 \ldots, S_n$.
Let $L_n $ \label{L_n} denote the length of the perimeter of $\hull ( S_0, \ldots, S_n )$ and $A_n $ \label{A_n} be the area of the convex hull. (See Figure \ref{fig:Intro}.)

\begin{figure}[h!]
\center
\includegraphics[width=0.75\textwidth]{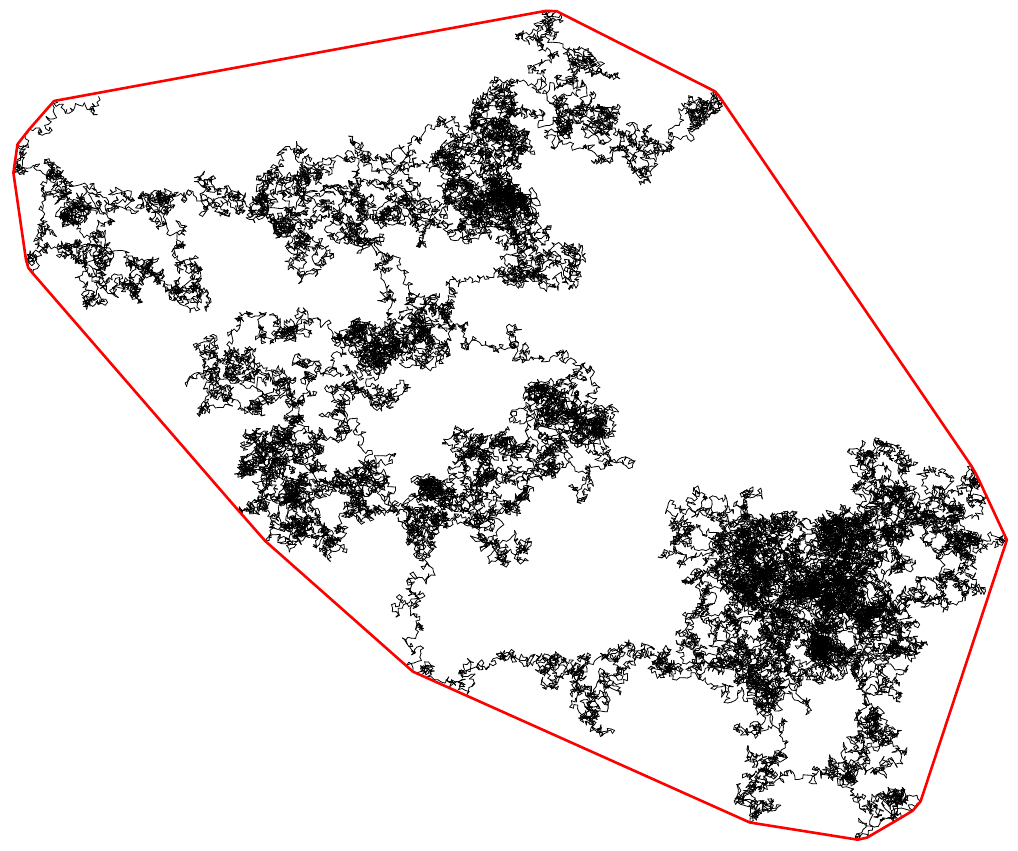}
\caption{Simulated path of a zero-drift random walk and its convex hull.}
\label{fig:Intro}
\end{figure} 

We will impose a moments condition of the following form:
\begin{description}
\item[\namedlabel{ass:moments}{M$_p$}]
Suppose that $\Exp [ \| Z_1 \|^p ] < \infty$.
\end{description} \label{2-norm}
For almost everything that follows, we will assume that at least the
$p=1$ case of \eqref{ass:moments} holds, and frequently we
will assume the $p = 2$ case.
For several of our results we 
assume that \eqref{ass:moments} holds for some $p>2$.
In any case, we will be explicit about which case we
assume at any particular point.

Given that \eqref{ass:moments} holds for some $p \geq 1$, 
then
 $\mu := \Exp Z_1 \in \R^2$, \label{mu} the mean drift vector
of the walk, is well defined.
If \eqref{ass:moments} holds for some $p \geq 2$, 
then
$\Sigma := \Exp [ (Z_1 - \mu)(Z_1-\mu)^\tra]$, \label{Sigma} the covariance
matrix associated with $Z$, is well defined;
 $\Sigma$ is positive semidefinite and symmetric.
We write $\sigma^2 := \trace \Sigma = \Exp [ \| Z_1 - \mu \|^2 ]$. \label{sigma^2}  Here and elsewhere   $Z_1$ and $\mu$ are viewed as column vectors, and $\| \blob \|$ is the Euclidean norm. 
We also introduce the decomposition $\sigma^2 = \spara + \sperp$ with \label{sperp}
\[ \spara := \Exp \left[ \left( ( Z_1 - \mu) \cdot \hat \mu \right)^2 \right] = \Exp [ ( Z_1 \cdot \hat \mu )^2 ] - \| \mu \|^2 \in \RP.\] \label{spara}
Here and elsewhere,  `$\cdot$' denotes the scalar product,  $\hat \mu := \| \mu \|^{-1} \mu$ for $\mu \neq 0$, \label{hat mu}
and $\RP := [0,\infty)$. 

\begin{figure}[h!]
  \centering
	\includegraphics[width=0.99\textwidth]{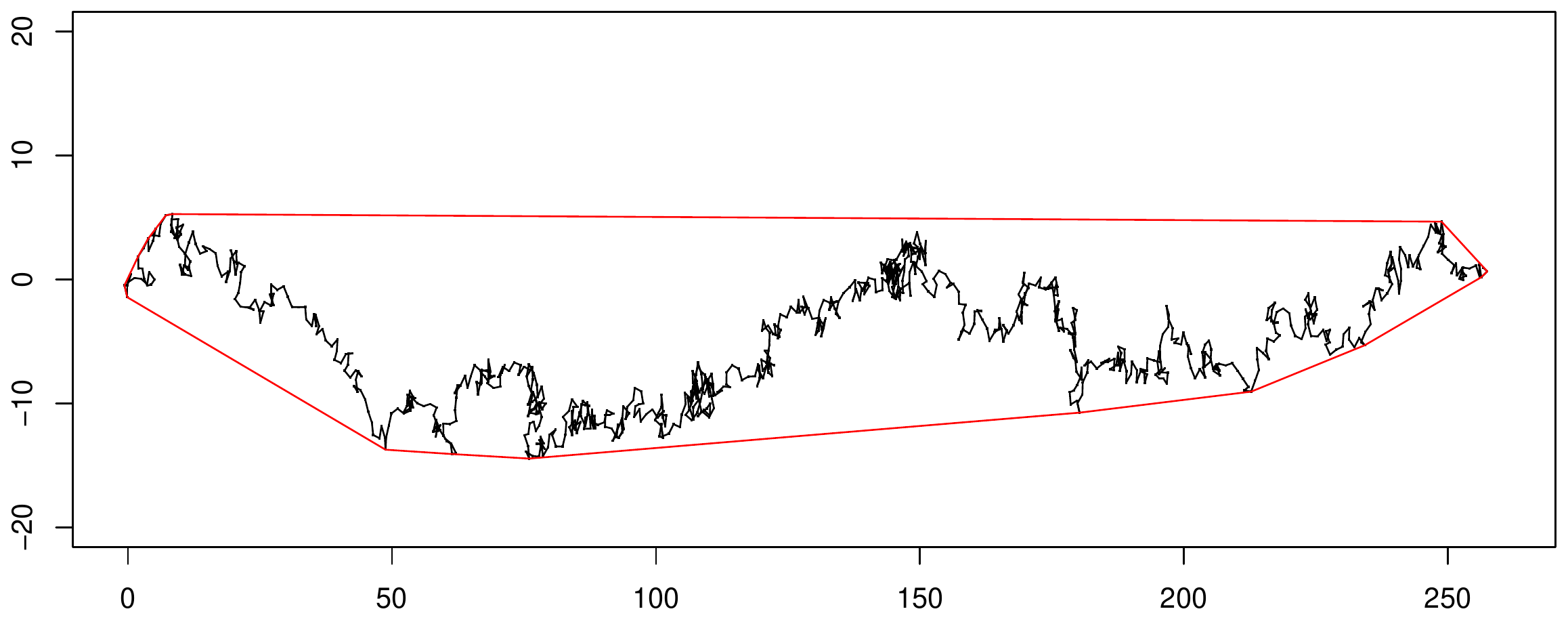}
  \caption{Example with mean drift $\Exp [ Z_1]$ of magnitude $\|\mu\| = 1/4$ and $n = 10^3$ steps.}\label{fig1}
\end{figure}

Convex hulls of random points have received much attention over the last several decades: see \cite{mcr} for an extensive survey,
 including more than 150 bibliographic references, and sources
of motivation more serious than our drunken gardener, such as modelling the `home-range' of animal populations.
An important tool in the study of random convex hulls
is provided by  a result  of Cauchy in classical convex geometry.
Spitzer and Widom \cite{sw}, using Cauchy's formula, and later Baxter \cite{baxter}, using a combinatorial argument,
showed that
\begin{equation}
\label{SW formula}
\Exp [ L_n ] = 2 \sum_{i=1}^n \frac{1}{i} \Exp \| S_i \| . \end{equation}
Note that  $\Exp [ L_n]$ thus scales like $n$ in the case where the one-step mean drift vector
$\Exp [ Z_1 ] \neq \0$ but like $n^{1/2}$
in the  case where $\Exp [ Z_1 ] = \0$ (provided $\Exp [ \| Z_1 \|^2 ] < \infty$).
The Spitzer--Widdom--Baxter result, in common with much of the literature,
is concerned with first-order properties of $L_n$:
see \cite{mcr} for a summary
of results in this direction for various random convex hulls, with a specific focus on (driftless) planar Brownian motion.

Much less is known about higher-order properties of $L_n$. 
There is a clear distinction between the zero drift case ($\Exp[Z_1]=\0$) and the non-zero drift case ($\| \Exp[ Z_1 ] \| > 0$). For example, denote $r_n := \inf_{\bx \in \partial \text{hull}(S_0, \dots, S_n)} \| \bx \|$. Note that $r_n$ is non decreasing in $n$, because $S_0 = \0 \in \text{hull}(S_0, \dots, S_n) \subseteq \text{hull}(S_0, \dots, S_{n+1})$. We investigated the asymptotic behaviour of $r_n$ in the following two different cases.

\begin{proposition}
\begin{enumerate}[(i)]
\item 
Suppose $\Exp [ \| Z_1 \|^2 ] < \infty$ and $\Exp [Z_1]= \0$. Then $\lim_{n \to \infty} r_n = \infty$ a.s.
\item
Suppose $\Exp \| Z_1 \| < \infty$ and $\Exp [Z_1]\neq \0$. Then $\lim_{n \to \infty} r_n < \infty$ a.s.
\end{enumerate}
\end{proposition}

\begin{proof}
\begin{enumerate}[(i)]
\item
In the first case, the random walk ($S_n; n \in \Z_+$) is recurrent (see e.g. \cite{durrett}). There exists $h \in \R_+$, depending on the distributioon of $Z_1$, such that $S_n$ will visit any ball of radius at least $h$ infinitely often (e.g., in the case of simple symmetric random walk on $\Z^2$, it suffices to take $h=1$). Let $r>0$. Then, $S_n$ will visit $B((r+h)\by; h)$ infinitely often for each $\by \in \{(1,1),(-1,1),(1,-1),(-1,-1)\}$. Here the notation $B(\bx; r)$ is a Euclidean ball (a disk) with centre $\bx \in \R^2$ and radius $r \in \R_+$.

So there exists some random time $N$ with $N < \infty$ a.s. such that $\{S_0,\dots,S_N\}$ contains a point in each of these four balls, and so $\text{hull}(S_0,\dots,S_N)$ contains the square with these points as its corners, which in turn contains $B(\0; r)$. So $\liminf_{n \to \infty} r_n \geq r$ for any $r \in \R_+$. So $\lim_{n \to \infty} r_n =\infty$. 

\item
In the second case, the random walk is transient (see \cite{durrett}). Let $W_i$ be a wedge with apex $S_i$ with a angle $\theta < \pi$ (say $\theta = \pi/4$) so that $\theta$ is bisected by $\Exp Z_1$. By the Strong Law of Large Numbers, $\| S_n /n - \Exp Z_1 \| \to 0 \as$ and so $S_n/n \cdot \Exp Z_1^{\perp} \to 0 \as$, where $\Exp Z_1^{\perp}$ is the normal vector of $\Exp Z_1$. This implies the number of points outside the wedge $W_i$ is finite for any $i \in \Z^+$. We take some $S_k$ inside the wedge $W_0$ and denote the set of finitely many points outside $W_k$ by $\{ S_{\sigma_j}: j=1,2,\ldots,m \}$. Note that $S_0$ is outside $W_k$ so the set $\{ S_{\sigma_j} \}$ is non-empty. Hence, there must be some $S_{\sigma_t} \in \{ S_{\sigma_j} \}$ standing on the boundary of the convex hull, $S_{\sigma_t} \in \partial \text{hull}(S_0, \dots, S_n)$ for all $n \geq \sigma_t$. Then, $\limsup_{n \to \infty} r_n \leq \| S_{\sigma_t} \| < \infty$, which implies $\lim_{n \to \infty} r_n < \infty$ a.s. since $r_n$ is non decreasing.
\end{enumerate}
\end{proof}

\begin{remark}
The key property for (i) is not (compact set) recurrence, but \emph{angular recurrence} in the sense that $S_n$ visits any cone with apex at $\0$ and non-zero angle infinitely often. Thus the same distinction between (i) and (ii) persists for random walks in $\R^d$, $d \geq 3$, with the notation extended in the natural way.
\end{remark}

Because of this distinction, we always separate the arguments of $L_n$ and $A_n$ into the cases of non-zero and zero drift.

To illustrate our model, here we give some pictures of simulation examples (see Figure \ref{sim123}).
\begin{figure}[!h]
	\includegraphics[width=0.43\textwidth]{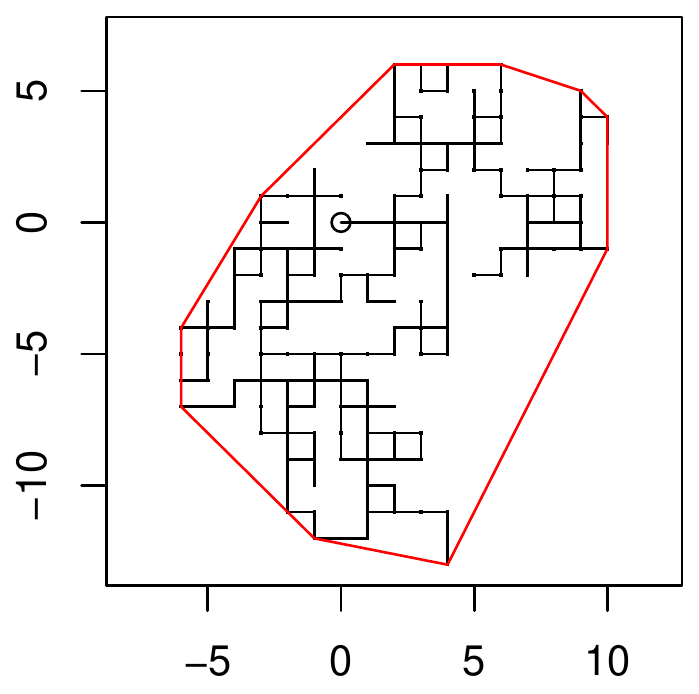} ~~
	\includegraphics[width=0.43\textwidth]{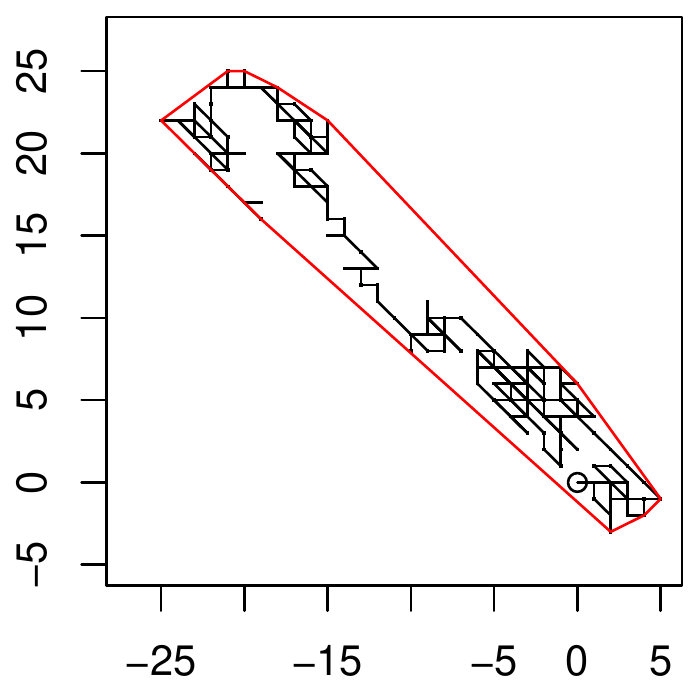} \vspace{1em} \\
	\includegraphics[width=0.43\textwidth]{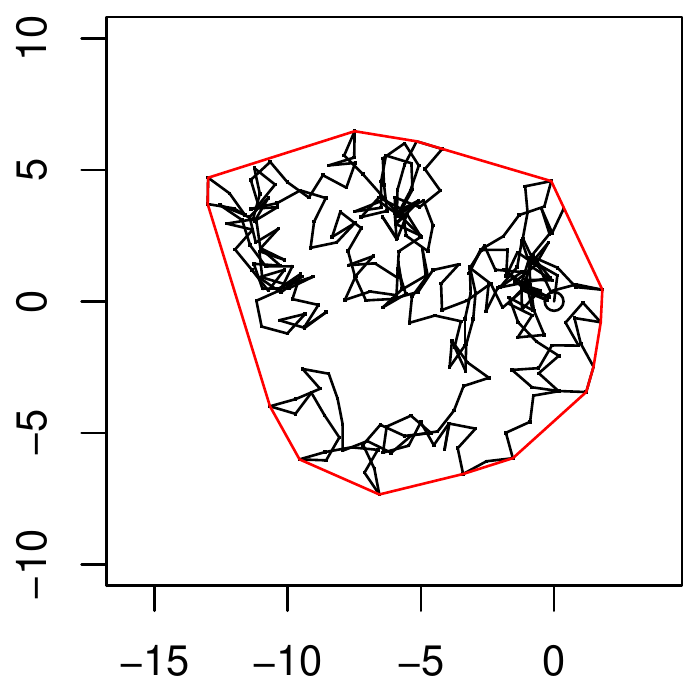}
  \caption{The number of steps $n=300$ for all three examples. The top left: Simple random walk on $\Z^2$. $Z_i$ takes $(\pm 1,0)$, $(0,\pm 1)$ each with probability 1/4. \newline
           The top right: $Z_i$ takes $(\pm 1,0)$, $(0,\pm 1)$, $(-1,1)$, $(1,-1)$ each with probability 1/6. \newline
           The bottom left: Pearson--Rayleigh random walk. $Z_i$ takes value uniformly on the unit circle. }\label{sim123}
\end{figure}

\section{Outline of the thesis}

Chapter \ref{chapter2} is some necessary mathematical prerequisites for our results. It includes the concepts of the study objects and the essential tools used in the rest chapters.

In Chapter \ref{chapter3}
we describe our scaling limit approach, and carry it through after presenting the necessary preliminaries;
the main new results of this chapter, Theorems \ref{thm:limit-zero} and \ref{thm:limit-drift},
give weak convergence statements for convex hulls of random walks in the case of zero and non-zero drift, respectively.
Armed with these weak convergence results, we present 
asymptotics for expectations and variances of the quantities $L_n$ and $A_n$ in Section 5.4, 6.4 and 6.5;
the arguments in this section rely in part on the scaling limit apparatus, and in part on direct random walk computations.
This section concludes with upper and lower bounds we found for the limiting variances. 

Snyder and Steele \cite{ss} showed that $n^{-1}L_n$ converges almost surely to a deterministic limit, and proved an upper bound on the variance $\Var[L_n]=O(n)$ \cite{ss}.
In Chapter \ref{chapter4}, we give a different approach to prove their major results, which includes 
the fact that $n^{-1}\Exp[L_n]$ converges (Proposition \ref{upper bound of E L_n}) and a simple expression for the limit in Proposition \ref{EL with drift}.
For the zero drift case, we give a new improved limit expression in Proposition \ref{limit of E L_n}.

Chapter \ref{chapter5} gives the convergence of $n^{-1}\Var[L_n]$ in Proposition \ref{upper bound for Var L_n}, which is first proved by Snyder and Steele \cite{ss}. 
They also gave the law of large numbers for $L_n$ in the non-zero drift case. But we found it also valid for the zero drift case (Proposition \ref{LLN for L_n}). 
Apart from that, the following of major results in this chapter are new.
For the non-zero drift case, we give a simple expression for the limit of $n^{-1}\Var[L_n]$ in Theorem \ref{thm1} \cite[Theorem 1.1]{wx}, which is non-zero for walks outside a certain degenerate class. This answers a question of Snyder and Steele. It is also the only case where the perimeter length $L_n$ is Gaussian. 
So we give a central limit theorem for $L_n$ in this case in Theorem \ref{thm2} \cite[Theorem 1.2]{wx}.
For the non-zero drift case, the limit expression of $n^{-1}\Var[L_n]$ is given in Proposition \ref{prop:var-limit-zero u0} \cite[Proposition 3.5]{wx2} and its upper and lower bounds are given by Proposition \ref{prop:var_bounds u0} \cite[Proposition 3.7]{wx2}.

Chapter \ref{chapter6} is an analogue of Chapter \ref{chapter5} for the area $A_n$. 
In Theorem \ref{prop:EA-zero} we give the asymptotic for the expected area $\Exp A_n$ with zero drift, which is a bit more general than the form given by Barndorff--Nielsen and Baxter \cite{bnb}.
Apart from that, the following of major results in this chapter are new.
We give the asymptotic for the expected area $\Exp A_n$ with drift in Proposition \ref{prop:EA-drift} \cite[Proposition 3.4]{wx2} and also
the asymptotics for their variance $\Var A_n$ in both zero drift (Proposition \ref{prop:var-limit-zero} \cite[Proposition 3.5]{wx2}) and non-zero drift cases (Proposition \ref{prop:var-limit-drift v+} \cite[Proposition 3.6]{wx2}).
Meanwhile, some upper and lower variance bounds are provided by the last section of this chapter.

%% file: chapter2.tex
\pagestyle{myheadings} \markright{\sc Chapter 2}

\chapter{Mathematical prerequisites}
\label{chapter2}

\section{Convergence of random variables}
\label{sec:convergence of random variables}

First of all, we define the different modes of convergence we will need in this thesis.

Let $X$ and $X_1, X_2, \dots $ be random variables in $\R$. 

$X_n$ converges \emph{almost surely} to $X$ ($X_n \toas X$) as $n \to \infty$ iff
$$\Pr\left(\{ \omega : X_n(\omega) \to X(\omega) \As n \to \infty\}\right) = 1 .$$

$X_n$ converges \emph{in probability} to $X$ ($X_n \topr X$) as $n \to \infty$ iff, for every $\eps > 0$,
$$\Pr\left(|X_n - X| > \eps \right) \to 0 \As n \to \infty .$$

The \emph{$L^p$ norm} of $X$ is defined by 
$$\|X\|_p := \left( \Exp |X|^p \right)^{1/p}.$$
$X_n$ converges \emph{in $L^p$} to $X$ ($X_n \tolp X$) for $p \geq 1$, as $n \to \infty$ iff
$$\Exp \left(|X_n - X|^p \right) \to 0,\text{ i.e. }\|X_n - X\|_p \to 0, \As n \to \infty .$$

Let $F_X(x) = \Pr(X \leq x), x \in \R$, be the distribution function of $X$ and let $C(F_X) = \{x : F_X(x) \text{ is continuous at } x\}$ be the continuity set of $F_X$. $X_n$ converges \emph{in distribution} to $X$ ($X_n \tod X$) as $n \to \infty$ iff
$$F_{X_n}(x) \to F_X(x) \As n \to \infty, \text{ for all } x \in C(F_X) .$$
The concept of convergence in distribution extends to random variables in $\R^d$ in terms of the joint distribution functions $\Pr[X_n^{(1)} \leq x^{(1)}, \dots, X_n^{(d)} \leq x^{(d)}]$.

These modes of convergence have the following logical relationships. 
\begin{align*}
X_n \tolp X & \,\rotatebox{-45}{$\Longrightarrow$}   \\
 & \qquad X_n \topr X \Longrightarrow X_n \tod X \\
X_n \toas X & \; \rotatebox{45}{$\Longrightarrow$}
\end{align*}

Now we collect some basic results on deducing convergence lemmas and theorems.

\begin{lemma}[Dominated convergence \cite{gut} p.57] \label{dominated convergence}
Let $X, Y$ and $X_1, X_2, \dots$ be random variables. Suppose that $|X_n| \leq Y$ for all $n$, where $\Exp Y < \infty$, and that $X_n \to X$ a.s. as $n \to \infty$. Then
$$\Exp |X_n - X| \to 0 \As n \to \infty ,$$
In particular,
$$\Exp X_n \to \Exp X \As n \to \infty .$$
\end{lemma}

\begin{lemma}[Pratt's lemma \cite{gut} p.221] \label{pratt's lemma}
Let $X$ and $X_1, X_2, \dots$ be random variables. Suppose that $X_n \to X$ almost surely as $n \to \infty$, and that
$$ |X_n| \leq Y_n \text{ for all } n,\quad Y_n \to Y \as,\quad \Exp Y_n \to \Exp Y \As n \to \infty.$$
Then
$$ X_n \to X \text{ in } L^1 \quad\text{and}\quad \Exp X_n \to \Exp X \As n \to \infty .$$
\end{lemma}

\begin{lemma}[The Borel--Cantelli lemma \cite{gut} p.96, 98] \label{borel-cantelli}
Let $\{A_n, n \geq 1\}$ be arbitrary events. Then
$$ \sum_{n=1}^{\infty} \Pr(A_n) < \infty \Longrightarrow \Pr(A_n \ \text{i.o.}) = 0 .$$
Moreover, 
suppose that $X_1, X_2, \dots $ are random variables. Then,
$$ \sum_{n=1}^{\infty} \Pr(|X_n| > \eps) < \infty \text{ for any } \eps > 0 \Longrightarrow X_n \to 0 \as \As n \to \infty  .$$
\end{lemma}

\begin{lemma}[Slutsky's theorem \cite{gut} p.249] \label{slutsky}
Let $X_1, X_2, \dots$ and $Y_1, Y_2, \dots$ be sequences of random variables,  Suppose that
$$X_n \tod X \text{ and } Y_n \topr a \As n \to \infty ,$$
where $a$ is some constant. Then,
$$X_n + Y_n \tod X +a \text{ and } X_n \cdot Y_n \tod X \cdot a .$$
\end{lemma}

Here we also introduce some useful concepts of uniform integrability.

A collection of random variables $X_i$, $i \in I$, is said to be \emph{uniformly integrable} if
$$\lim_{M \to \infty} \left(\sup_{i\in I} \Exp(|X_i| \1( |X_i| > M))\right) = 0 .$$

\begin{lemma}
Let $X$ and $X_1, X_2, \dots$ be random variables. If $X_n \to X$ in probability then the following are equivalent:
\begin{enumerate}[(i)]
\item $\{X_n\}_{i=1}^\infty$ is uniformly integrable.
\item $X_n \to X$ in $L^1$.
\item $\Exp |X_n| \to \Exp |X| < \infty$.
\end{enumerate}
\end{lemma}

\begin{lemma}[convergence of means \cite{kallenberg} p.45] \label{convergence of means}
Let $X, X_1, X_2, \dots$ be $\RP$-valued random variables with $X_n \tod X$. If $\{X_i\}_{i=1}^{\infty}$ is uniformly integrable, then $\Exp X_n \to \Exp X$ as $n \to \infty$.
\end{lemma}

\section{Martingales}
\label{sec:martingales}

A sequence $\{X_n\}_{i=1}^\infty$ of random variables is \emph{$\{\cF_n\}$-adapted} if $X_n$ is $\cF_n$-measurable for all $n$, which means for any $k \in \R$, 
$\{\omega: X_n(\omega) \leq k\} \in \cF_n$.

An integrable \{$\cF_n$\}-adapted sequence {$X_n$} is called a \emph{martingale} if
$$ \Exp(X_{n+1} \mid \cF_n ) = X_n \as \text{ for all } n \geq 0. $$
It is called a \emph{submartingale} if
$$ \Exp(X_{n+1} \mid \cF_n ) \geq X_n \as \text{ for all } n \geq 0, $$
and a \emph{supermartingale} if
$$ \Exp(X_{n+1} \mid \cF_n ) \leq X_n \as \text{ for all } n \geq 0. $$
An integrable, \{$\cF_n$\}-adapted sequence \{$D_n$\} is called a \emph{martingale difference sequence} if
$$ \Exp (D_{n+1} \mid \cF_n)=0 \text{ for all } n \geq 0. $$
Then, the sequence of $M_n := \sum_{k=1}^n D_k$ is $\{\cF_n\}$-martingale since 
$$\Exp [M_{n+1} - M_n \mid \cF_n] = \Exp [D_{n+1} \mid \cF_n] = 0,$$
which indicate 
$$\Exp [M_{n+1} \mid \cF_n] = M_n.$$

\begin{lemma}[Orthogonality of martingale differences \cite{gut} p.488] \label{martingale diff. orth.}
Let $\{D_n\}_{n=0}^\infty$ be a martingale difference sequence. Then $\Exp[D_m D_n]=0$ for $m \neq n$. Hence,
$$\Var\left(\sum_{i=0}^n D_i\right) = \sum_{i=0}^n \Var(D_i) .$$
\end{lemma}

We use a standard martingale difference construction based on resampling.
Consider the functional on $\R^n$, $f: \R^n \to \R$. Let $Y_1, Y_2, \dots, Y_n$ be iid. random variables and $W_n = f(Y_1, \dots, Y_n)$. 
Let $Y'_1, Y'_2, \dots, Y'_n$ be independent copies of $Y_1, Y_2, \dots, Y_n$ and 
$$W_n^{(i)} = f(Y_1, \dots, Y_{i-1}, Y'_i, Y_{i+1}, \dots, Y_n) .$$
Let $D_{n,i} = \Exp [W_n - W_n^{(i)} \mid \cF_i]$ where $\cF_i = \sigma(Y_1, \dots, Y_i)$.

\begin{lemma} \label{resampling}
Let $n \in \N$. Then
\begin{enumerate}[(i)]
\item $W_n - \Exp W_n = \sum_{i=1}^n D_{n,i}$;
\item $\Var (W_n) = \sum_{i=1}^n \Exp [D_{n,i}^2]$ whenever the latter sum is finite.
\end{enumerate}
\end{lemma}

\begin{proof}
The idea is well known. 
Since $W_n^{(i)}$ is independent of $Y_i$, 
$$\Exp [W_n^{(i)} \mid \cF_i] = \Exp[W_n^{(i)} \mid \cF_{i-1}] = \Exp [W_n \mid \cF_{i-1}] .$$
So,
$$D_{n,i} = \Exp[W_n \mid \cF_i] - \Exp[W_n \mid \cF_{i-1}] .$$
Hence $D_{n,i}$ is martingale differences, since
$$\Exp[D_{n,i} \mid \cF_{i-1}] = \Exp[W_n \mid \cF_{i-1}] - \Exp [W_n \mid \cF_{i-1}] = 0 $$
and
$$\sum_{i=1}^n D_{n,i} = \Exp[W_n \mid \cF_n] - \Exp[W_n \mid \cF_0] = W_n - \Exp W_n .$$
So,
$$\Exp\left[\left(\sum_{i=1}^n D_{n,i} \right)^2\right] = \Var (W_n) .$$
But by orthogonality of martingale differences, (Lemma \ref{martingale diff. orth.}),
$$\Var(W_n) = \sum_{i=1}^n \Exp [D_{n,i}^2] .$$
\end{proof}

Note that by the conditional Jensen's inequality $\left(\Exp([\,\xi \mid \cF])\right)^2 \leq \Exp [\,\xi^2 \mid \cF]$, we have
$$D_{n,i}^2 \leq \Exp\left[\left(W_n - W_n^{(i)}\right)^2 \mid \cF_i\right] .$$
So from part (ii) of Lemma \ref{resampling},
$$\Var(W_n) \leq \sum_{i=1}^n \Exp\left[\left( W_n^{(i)}-W_n \right)^2\right] .$$
This gives a upper bound for the variance of $W_n$, which is a factor of $2$ larger than the upper bound obtained from the Efron--Stein inequality (equation (2.3) in \cite{ss}):

\begin{lemma} \label{lem:efron-stein}
$$\Var(W_n) \leq \frac{1}{2}\sum_{i=1}^n \Exp\left[\left( W_n^{(i)}-W_n \right)^2\right] .$$
\end{lemma}

\section{Reflection principle for Brownian motion}

\begin{lemma}[Reflection principle \cite{morters} p.44] \label{reflection principle}
If $T$ is a stopping time and $\{w(t):t \geq 0 \}$ is a standard 1-dimensional Brownian motion, then the process $\{w^*(t):t\geq 0 \}$ called Brownian motion reflected at $T$ and defined by
$$w^*(t) = w(t)\1\{t\leq T\} + \left(2 w(T)- w(t)\right)\1\{t>T\} $$
is also a standard Brownian motion.
\end{lemma}

\begin{corollary} \label{reflection}
Suppose $r>0$ and $\{w(t):t \geq 0 \}$ is a standard 1-dimensional Brownian motion. Then, 
$$\Pr\left( \sup_{0 \leq s \leq t} w(s)>r \right)= 2\Pr\left(w(t)>r \right) .$$
\end{corollary}

\section{Useful inequalities}

We collect some useful inequalities which is useful in the next chapters.

\begin{lemma}[Markov's inequality \cite{gut} p.120]
Let $X$ be a random variable. Suppose that $\Exp|X|^r < \infty$ for some $r>0$, and let $x > 0$. Then,
$$ \Pr(|X| > x) \leq \frac{\Exp|X|^r}{x^r} . $$
\end{lemma}

\begin{lemma}[Chebyshev's inequality \cite{gut} p.121]
Let $X$ be a random variable. Suppose that $\Var X < \infty$. Then for $x> 0$,
$$ \Pr(|X - \Exp X| > x) \leq \frac{\Var X}{x^2} . $$
\end{lemma}

\begin{lemma}[The Cauchy--Schwarz inequality \cite{gut} p.130] \label{Cauchy-Schwarz ineq.}
Suppose that random variables $X$ and $Y$ have finite variances. Then,
$$ |\Exp XY| \leq \Exp |XY| \leq \|X\|_2  \|Y\|_2 = \sqrt{\Exp (X^2) \Exp (Y^2)} .$$
\end{lemma}

The next result generalises the Cauchy--Schwarz inequality.

\begin{lemma}[The H\"older inequality \cite{gut} p.129]
Let $X$ and $Y$ be random variables. Suppose that $p^{-1} + q^{-1} = 1$, $\Exp |X|^p < \infty$ and $\Exp |Y|^q < \infty$, then
$$ |\Exp XY | \leq \Exp |XY | \leq \|X\|_p \|Y\|_q = (\Exp X^p)^{1/p} (\Exp Y^q)^{1/q}.$$
\end{lemma}

\begin{lemma}[The Minkowski inequality \cite{gut} p.129] \label{minkowski}
Let $p \geq 1$. Suppose that $X$ and $Y$ are random variables, such that $\Exp |X|^p < \infty$ and $\Exp |Y|^p < \infty$. Then,
$$ \|X + Y\|_p \leq \|X\|_p + \|Y\|_p .$$
\end{lemma}
This is the triangle inequality for the $L^p$ norm.

Now we introduce some inequalities on martingales.

\begin{lemma}[Doob's inequality \cite{durrett} p.214] \label{Doob}
If $X_n$ is a martingale, then for $1 < p < \infty$,
$$ \Exp \left[\left( \max_{0 \leq m \leq n} |X_m| \right)^p \right] \leq \left(\frac{p}{p-1}\right)^p \Exp(|X_n|^p) .$$
\end{lemma}

\begin{lemma}[Azuma--Hoeffding inequality \cite{penrose} p.33] \label{azuma-hoeffding}
Let $D_{n,i}$ $(i=1,\dots,n)$ be a martingale difference sequence adapted to a filtration $\FF_i$, which means $D_{n,i}$ is $\FF_i$-measurable and $\Exp[D_{n,i}|\FF_{i-1}]=0$. Then, for any $t>0$,
$$\Pr\left(\Big| \sum_{i=1}^n D_{n,i} \Big| >t\right) \leq 2 \exp\left( -\frac{t^2}{2n d_{\infty}^2} \right),$$
where $d_{\infty}$ is such that $|D_{n,i}| \leq d_{\infty}$ a.s. for all $n,i$.
\end{lemma}

We also introduce some inequalities for sums of independent random variables.

\begin{lemma}[Marcinkiewicz--Zygmund inequality \cite{gut} p.151] \label{marcinkiewicz}
Let $p \geq 1$. Suppose that $X, X_1, X_2, \dots, X_n$ are independent, identically distributed random variables with mean 0 and $\Exp |X|^p < \infty$. 
Set $S_n = \sum_{k=1}^n X_k$. Then there exists a constant $B_p$ depending only on $p$, such that
$$ \Exp |S_n|^p \leq \begin{cases} B_p n \Exp|X|, & \text{if  } 1 \leq p \leq 2 , \\
 B_p n^{p/2} \Exp |X|^{p/2}, & \text{if  } p >2 . 
\end{cases} $$
\end{lemma}

\begin{lemma}[Rosenthal's inequality \cite{gut} p.151] \label{rosenthal}
Let $p \geq 1$. Suppose that $X_1, X_2, \dots, X_n$ are independent random variables such that $E|X_k|^p < \infty$ for all $k$. Set $S_n = \sum_{k=1}^n X_k$. Then,
$$ \Exp |S_n|^p \leq \max\left\{2^p \sum_{k=1}^n \Exp |X_k|^p , 2^{p^2} \left( \sum_{k=1}^n \Exp |X_k| \right)^p \right\}. $$
\end{lemma}

\section{Useful theorems and lemmas}

\begin{lemma}[Fubini's theorem \cite{gut} p.65] \label{fubini}
Let ($\Omega_1, \cF_1, P_1$) and ($\Omega_2, \cF_2, P_2$) be probability spaces, and
consider the product space ($\Omega_1 \times \Omega_2, \cF_1 \times \cF_2, P$), where $P = P_1 \times P_2$ is
the product measure. Suppose that $\bX = (X_1, X_2)$ is a two-dimensional random variable, and that $g$ is $\cF_1 \times \cF_2$-measurable, and (i) non-negative or (ii) integrable. Then,
$$ \Exp g(\bX) = \int_{\Omega} g(\bX)\, \ud P = \int_{\Omega_1} \left( \int_{\Omega_2}g(\bX)\, \ud P_2\right) \ud P_1 = \int_{\Omega_2} \left( \int_{\Omega_1}g(\bX)\, \ud P_1\right) \ud P_2 .$$
\end{lemma}

\begin{lemma}
\label{convergence of Cesaro mean}
Let $\{y_n \}_{n=1}^{\infty}$ be a sequence of real numbers and let $y \in \R$. If $ y_n \to y$ as $n \to \infty$, then $n^{-1}\sum_{i=1}^n y_i \to y$ as $n \to \infty$.
\end{lemma}
\begin{proof}
By assumption, for any $\eps >0$ there exists $n_0 \in \N$ such that $|y_n -y| \leq \eps$ for all $n \geq n_0$. Then,
\begin{align*}
\left| \frac{1}{n} \sum_{i=1}^n y_i -y \right| & = \left| \frac{1}{n} \sum_{i=1}^n (y_i-y) \right| \\
& \leq \left| \frac{1}{n} \sum_{i=1}^{n_0}(y_i-y) \right| + \left| \frac{1}{n} \sum_{i=n_0+1}^n (y_i-y) \right| \\
& \leq \frac{1}{n} \sum_{i=1}^{n_0} |y_i-y| + \frac{1}{n} \sum_{i=n_0 +1}^{n} |y_i-y| \\
& \leq \frac{1}{n} \sum_{i=1}^{n_0} |y_i-y| + \eps \\
& \leq 2\eps ,\end{align*}
for all $n$ big enough. Since $\eps >0$ was arbitrary, the result follows.
\end{proof}

\section{Multivariate normal distribution}

Let $\Sigma$ be a symmetric positive semi-definite ($d \times d$) matrix. Then, there exists an unique positive semi-definite symmetric matrix $\Sigma^{1/2}$ such that $\Sigma = \Sigma^{1/2} \Sigma^{1/2}$ \cite{mardia}. The matrix
$\Sigma^{1/2}$ can also be regarded as a linear transform of $\R^d$ given by ${\bf x} \mapsto \Sigma^{1/2} {\bf x}$.

For a random variable $Y$, the notation $Y \sim \cN(0, \Sigma)$ means $Y$ has $d$ dimensional normal distribution with mean $0$ and covariance matrix $\Sigma$.
In the degenerate case, all entries of the covariance matrix is $0$, $\Sigma = 0$, which means that $Y=0$ almost surely.

\begin{lemma} \label{linear transformation}
Suppose $X \sim \cN(0, I)$ and let $Y = \Sigma^{1/2} X$. Then $Y \sim \cN(0, \Sigma)$.
\end{lemma}

\begin{lemma}[Multidimensional Central Limit Theorem \cite{mardia} p.62] \label{Mult CLT}
Suppose $\{Z_i\}_{i=1}^\infty$ is a sequence of i.i.d. random variables on $\R^d$. $S_n = \sum_{i=1}^n Z_i$ is a random walk on $\R^d$. 
If $\Exp(\|Z_1\|^2) < \infty$, $\Exp Z_1 = 0$ and $\Exp(Z_1 Z_1^\top)= \Sigma$, then
$$n^{-1/2} S_n \tod \cN(0,\Sigma) .$$
\end{lemma}

\section{Analytic and Geometric prerequisites}

We recall a few basic facts from real analysis: \cite{rudin} is an excellent general reference.
The \emph{Heine--Borel theorem} states that a set in $\R^d$ is compact if and only if it is closed
and bounded \cite[p.~40]{rudin}. Compactness is preserved under continuous mappings: if $(X,\rho_X)$ is a compact metric space and $(Y, \rho_Y)$ is a metric space,
and $f : (X, \rho_X) \to (Y, \rho_Y)$ is continuous, then the image $f(X)$ is compact \cite[p.~89]{rudin}; moreover $f$ is uniformly continuous on $X$ \cite[p.~91]{rudin}.
For any such uniformly continuous $f$, there is a monotonic \emph{modulus of continuity} $\mu_f : \RP \to \RP$ such that $\rho_Y ( f(x_1), f(x_2) ) \leq \mu_f ( \rho_X (x_1, x_2 ))$
for all $x_1, x_2 \in X$, and for which $\mu_f ( \rho ) \downarrow 0$ as $\rho \downarrow 0$ (see e.g.~\cite[p.~57]{kallenberg}).

Let $d$ be a positive integer.
For $T > 0$, let $\cC ( [0,T] ; \R^d )$ \label{cC} denote the class of continuous functions
from $[0,T]$ to $\R^d$. Endow $\cC ( [0,T] ; \R^d )$ with the supremum metric
\[ \rho_\infty ( f, g) := \sup_{t \in [0,T]} \rho  ( f(t), g(t) ) , ~\text{for } f,g \in \cC ( [0,T] ; \R^d ). \] \label{rho_infty}
Let $\cC^0 ( [0,T] ; \R^d )$ \label{cC^0} denote those functions in $\cC ( [0,T] ; \R^d )$  that map $0$ to the origin in $\R^d$.

Usually, we work with $T=1$, in which case we write simply
\[ \cC_d := \cC ( [0,1] ; \R^d ) , ~~\text{and}~~ \cC_d^0 := \{ f \in \cC_d : f(0) = \0 \} .\] \label{cC_d} 
\label{cC_d^0}

For $f \in \cC ( [0,T] ; \R^d )$ and $t \in [0,T]$, define $f [0,t] := \{ f(s) : s \in [0,t] \}$, the image of $[0,t]$ under $f$.  Note that, since $[0,t]$ is compact and $f$ is continuous,
the \emph{interval image} $f [0,t]$ is compact.
We view elements $f \in \cC ( [0,T] ; \R^d )$ as \emph{paths} indexed by time $[0,T]$, so that $f[0,t]$ is the section of the path up to time $t$.

We need some notation and concepts from convex geometry: we found \cite{gruber} to be very useful,
supplemented by \cite{sw} as a convenient reference for a little integral geometry.
Let $d$ be a positive integer.
Let $\rho(\bx,\by) = \| \bx - \by\|$ denote the Euclidean distance between $\bx$ and $\by$ in $\R^d$. For a set $A \subseteq \R^d$, write 
 $\partial A$ for the boundary of $A$ (the intersection
of the closure of $A$ with the closure of $\R^d \setminus A$), and $\Int (A) := A \setminus \partial A$ for the interior
of $A$. 
For  $A \subseteq \R^d$
and a point $\bx \in \R^d$, set $\rho(\bx,A) := \inf_{\by \in A} \rho(\bx,\by)$, \label{rho(x,A)}
with the usual convention that $\inf \emptyset = +\infty$.
We write $\lambda_d$ for Lebesgue measure on $\R^d$.
 Write $\SS_{d-1} := \{ \bu \in \R^d : \| \bu \| = 1 \}$ \label{cS_d-1}
for the unit sphere in $\R^d$.

Let $\cK_d$ \label{cK_d} denote the collection of convex compact sets in $\R^d$, and write
\[ \cK^0_d := \{ A \in \cK_d : \0 \in A \} \] \label{cK_d^0}
for those sets in $\cK_d$ that include the origin. The Hausdorff metric on $\cK^0_d$
will be denoted
\[ \rho_H ( A, B ) := \max \Big\{ \sup_{\bx \in B} \rho(\bx,A) , \sup_{\by \in A} \rho(\by,B) \Big\} ~\text{for } A,B \in \cK_d.\] \label{rho_H}
Given $A \in \cK_d$, for $r > 0$ set
\[ \pi_r ( A) := \{ \bx \in \R^d : \rho (\bx,A) \leq r \} ,\] \label{pi_r}
the \emph{parallel body} of $A$ at distance $r$.
Note that, 
two equivalent descriptions of $\rho_H$ (see e.g.\ Proposition 6.3 of \cite{gruber}) are
for $A, B \in \cK^0_d$,
\begin{align}
\label{eq:hausdorff_minkowski} \rho_H (A, B) & = \inf \left\{ r \geq 0 : A \subseteq \pi_r ( B ) \text{ and }  B \subseteq \pi_r ( A ) \right\};  \text{ and } \\
\label{eq:hausdorff_support} \rho_H (A,B) & = \sup_{e \in \Sp_{d-1} } \left| h_A (e) - h_B (e) \right| ,
\end{align}
where $h_A ( \bx) := \sup_{\by \in A} (\bx \cdot \by )$ is the \emph{support function} of $A$ and $\bx \cdot \by$ is the inner product of $\bx$ and $\by$, i.e. $(x_1,y_1)\cdot (x_2, y_2) = x_1 x_2 + y_1 y_2$. \label{h_A()}

\section{Continuous mapping theorem and Donsker's Theorem}
\label{sec:CMT and Donsker}

We consider random walks in $\R^d$ in this section. First we need to define the weak convergence in $\R^d$.

Suppose $(\Omega, \cF, \Pr)$ is a probability space and $(M, \rho)$ is a metric space. For $n \geq 1$, suppose that
$$X_n, X: \Omega \longrightarrow M$$
are random variables taking values in $M$. If 
$$\Exp f(X_n) \to \Exp f(X)\ \As n\to \infty,$$
for all bounded, continuous functional $f: M \longrightarrow \R$, then we say that $X_n$ \emph{converges weakly} to $X$ and write $X_n \Rightarrow X$.
The weak convergence generalises the concept of convergence in distribution for random variables on $\R^d$. 

\begin{lemma}[continuous mapping theorem \cite{kallenberg} p.41] \label{continuous mapping}
Fix two metric spaces $(M_1,\rho_1)$ and $(M_2,\rho_2)$. Let $X, X_1, X_2, \dots$ be random variables taking values in $M_1$ with $X_n \Rightarrow X$. Suppose $f$ is a
mapping on $(M_1,\rho_1) \to (M_2,\rho_2)$, which is continuous everywhere in $M_1$ apart from possible on a set $A \subseteq M_1$ with $\Pr(X \in A) = 0$. 
Then, $f(X_n) \Rightarrow f(X)$.
\end{lemma}

We generalise the definition of $Z_i$ and $S_n$ a little in this section. Let $\{Z_i\}_{i=1}^\infty$ be a i.i.d. random vectors on $\R^d$ and $S_n = \sum_{i=1}^n Z_i$.
For each $n \in \N$ and all $t \in [0,1]$, define
\[ X_n (t) :=    S_{\lfloor nt \rfloor} + (nt - \lfloor nt \rfloor ) \left( S_{\lfloor nt \rfloor +1} - S_{\lfloor nt \rfloor} \right)  = S_{\lfloor nt \rfloor} + (nt - \lfloor nt \rfloor ) Z_{\lfloor nt \rfloor +1} .\]
Let $b:=( b(s) )_{s \in [0,1]}$ \label{b} denote standard Brownian motion in $\R^d$, started at $b(0) = 0$.

\begin{lemma}[Donsker's Theorem]
\label{thm:donsker} Let $d \in \N$.
Suppose that $\Exp (\| Z_1 \|^2) <\infty$,
$\| \Exp Z_1 \| = 0$, and $\Exp [ Z_1  Z_1^\top ] = \Sigma$ . Then, as $n \to \infty$,
\[ n^{-1/2} X_n \Rightarrow \Sigma^{1/2}b, \]
in the sense of weak convergence on $(\cC_d^0 , \rho_\infty )$.
\end{lemma}

\begin{remark}
Donsker's theorem generalizes the multidimensional central limit theorem (Lemma \ref{Mult CLT})
to a  \emph{functional} central limit theorem,
because weak convergence of paths implies convergence in distribution of the endpoints.
Indeed, taking $t=1$ in Donsker's Theorem, the marginal convergence gives
$$n^{-1/2} X_n(1) = n^{-1/2} S_n \tod \Sigma^{1/2} b(1) .$$
Here by Lemma \ref{linear transformation}, $\Sigma^{1/2} b(1) \sim \cN(0, \Sigma)$ since $b(1) \sim \cN(0,I)$. 
Then we have $n^{-1/2} S_n \tod \cN(0, \Sigma)$, which is Lemma \ref{Mult CLT}. 
\end{remark}

\section{Cauchy formula}

For this section we take $d=2$.
We consider the $\cA: \cK_2 \to \RP$ \label{cA}
and $\cL : \cK_2 \to \RP$ \label{cL} given by the area and the perimeter length of convex compact sets in the plane. Formally,
 we may define 
 \begin{equation}
 \label{eq:L-def}
 \cA (A) := \lambda_2 (A) , ~~\text{and} ~~ \cL (A) := \lim_{r \downarrow 0} \left( \frac{\lambda_2 ( \pi_r (A)) - \lambda_2 (A)}{r} \right),
 \text{ for } A \in \cK_2 .\end{equation}
 The limit in \eqref{eq:L-def} exists by the \emph{Steiner formula} of integral geometry (see e.g.~\cite{schneider-weil}),
 which expresses $\lambda_2 ( \pi_r(A))$ as a quadratic polynomial in $r$ whose coefficients
 are given in terms of the \emph{intrinsic volumes} of $A$:
 \begin{equation}
 \label{eq:steiner}
\lambda_2 ( \pi_r(A)) = \lambda_2 (A) + r \cL (A) + \pi r^2 \1 \{ A \neq \emptyset \} .\end{equation}
 %
 In particular,
 \[ \cL (A) = \begin{cases}\phantom{2} \cH_{1} ( \partial A ) & \text{if } \Int (A) \neq \emptyset , \\
2 \cH_{1} ( \partial A ) & \text{if } \Int (A)  = \emptyset , 
\end{cases} \]
 where $\cH_{d}$ is $d$-dimensional Hausdorff measure on Borel sets. 
 We observe the  translation-invariance and scaling properties
 \[ \cL ( x + \alpha A ) = \alpha \cL (A) , ~~\text{and}~~ \cA ( x+ \alpha A) = \alpha^2 \cA (A) ,\]
 where for $A \in \cK_2$, $x + \alpha A = \{x + \alpha y : y \in A \} \in \cK_2$.

For $A \in \cK_2$, Cauchy obtained the following formula:
\begin{equation}
\label{cauchy0}
 \cL (A) = \int_0^\pi \left( \sup_{\by \in A} ( \by \cdot \be_\theta ) - \inf_{\by \in A} ( \by \cdot \be_\theta ) \right) \ud \theta .
\end{equation}
We will need the following consequence of (\ref{cauchy0}).

\begin{proposition} \label{cauchyhull}
Let $K = \{ \bz_0, \ldots, \bz_n \}$ be a finite point set in $\R^2$, and let $\CC = \hull (K)$.
Then
\begin{equation}
\label{cauchy1}
\cL ( \CC ) = \int_0^\pi \left( \max_{0 \leq i \leq n} ( \bz_i \cdot \be_\theta ) - \min_{0 \leq i \leq n } ( \bz_i \cdot \be_\theta ) \right) \ud \theta.
\end{equation}
\end{proposition}

In particular, for the case of our random walk, (\ref{cauchy1}) says
\begin{equation}
\label{cauchy_} 
L_n = \cL ( \hull(S_0, \dots, S_n) ) = \int_0^\pi  \left( \max_{0 \leq i \leq n} ( S_i \cdot \be_\theta ) - \min_{0 \leq i \leq n } ( S_i \cdot \be_\theta ) \right) \ud \theta.
\end{equation}
An immediate but useful consequence of (\ref{cauchy_}) is that
\begin{equation}
\label{L_monotone}
L_{n+1} \geq L_n, \as
\end{equation}

In the case where $K$ is a finite point set, $\hull ( K)$
is a convex polygon, the boundary of which contains vertices $\V \subseteq K$
(extreme points of the convex hull) and the line-segment edges connecting them;
note that $\hull (K) = \hull (\V)$.

Now, by convexity,
\[ \sup_{\by \in \CC} ( \by \cdot \be_\theta ) = \max_{0 \leq i \leq n} ( \bz_i \cdot \be_\theta) = \sup_{\by \in \V} ( \by \cdot \be_\theta ) ,\]
and similarly for the infimum. So (\ref{cauchy0}) does indeed imply (\ref{cauchy1}). However, to keep this presentation as self-contained as possible, we give a direct proof of (\ref{cauchy1}) without appealing to the more general result (\ref{cauchy0}).

\begin{proof}[Proof of Proposition \ref{cauchyhull}]
The above discussion shows that it suffices to consider the case where $\V =K$ in which all of the $\bz_i$ are on the boundary of the convex hull. Without loss of generality, suppose that $\0 \in \CC$. Then we may rewrite (\ref{cauchy1}) as
$$\cL(\CC)= \int_0^{2\pi} \max_{0 \leq i \leq n} (\bz_i \cdot \be_{\theta}) \,\ud \theta .$$
Suppose also that $\bz_i = \|\bz_i\|\be_{\theta_i}$ in polar coordinates, labelled so that $0 \leq \theta_0 < \theta_1 < \dots < \theta_n < 2\pi$. Thus starting from the rightmost point of $\partial \CC$ on the horizontal axis and traversing the boundary anticlockwise, one visits the vertices $\bz_0,\bz_1,\dots,\bz_n$ in order.

\begin{figure}[h]
  \centering
	\includegraphics[width=0.85\textwidth]{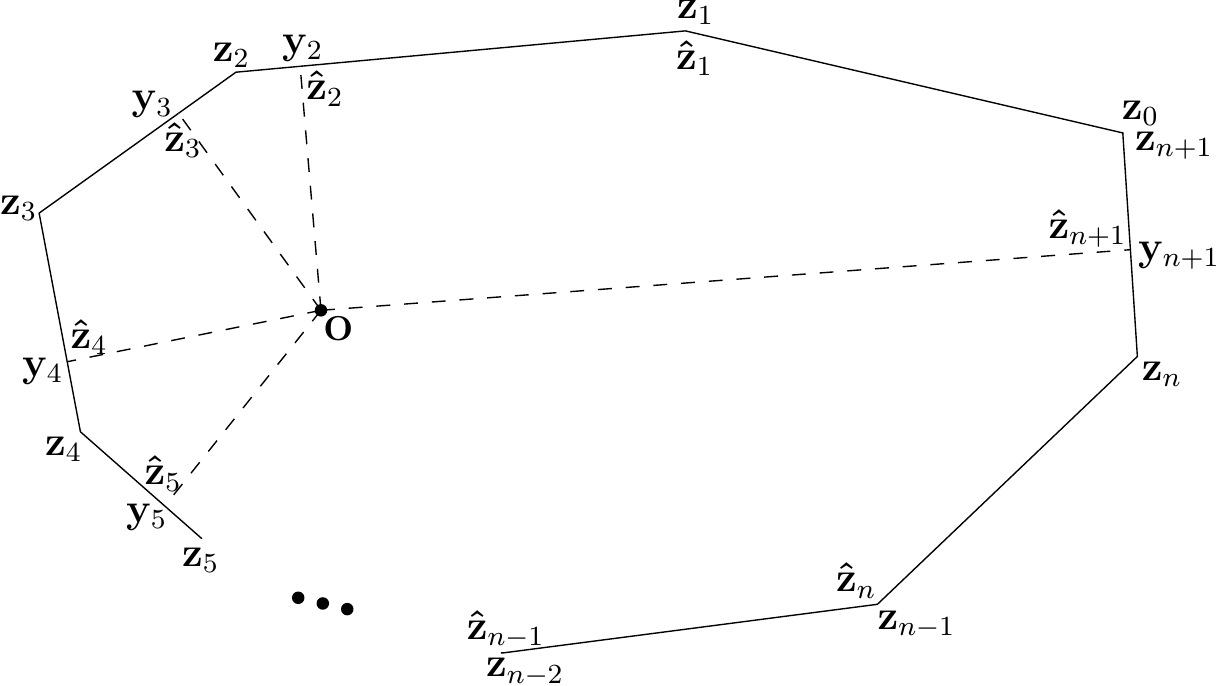}
  \caption{Proof of Proposition \ref{cauchyhull}}
  \label{ipe 1}
\end{figure}

Let $\bz_{n+1}:=\bz_0$. Draw the perpendicular line of $\bz_k - \bz_{k-1}$ passing through point $\0$ and denote the foot as $\by_k$. For $1 \leq k \leq n+1$, let 
\begin{displaymath}
   \hat{\bz}_k := \left\{
     \begin{array}{ll}
       \by_k, & \mbox{if}\ \by_k \in\ \mbox{line segment}\ \overline{\bz_{k-1}\bz_k} \\
       \bz_k, & \mbox{if}\ \by_k \in\ \mbox{extended line of}\ \overrightarrow{\bz_{k-1}\bz_k} \\
       \bz_{k-1}, & \mbox{if}\ \by_k \in\ \mbox{extended line of}\ \overrightarrow{\bz_{k}\bz_{k-1}}
     \end{array}
   \right.
\end{displaymath}
and let $\hat{\bz}_0 := \hat{\bz}_{n+1}$.
Notice that $\hat{\bz}_1, \dots, \hat{\bz}_{n+1}$ are ordered in the same way as $\bz_0,\dots,\bz_n$ (see Figure \ref{ipe 1}). Therefore, 
$$\partial \CC = \bigcup_{k=0}^{n}\left[(\hat{\bz}_{k+1}-\bz_k) \cup (\bz_k - \hat{\bz}_k )\right] .$$

Write $\hat{\bz}_{i} = \|\hat{\bz}_i\|\be_{\hat{\theta}_i}$ for $0 \leq i \leq n+1$ in the polar coordinates, we have 
$$ \int_0^{2\pi} \max_{0 \leq i \leq n} (\bz_i \cdot \be_{\theta}) \,\ud \theta = \sum_{k=0}^{n} \int_{\hat{\theta}_k}^{\hat{\theta}_{k+1}} \bz_k \cdot \be_{\theta} \,\ud \theta .$$
Consider $\int_{\hat{\theta}_k}^{\hat{\theta}_{k+1}} \bz_k \cdot \be_{\theta} \,\ud \theta$. 
Let $\bz_k := (\alpha_1, \beta_1)$, $\bz_{k+1} := (\alpha_2,\beta_2)$ and $\bz_{k-1} := (\alpha_0,\beta_0)$. Without loss of generality, we can set $\beta_1 = 0$ and $\alpha_1 >0$. Then we have $\beta_2 \geq 0$, $\beta_0 \leq 0$, $0 \leq \hat{\theta}_{k+1} \leq \pi/2$ and $-\pi/2 \leq \hat{\theta}_{k} \leq 0$. 
So,
\begin{align*}
\int_{\hat{\theta}_k}^{\hat{\theta}_{k+1}} \bz_k \cdot \be_{\theta} \,\ud \theta  
= & \int_{\hat{\theta}_k}^{\hat{\theta}_{k+1}}(\alpha_1,0)\cdot(\cos \theta,\sin \theta) \,\ud\theta \\
= & \alpha_1 (\sin \hat{\theta}_{k+1} - \sin \hat{\theta}_k) \\
= & \alpha_1 \left(\frac{\| \hat{\bz}_{k+1}-\bz_k \|}{\alpha_1} - \frac{-\| \bz_k - \hat{\bz}_k \|}{\alpha_1} \right) \\
= & \| \hat{\bz}_{k+1}-\bz_k \| + \| \bz_k - \hat{\bz}_k \| .
\end{align*}
Hence,
$$ \int_0^{2\pi} \max_{0 \leq i \leq n} (\bz_i \cdot \be_{\theta}) \,\ud \theta = \sum_{k=0}^{n} \int_{\hat{\theta}_k}^{\hat{\theta}_{k+1}} \bz_k \cdot \be_{\theta} \,\ud \theta 
= \sum_{k=0}^{n} \left( \| \hat{\bz}_{k+1}-\bz_k \| + \| \bz_k - \hat{\bz}_k \| \right) = L(\CC) . \qedhere$$
\end{proof}

%% file: chapter3.tex
\pagestyle{myheadings} \markright{\sc Chapter 3}

\chapter{Scaling limits for convex hulls}
\label{chapter3}

\section{Overview}
\label{sec:outline}

For some of the results that follow, scaling limit ideas are useful.
Recall that $S_n = \sum_{k=1}^n Z_k$ is the location of our random walk in $\R^2$ after $n$ steps. Write $\cS_n := \{ S_0, S_1, \ldots, S_n \}$. \label{cS_n}
Our strategy to study properties of the random convex set $\hull \cS_n$ (such as $L_n$ or $A_n$)
is to seek a weak limit for 
a suitable scaling of $\hull \cS_n$, which
we must hope to be 
the convex hull of some scaling limit representing the walk $\cS_n$.

In the case of  zero drift ($\mu = 0$) a candidate scaling limit for the walk is readily identified
 in terms  
 of planar Brownian motion. For the case $\mu \neq 0$,  the `usual' approach of
centering and then scaling the walk (to again obtain planar Brownian motion) is not 
 useful in our context, as this transformation
does not act on the convex hull in any sensible way. A better idea is to scale space differently in the direction of $\mu$ and in the
orthogonal direction.

In other words, in either case we consider
$\phi_n (\cS_n)$ for some \emph{affine} continuous scaling function
$\phi_n : \R^2 \to \R^2$. 
The convex hull is preserved under affine transformations, so 
\[ \phi_n ( \hull \cS_n ) = \hull  \phi_n ( \cS_n )  ,\]
the convex hull of a random set 
which will have a weak limit. We will then be able to deduce
scaling limits for quantities $L_n$ and $A_n$ provided, first, that we work in suitable spaces
on which our functionals of interest 
enjoy continuity, so that we can appeal to the continuous mapping theorem for weak limits,
and, second, that $\phi_n$ acts on length and area by simple scaling. The usual $n^{-1/2}$ scaling
when $\mu =0$ is fine; for $\mu \neq 0$ we scale space in one coordinate by $n^{-1}$ and in the other by $n^{-1/2}$,
which acts nicely on area, but \emph{not}  length. Thus these methods work exactly in the three cases
corresponding to \eqref{eq:three_vars}.

In view of the scaling limits that we expect, it is natural to work not with point
sets like $\cS_n$, but with continuous \emph{paths}; instead of $\cS_n$
we consider the interpolating path constructed as follows.
For each $n \in \N$ and all $t \in [0,1]$, define
\[ X_n (t) :=    S_{\lfloor nt \rfloor} + (nt - \lfloor nt \rfloor ) \left( S_{\lfloor nt \rfloor +1} - S_{\lfloor nt \rfloor} \right)  = S_{\lfloor nt \rfloor} + (nt - \lfloor nt \rfloor ) Z_{\lfloor nt \rfloor +1} .\]
Note that $X_n (0) =  S_0$ and $X_n (1) =   S_n$. 
Given $n$, we are interested in the convex hull of the image in $\R^2$ of the interval
$[0,1]$ under 
the continuous function 
$X_n$. Our scaling limits will be of the same form.

\section{Convex hulls of paths}

In this section we study some basic properties of the map from a continuous path to its convex hull.
Let $f \in \cC ([0,T] , \R^d)$. For any $t \in [0,T]$,  $f[0,t]$ is compact, and so
Carath\'eodory's theorem for convex hulls (see Corollary 3.1 of \cite[p.\ 44]{gruber})
shows that $\hull ( f [0,t] )$ is compact. So $\hull ( f [0,t] ) \in \cK_d$ is convex, bounded, and closed; in particular, it is a Borel set.

For reasons that we shall see, it mostly suffices to work with paths parametrized over the interval $[0,1]$.
For $f \in \cC_d$, define
\[ H (f) := \hull \left( f [ 0,1 ] \right) .\] \label{H()}

First we prove continuity of the map $f \mapsto H (f)$.

\begin{lemma}
\label{lem:path-hull}
For any $f, g \in \cC^0_d$, we have
\begin{equation}
\label{eq:H-comparison}
  \rho_H ( H(f) , H(g) ) \leq \rho_\infty ( f,g).\end{equation}
Hence the function $H : ( \cC^0_d , \rho_\infty ) \to ( \cK^0_d , \rho_H )$ is continuous.
\end{lemma}
 \begin{proof}
 Let $f, g \in \cC^0_d$. Then $H(f)$ and $H(g)$ are non-empty, as they both contain $f(0) = g(0)=\0$.
Consider  $\bx \in H(f)$. Since the convex hull of a set is the set of all convex combinations of points of the set (see Lemma 3.1 of \cite[p.\ 42]{gruber}), 
there exist a finite positive integer $n$, weights $\lambda_1, \dots ,\lambda_n \geq 0$ with $\sum_{i=1}^n \lambda_i =1$, and $t_1, \dots, t_n \in [0,1]$ for which 
$\bx = \sum_{i=1}^n \lambda_i f(t_i)$.
 Then, taking $\by = \sum_{i=1}^n \lambda_i g(t_i)$, we have that $\by \in H(g)$ and, by the triangle inequality,
 \[ \rho (\bx,\by) = \sum_{i=1}^n \lambda_i \rho( f(t_i) , g(t_i) )
              \leq \rho_\infty (f,g) .\]
             Thus, writing $r = \rho_\infty (f,g)$,
              every $\bx \in H(f)$ has $\bx \in \pi_r ( H (g) )$,
             so $H(f) \subseteq \pi_r ( H (g) )$. The symmetric argument gives
             $H(g) \subseteq \pi_r ( H (f) )$. Thus, by \eqref{eq:hausdorff_minkowski},
              we obtain \eqref{eq:H-comparison}.
 \end{proof}

 Given $f \in \cC_d$, let $E (f) := \ext ( H(f) )$, the extreme points of the convex hull (see \cite[p.\ 75]{gruber}). 
 The set $E(f)$ is the smallest set (by inclusion) that generates $H(f)$ as its convex hull, i.e., for any 
 $A$ for which $\hull (A) = H(f)$, we have $E(f) \subseteq A$; see Theorem 5.5 of \cite[p.\ 75]{gruber}. 
 In particular, $E(f) \subseteq f [0,1]$.
 
 \begin{lemma}
 \label{lem:hull-max}
 Let $f \in \cC_d$.
Let $q : \R^d \to \R$ be continuous and convex. Then $q$ attains its supremum over $H(f)$ at a point of $f$, i.e.,
\[ \sup_{\bx \in H(f)} q (\bx) = \max_{t \in [0,1]} q (f(t)) .\]
 \end{lemma}
 \begin{proof}
 Theorem 5.6 of \cite[p.\ 76]{gruber} shows that any continuous convex function on $H(f)$ attains its maximum
 at a point of $E(f)$. Hence, since $E(f) \subseteq f [0,1]$,
 \[ \sup_{\bx \in H(f)} q (\bx) = \sup_{\bx \in E(f)} q (\bx) \leq \sup_{\bx \in f[0,1]} q(\bx) . \]
 On the other hand, $f[0,1] \subseteq H(f)$, so
 $\sup_{\bx \in f[0,1]} q(\bx) \leq \sup_{\bx \in H(f)} q (\bx)$. Hence  
 \[  \sup_{\bx \in H(f)} q (\bx) = \sup_{\bx \in f[0,1]} q(\bx) = \sup_{t \in [0,1]} q(f(t)) .\]
 Since $q \circ f$ is the composition of two continuous functions, it is itself continuous, and so the supremum is attained in the compact set $[0,1]$.
 \end{proof}
 
 For $A \in \cK^0_d$, the \emph{support function} of $A$ is $h_A : \R^d \to \RP$ defined by
 \[ h_A ( \bx) := \sup_{\by \in A} (\bx \cdot \by ) .\]
  For $A \in \cK_2^0$,
 \emph{Cauchy's formula} \eqref{cauchy0} states
\[ \cL (A) = \int_{\SS_1} h_A (\bu) \ud \bu = \int_0^{2\pi} h_A ( \be_\theta ) \ud \theta .\]

 We end this section by showing that the map $t \mapsto \hull ( f[0,t] )$ on $[0,T]$ is continuous if $f$ is continuous on $[0,T]$,
 so that the continuous trajectory $t \mapsto f(t)$ is accompanied by a continuous `trajectory' of its convex hulls. This observation was made
 by El Bachir \cite[pp.~16--17]{elbachir}; we take a different route based on the path space result Lemma \ref{lem:path-hull}. First we need a lemma.

 \begin{lemma}
 \label{lem:path-stretch}
 Let $T >0$ and $f \in \cC ( [0,T] ; \R^d)$. Then the map defined for $t \in [0,T]$ by 
$t \mapsto  g_t$, where $g_t : [0,1] \to \R^d$ is given by $g_t (s) = f (t s)$, $s \in [0,1]$, is
a continuous function from $([0,T], \rho)$ to $( \cC_d, \rho_\infty )$.
 \end{lemma}
 \begin{proof}
 First we fix $t \in [0,T]$ and show that $s \mapsto g_t (s)$ is continuous, so that $g_t \in \cC_d$ as claimed.
 Since $f$ is continuous on the compact interval $[0,T]$, it is uniformly continuous,
 and admits  
 a monotone modulus of continuity $\mu_f$. Hence
 \[ \rho ( g_t (s_1) , g_t (s_2) ) = \rho ( f(ts_1) , f(ts_2) ) \leq \mu_f ( \rho (ts_1 , ts_2)) = \mu_f ( t \rho (s_1, s_2 ) ) ,\]
 which tends to $0$ as $\rho(s_1, s_2) \to 0$. 
 Hence $g_t \in \cC_d$.

It remains to show that $t \mapsto g_t$ is continuous. But on $\cC_d$,
\begin{align*} \rho_\infty ( g_{t_1}, g_{t_2} ) & = \sup_{s \in [0,1]} \rho ( f(t_1 s) , f(t_2 s) ) \\
& \leq \sup_{s \in [0,1]} \mu_f ( \rho (t_1 s, t_2 s) ) \\
& \leq \mu_f ( \rho ( t_1, t_2 ) ) ,\end{align*}
which tends to $0$ as $\rho (t_1, t_2) \to 0$, again using the uniform continuity of $f$.
\end{proof}

Here is the path continuity result for convex hulls of continuous paths; cf \cite[p.~16--17]{elbachir}.

\begin{corollary}
\label{cor:point-hull}
Let $T >0$ and $f \in \cC^0 ( [0,T] ; \R^d)$ with $f(0) = \0$. Then the map defined for $t \in [0,T]$ by 
$t \mapsto  \hull ( f[0,t] )$ is
a continuous function from $([0,T], \rho)$ to $( \cK^0_d, \rho_H )$.
\end{corollary}
\begin{proof}
By Lemma \ref{lem:path-stretch}, $t \mapsto g_t$ is continuous, where $g_t(s) = f(ts)$, $s \in [0,1]$. Note that,
since $f(0)=\0$, $g_t \in \cC_d^0$.
But the sets $f [0,t]$ and $g_t [0,1]$ coincide, so $\hull ( f [0,t] )  = H (g_t)$, and, by Lemma \ref{lem:path-hull}, $g_t \mapsto H(g_t)$ is continuous.
Thus $t \mapsto H(g_t)$ is the composition of two continuous functions, hence itself a continuous function:
\[ \begin{array}{ccccc}
[0,T] & \longrightarrow & \cC^0_d & \longrightarrow & \cK^0_d \\
t & \mapsto & g_t & \mapsto & H(g_t) 
 \end{array} \qedhere \]
\end{proof}

Recall definitions of the functionals for perimeter length $\cL$ and area $\cA$ in \eqref{eq:L-def}. 
We give the following inequalities in the metric spaces.
\begin{lemma}
\label{lem:functional-continuity}
Suppose that $A, B \in \cK^0_2$. Then
 \begin{align}
\label{eq:L-comparison}
  \rho ( \cL(A) , \cL(B) ) & \leq 2 \pi \rho_H (A,B) ;\\
  \label{eq:A-comparison}
  \rho ( \cA(A) , \cA(B) ) & \leq \pi   \rho_H (A,B)^2 +  ( \cL(A) \vee \cL(B) ) \rho_H (A,B)  .
  \end{align}
Hence, the functions $\cL$ and $\cA$ are both continuous from $(\cK^0_2 , \rho_H )$ to $( \RP , \rho)$.
\end{lemma}
\begin{proof}
First consider $\cL$. By Cauchy's formula,
\begin{align*}
\left| \cL (A) - \cL(B) \right|   & = \left| \int_{\SS_1} \left( h_{A} ( \bu ) - h_{B} ( \bu ) \right) \ud \bu \right| \\
& \leq \int_{\SS_1} \sup_{\bu \in \SS_1} \left| h_{A} ( \bu ) - h_{B} ( \bu ) \right| \ud \bu   = 2 \pi \rho_H ( A , B ) ,\end{align*}
by the triangle inequality and then \eqref{eq:hausdorff_support}. This gives \eqref{eq:L-comparison}.

Now consider $\cA$. Set $r = \rho_H (A,B)$. Then, by \eqref{eq:hausdorff_minkowski}, $A \subseteq \pi_r (B)$.
Hence
\[ \cA (A) \leq \cA ( \pi_r (B) ) \leq \cA(B) + r \cL (B) + \pi r^2 ,\]
by \eqref{eq:steiner}. With the analogous argument starting from $B \subseteq \pi_r (A)$, we get \eqref{eq:A-comparison}.
\end{proof}

\section{Brownian convex hulls as scaling limits}
\label{sec:Brownian-hulls}

Now we return to considering the random walk $S_n = \sum_{k=1}^n Z_k$ in $\R^2$.
The two different scalings outlined in Section \ref{sec:outline}, for the cases $\mu =0$ and $\mu \neq 0$,
lead to different scaling limits for the random walk. Both are associated with Brownian motion.

In the case $\mu =0$, the scaling limit is the usual planar Brownian motion, at least when $\Sigma = I$, the identity matrix.
Let $b:=( b(s) )_{s \in [0,1]}$ denote standard Brownian motion in $\R^2$, started at $b(0) = 0$.  
For convenience we may assume $b \in \cC_2^0$ (we can work on a probability space for which continuity holds for all sample points, rather than merely almost all).
For $t \in [0,1]$, let 
\begin{equation} \label{h_t}
h_t := \hull   b[0,t]   \in \cK_2^0
\end{equation}
denote the convex hull of the Brownian path up to time $t$.
By Corollary \ref{cor:point-hull}, $t \mapsto h_t$ is continuous. Much is known about the properties of $h_t$: see e.g.\ 
\cite{chm,elbachir,evans,klm}.
We also set
\begin{equation} \label{eqn:def of Lt At for BM}
 \ell_t := \cL ( h_t ) , ~~\text{and}~~ a_t := \cA ( h_t ) ,
\end{equation} 
the perimeter length and area of the standard Brownian convex hull.
By Lemma \ref{lem:functional-continuity},
the processes $t \mapsto \ell_t$ and $t \mapsto a_t$ also have continuous sample paths.

We also need to work with the case of general covariances $\Sigma$;
to do so we introduce more notation and recall some facts about multivariate Gaussian random vectors.
For definiteness, we view vectors as Cartesian column vectors when required.
Since $\Sigma$ is positive semidefinite and symmetric, 
there is a (unique) positive semidefinite symmetric matrix square-root $\Sigma^{1/2}$ \label{Sigma^1/2}
for which $\Sigma = (\Sigma^{1/2} )^2$.
The   map $x \mapsto \Sigma^{1/2} x$ associated with $\Sigma^{1/2}$ is a linear transformation on $\R^2$
with Jacobian $\det \Sigma^{1/2} = \sqrt{ \det \Sigma}$; 
hence  $\leb ( \Sigma^{1/2} A ) =  \leb (A)  \sqrt{ \det \Sigma }$
for any measurable $A \subseteq \R^2$. 

If $W \sim \cN (0, I)$, then  by Lemma \ref{linear transformation}, $\Sigma^{1/2} W \sim \cN (0, \Sigma)$,
a bivariate normal distribution with mean $0$
and covariance $\Sigma$; the notation permits $\Sigma =0$,
in which case $\cN(0,0)$ stands for the degenerate
normal distribution with point mass at $0$. Similarly, given $b$ a standard Brownian motion on $\R^2$, the diffusion $\Sigma^{1/2} b$
is \emph{correlated} planar Brownian motion with covariance matrix $\Sigma$. 
Recall that `$\Rightarrow$' (see Section \ref{sec:CMT and Donsker}) indicates weak convergence.

\begin{theorem}
\label{thm:limit-zero} 
 Suppose that $\Exp ( \| Z_1\|^2 ) < \infty$ and  $\mu =0$.
Then, as $n \to \infty$,
   \[ n^{-1/2} \hull \{ S_0, S_1, \ldots, S_n \} \Rightarrow \Sigma^{1/2} h_1 , \]
   in the sense of weak convergence on $(\cK_2^0 , \rho_H )$.
  \end{theorem}
\begin{proof}
Donsker's theorem (see Lemma \ref{thm:donsker})
implies that $n^{-1/2} X_n \Rightarrow \Sigma^{1/2} b$
on $(\cC_2^0 , \rho_\infty )$. 
Now, the point set $X_n [0,1]$ is the union of the line segments
$\{  S_{k} +\theta (S_{k+1} - S_k)  : \theta \in [0,1] \}$
over $k=0,1,\ldots, n-1$. Since the convex hull is preserved under  affine transformations,   
\[ 
H ( n^{-1/2} X_n ) =
 n^{-1/2} H (X_n) = n^{-1/2} \hull \{ S_0, S_1, \ldots, S_n \}  .\]
By Lemma \ref{lem:path-hull}, $H$ is continuous, and so the continuous mapping theorem 
(see Lemma \ref{continuous mapping}) implies that 
$$n^{-1/2} \hull \{ S_0, S_1, \ldots, S_n \} \Rightarrow H ( \Sigma^{1/2} b ) \text{ on } (\cK_2^0 , \rho_H ) .$$
Finally, invariance of the convex hull under affine transformations shows $H (\Sigma^{1/2} b ) = \Sigma^{1/2} H (b) = \Sigma^{1/2} h_1$.
\end{proof}

Theorem \ref{thm:limit-zero} together with the continuous mapping theorem and Lemma \ref{lem:functional-continuity}
implies the following distributional limit results in the case $\mu =0$. Recall that `$\tod$' (see Section \ref{sec:convergence of random variables}) denotes
convergence in distribution for $\R$-valued random variables.

\begin{corollary}
\label{cor:zero-limits}
 Suppose that $\Exp ( \| Z_1\|^2 ) < \infty$ and 
$\mu =0$. 
Then, as $n \to \infty$,
   \[ n^{-1/2} L_n \tod \cL ( \Sigma^{1/2} h_1 ) , ~~\text{and} ~~ n^{-1 } A_n \tod  \cA ( \Sigma^{1/2} h_1  ) = a_1 \sqrt{\det \Sigma} .\]
\end{corollary}

\begin{remark}
Recall that $a_1 = \cA(h_1)$ is the area of the standard 2-dimensional Brownian convex hull run for unit time.
The distributional limits for $ n^{-1/2} L_n$ and $ n^{-1 } A_n$
in Corollary \ref{cor:zero-limits} are supported on $\RP$ and, as we will show in  Proposition \ref{prop:var_bounds u0} and Proposition \ref{prop:var_bounds v0 v+} below, are non-degenerate if $\Sigma$ is positive definite;
hence they are \emph{non-Gaussian} excluding trivial cases.
\end{remark}

In the case $\mu \neq 0$,  the scaling limit 
can be viewed as a space-time trajectory of one-dimensional Brownian motion. Let $w:=( w(s) )_{s \in [0,1]}$ \label{w} denote standard Brownian motion in $\R$, started at $w(0) = 0$;
similarly to above, we may take $w \in \cC_1^0$.
Define $\tilde b   \in \cC_2^0$ in Cartesian coordinates via
\[ \tilde b (s) =  (  s , w(s) ) , ~ \text{for } s \in [0,1]; \] \label{tilde b}
thus $\tilde b [0,1]$ is the space-time diagram of one-dimensional Brownian motion run for unit time.
For $t \in [0,1]$, let $\tilde h_t := \hull \tilde b [0,t] \in \cK_2^0$, \label{tilde h_t} 
and define $\tilde a_t := \cA ( \tilde h_t )$. \label{tilde a_t}
(Closely related to $\tilde h_t$ is the greatest \emph{convex minorant} of $w$ over $[0,t]$, which is
of interest in its own right, see e.g.~\cite{pitman-ross} and references therein.)
  
Suppose $\mu \neq 0$ and $\sperp \in (0,\infty)$. 
Given $\mu \in \R^2 \setminus \{ 0 \}$,
let $\hat \mu_\perp$ be the unit vector perpendicular   to  $\mu$ obtained by rotating $\hat \mu$ by $\pi/2$ anticlockwise.
For $n \in \N$, define $\psi^\mu_n : \R^2 \to \R^2$ by the image of $x \in \R^2$ in Cartesian components:
\[ \psi^\mu_n ( x ) = \left( \frac{x \cdot \hat \mu}{n \| \mu \|  } , \frac{x \cdot \hat \mu_\perp}{ \sqrt { n \sperp }  } \right) .\]
In words, $\psi^\mu_n$ rotates $\R^2$, mapping $\hat \mu$ to the unit vector in the horizontal direction,
and then scales space with a horizontal shrinking factor $\| \mu \| n$ and a vertical factor $ \sqrt { n \sperp } $; 
see Figure \ref{fig rotate} for an illustration.

\begin{figure}[t]
\center
\includegraphics[width=0.9\textwidth]{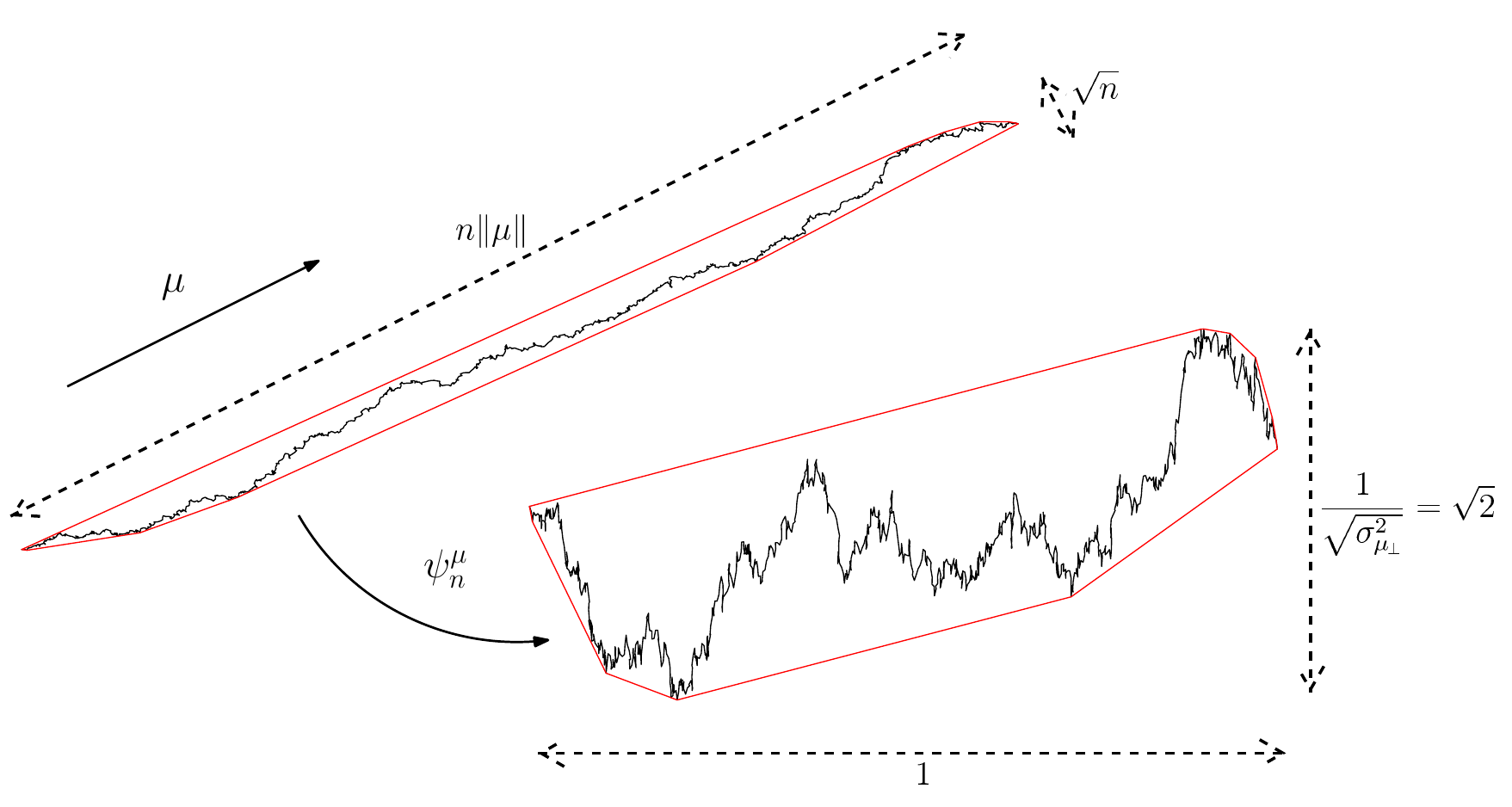}
\caption{Simulated path of $n=1000$ steps a random walk with drift $\mu = ( \frac{1}{2}, \frac{1}{4} )$ and its convex hull
(\emph{top left}) and (not to the same scale) the image under $\psi_n^\mu$ (\emph{bottom right}).}
\label{fig rotate}
\end{figure}

 \begin{theorem}
  \label{thm:limit-drift}  
  Suppose that $\Exp ( \| Z_1\|^2 ) < \infty$, $\mu \neq 0$, and $\sperp >0$.
 Then, as $n \to \infty$,
   \[ \psi^\mu_n (  \hull \{ S_0, S_1, \ldots, S_n \} ) \Rightarrow \tilde h_1, \]
   in the sense of weak convergence on $(\cK_2^0 , \rho_H )$.
  \end{theorem}
\begin{proof}
Observe that $\hat \mu \cdot S_n$ is a random walk on $\R$ with one-step mean drift $\hat \mu \cdot \mu = \| \mu \| \in (0,\infty)$,
while $\hat \mu_\perp \cdot S_n$ is a walk with mean drift $\hat \mu_\perp \cdot \mu = 0$
and increment variance
\begin{align*}
\Exp \left[ ( \hat \mu_\perp \cdot Z  )^2   \right]
& = \Exp \left[ ( \hat \mu_\perp \cdot ( Z - \mu)  )^2   \right] \\
& = \Exp [ \| Z - \mu \|^2 ] - \Exp [ (\hat \mu \cdot (Z - \mu) )^2 ] = \sigma^2 - \spara \\
& = \sperp .
\end{align*}
According to the strong law of large numbers, for any $\eps>0$ there exists $N_\eps \in \N$ a.s.\ such that
$| m^{-1} \hat \mu \cdot S_m - \| \mu \| | < \eps$ for $m \geq N_\eps$.
Now we have that
\begin{align*} \sup_{N_\eps/n \leq t \leq 1 } \left| \frac{ \hat \mu \cdot S_{\lfloor nt \rfloor}}{n} - t \| \mu \| \right|
& \leq \sup_{N_\eps/n \leq t \leq 1 } \left( \frac{\lfloor nt \rfloor}{n} \right)
\left| \frac{ \hat \mu \cdot S_{\lfloor nt \rfloor}}{\lfloor nt \rfloor} - \| \mu \| \right| \\
& \quad + \| \mu \| \sup_{0 \leq t \leq 1 } \left| \frac{\lfloor nt \rfloor}{n} - t \right| \\
& \leq  \sup_{N_\eps/n \leq t \leq 1}
\left| \frac{ \hat \mu \cdot  S_{\lfloor nt \rfloor}}{\lfloor nt \rfloor} - \| \mu \| \right| + \frac{ \| \mu \|}{n} \\
& \leq \eps +  \frac{ \| \mu \|}{n}.
\end{align*}
On the other hand,
\[ \sup_{0 \leq t \leq N_\eps/n } \left| \frac{ \hat \mu \cdot S_{\lfloor nt \rfloor}}{n} - t \| \mu \| \right|
\leq \frac{1}{n} \max \{  \hat \mu \cdot  S_0,   \ldots,  \hat \mu \cdot S_{N_\eps} \}  + \frac{N_\eps \| \mu \|}{n} \to 0, \as ,\]
since $N_\eps < \infty$ a.s. Combining these last two displays and using the fact that $\eps>0$ was arbitrary, 
we see that 
$$\sup_{0 \leq t \leq 1}  \left| n^{-1}   \hat \mu \cdot S_{\lfloor nt \rfloor} - t \| \mu \| \right| \to 0, \text{ a.s. (the
functional version of the strong law)}. $$
Similarly, 
$$\sup_{0 \leq t \leq 1}  \left| n^{-1}   \hat \mu \cdot S_{\lfloor nt \rfloor +1} - t \| \mu \| \right| \to 0, \text{ a.s. as well}.$$
Since $X_n(t)$ interpolates $S_{\lfloor nt \rfloor}$ and  $S_{\lfloor nt \rfloor +1}$, it follows that
$$\sup_{0 \leq t \leq 1}  \left| n^{-1}   \hat \mu \cdot X_n(t)  - t \| \mu \| \right| \to 0, \as $$ 
In other words,
$(n \|\mu \|)^{-1} X_n \cdot \hat \mu$ converges a.s.\ to the identity function $t \mapsto t$ on $[0,1]$.

For the other component,  Donsker's theorem (Lemma \ref{thm:donsker}) gives $( n \sperp)^{-1/2} X_n \cdot \hat \mu_\perp \Rightarrow w$ on $(\cC_1^0, \rho_\infty)$.
It follows that,   as $n \to \infty$,
  $\psi^\mu_n (  X_n ) \Rightarrow \tilde b$, 
   on $(\cC_2^0 , \rho_\infty )$. Hence by Lemma \ref{lem:path-hull} and since $\psi_n^\mu$ acts as an affine transformation on $\R^2$,
\[ \psi_n^\mu ( H ( X_n ) ) = H  ( \psi_n^\mu ( X_n  ) ) \Rightarrow H ( \tilde b ) ,\]
on $(\cK_2^0 , \rho_H )$, and the result follows.
\end{proof}

Theorem \ref{thm:limit-drift}  with the continuous mapping theorem (Lemma \ref{continuous mapping}), Lemma \ref{lem:functional-continuity},
and the fact that $\cA ( \psi_n^\mu ( A )) = n^{-3/2} \| \mu \|^{-1} ( \sperp )^{-1/2} \cA ( A )$
for measurable $A \subseteq \R^2$,
 implies
 the following distributional limit for $A_n$ in the case $\mu \neq 0$.

\begin{corollary}
\label{cor:A-limit-drift}
 Suppose that $\Exp ( \| Z_1\|^2 ) < \infty$, $\mu \neq 0$, and $\sperp >0$.
Then 
   \[ n^{-3/2} A_n \tod \| \mu \| ( \sperp )^{1/2} \tilde a_1 , \text{ as }  n \to \infty .  \]
\end{corollary}

 \begin{remarks}
(i) Only the $\sperp >0$ case is non-trivial, since
$\sperp =0$ if and only if $Z$ is parallel to $\pm \mu$ a.s., in which case all the points
$S_0, \ldots, S_n$ are collinear and $A_n = 0$ a.s.\ for all $n$.\\
(ii)
The limit in Corollary \ref{cor:A-limit-drift} is non-negative and non-degenerate (see Proposition \ref{prop:var_bounds v0 v+}
below) and hence non-Gaussian.
\end{remarks}

The framework of this chapter shows that whenever a discrete-time
process in $\R^d$ converges weakly to a limit on the space of continuous paths, the corresponding convex hulls
converge. It would be of interest to extend the framework to admit discontinuous limit processes,
such as L\'evy processes with jumps \cite{klm} that arise as scaling limits of random walks whose increments
have infinite variance.

%% file: chapter4.tex
\pagestyle{myheadings} \markright{\sc Chapter 4}

\chapter{Spitzer--Widom formula for the expected perimeter length and its consequences}
\label{chapter4}

\section{Overview}

Our contribution in this Chapter is giving a new proof of the Spitzer--Widom formula in Section \ref{sec: proof of SW} and giving the asymptotics for the expected perimeter length in Section \ref{sec: asy for per length} by using that formula.
Firstly, we show how to deduce the Spitzer--Widom formula from the Cauchy formula.

The following theorem is Theorem 2 in \cite{sw}.
\begin{theorem}[Spitzer--Widom formula]
Suppose that $\Exp \|Z_1\| < \infty$. Then $$ \Exp L_n = 2 \sum_{k=1}^n \frac{1}{k} \Exp \| S_k \|. $$
\end{theorem}

The basis for our derivation of the Spitzer--Widdom formula is an analogous result for
\emph{one-dimensional} random walk, stated in Lemma \ref{kac2} below, which is itself
a consequence of the combinatorial result given in Lemma \ref{kac}. Lemma \ref{kac} was stated by Kac \cite[pp.\ 502--503 and Theorem 4.2 on p.\ 508]{kac}
and attributed to Hunt; the proof given is due to Dyson. Lemma \ref{kac2} is variously attributed to Chung, Hunt, Dyson and Kac; it is also related
to results of Sparre Andersen \cite{andersen} and is a special case of what has become known as the Spitzer or Spitzer--Baxter identity \cite[Ch.\ 9]{kallenberg} for random walks,
which is a  more sophisticated result usually deduced from Wiener--Hopf Theory.

\section{Derivation of Spitzer--Widom formula}
\label{sec: proof of SW}

Let $X_1, X_2, \dots$ be i.i.d. random variables. Let $T_n=\sum_{i=1}^n X_i$ and $M_n=\max\{0,T_1,\dots,T_n\}$. 
Let $\sigma: (1,2,\dots,n)\mapsto (\sigma_1,\sigma_2,\dots,\sigma_n) \in \Z_+^n$ be a permutation on $\{1,\dots,n\}$. 
Then $(\pi_n; \circ)$ is a group consisting of $\sigma$ under the composition operation. For $\sigma \in \pi_n$, let $T_n^{\sigma} = \sum_{i=1}^n X_{\sigma_i}$ 
and $M_n^{\sigma}=\max\{0,T_1^{\sigma},\dots,T_n^{\sigma} \}$.

\begin{lemma} \label{kac}
$$\sum_{\sigma\in \pi_n} M_n^{\sigma} = \sum_{\sigma \in \pi_n} X_{\sigma_1} \sum_{k=1}^n \1\{T_k^\sigma > 0\}.$$
\end{lemma}
\begin{proof}
Note that if $T_k^\sigma \leq 0$, then $M_k^\sigma - M_{k-1}^\sigma = 0$. If $T_k^\sigma > 0$, then 
$$ M_k^\sigma = \max(T_1^{\sigma},T_2^{\sigma},\dots,T_k^{\sigma}) = X_{\sigma_1} + \max(0,X_{\sigma_2},X_{\sigma_2} + X_{\sigma_3},\dots,\sum_{l=2}^k X_{\sigma_l}) .$$
Combining these two cases, we get
\begin{align*}
M_k^\sigma - M_{k-1}^\sigma = 
& \1\{T_k^\sigma>0\} \Bigg[ X_{\sigma_1} + \max\Big(0,X_{\sigma_2},X_{\sigma_2} + X_{\sigma_3},\dots,\sum_{l=2}^k X_{\sigma_l}\Big) \\
& - \max\Big(0,X_{\sigma_1},X_{\sigma_1}+X_{\sigma_2},\dots,\sum_{j=1}^{k-1} X_{\sigma_j}\Big) \Bigg].
\end{align*}

Fix $k \in \{1,\dots,n \}$. Let $G(\omega_{k+1},\dots,\omega_n)$ be the subset of $\pi_n$ consisting of permutations whose last $(n-k)$ indices are $\omega_{k+1},\dots,\omega_n$, where $1 \leq \omega_i \leq n$.
Then $\pi_n$ is decomposed into $\frac{n!}{k!}$ disjoint subsets $G(\omega_{k+1},\dots,\omega_n)$ of size $k!$.

Denote $$f(\sigma_1,\dots,\sigma_{k-1},\sigma_k)
:= \max\Big(0,X_{\sigma_1}, X_{\sigma_1}+X_{\sigma_2}, \dots,\sum_{j=1}^{k-1} X_{\sigma_j}\Big) .$$
Then,
$$ M_k^\sigma - M_{k-1}^\sigma = \1\{T_k^\sigma>0\} \left[ X_{\sigma_1} + f(\sigma_2, \dots, \sigma_k, \sigma_1)- f(\sigma_1,\dots,\sigma_{k-1},\sigma_k) \right]. $$

Summing both sides of the equation over $\{\sigma \in \pi_n \}$, since 
$$\sum_{\sigma \in \pi_n} =  \sum_{1 \leq \sigma_{k+1},\dots,\sigma_n \leq n} \sum_{\sigma \in G(\sigma_{k+1},\dots,\sigma_n)} ,$$
and
$$\sum_{\sigma \in G(\sigma_{k+1},\dots,\sigma_n)} f(\sigma_2, \dots, \sigma_k, \sigma_1) = \sum_{\sigma \in G(\sigma_{k+1},\dots,\sigma_n)} f(\sigma_1,\dots,\sigma_{k-1},\sigma_k) ,$$
we get 
\begin{equation} \label{M_k}
\sum_{\sigma\in \pi_n} \left( M_k^{\sigma}-M_{k-1}^{\sigma} \right) = \sum_{\sigma \in \pi_n} X_{\sigma_1} \1\{T_k^\sigma > 0\} .
\end{equation}

The result is implied by summing both sides of the equation (\ref{M_k}) from $k=1$ to $n$. Note that $M_0^\sigma = \max(0) = 0$.
\end{proof}

Here we use the notation $x^+ := x \1\{ x>0 \}$ \label{x^+}
 and $x^- := -x \1\{ x<0 \}$ \label{x^-} for $x \in \R$. So $x=x^+ - x^-$ and $|x|= x^+ + x^-$.

The following result on the expected maximum of 1-dimensional random walk is variously attributed to Chung, Hunt, Dyson and Kac.
A combinatorial proof similar to the one given here can be found on page 301-302 of \cite{chung}.
\begin{lemma} \label{kac2}
Suppose that $\Exp |X_k| < \infty$. Then, $$\Exp M_n = \sum_{k=1}^n \frac{\Exp(T_k^+)}{k}.$$
\end{lemma}
\begin{proof}
By Lemma \ref{kac}, we have
\begin{align*}
\Exp M_n = \Exp M_n^{\sigma} 
& = \frac{1}{n!}\sum_{\sigma \in \pi_n} \Exp M_n^{\sigma} \\
& = \frac{1}{n!} \sum_{\sigma\in \pi_n} \Exp \big[X_{\sigma_1} \sum_{k=1}^n \1\{T_k^{\sigma} > 0 \} \big] \\
& = \Exp \big[ X_1 \sum_{k=1}^n \1\{T_k > 0\} \big] ,
\end{align*}
since the $X_i$ are i.i.d., $\Exp ( X_1 \1\{T_k>0 \}) = \Exp( X_i \1\{T_k>0 \})$ for any $1 \leq i \leq k$. Also, $\Exp(X_1 \1\{ T_k>0\}) = k^{-1} \Exp(T_k \1\{T_k>0\})$. Then, 
\begin{align*}
\Exp \big[ X_1 \sum_{k=1}^n \1\{T_k > 0\} \big] 
& = \sum_{k=1}^n \Exp \big[ X_1 \1\{T_k > 0\} \big] \\
& = \sum_{k=1}^n \Exp \big[ \frac{T_k}{k} \1\{T_k>0\} \big] \\
& = \sum_{k=1}^n \frac{\Exp(T_k^+)}{k} . \qedhere 
\end{align*}
\end{proof}
\begin{remark}
\emph{Fluctuation theory} for one-dimensional random walks concerns a series of important identities involving the distributions
of $M_n$, $T_n$, and other quantities associated with the random walk path. A cornerstone of the theory is the celebrated double generating-function identity of
Spitzer which states that
\[ \sum_{n=0}^\infty t^n \Exp [ \re^{ i u M_n } ] = \exp \left\{ \sum_{k=1}^\infty \frac{t^k}{k} \Exp [ \re^{iu T_k^+} ] \right\} \]
for $|t| <1$. Lemma~\ref{lem:path-stretch} is a corollary to Spitzer's identity, obtained on differentiating with respect to $u$ and setting $u=0$. The proof of
 Spitzer's identity may be approached from an analytic perspective, using the Wiener--Hopf factorization (see e.g.\ Resnick \cite[Ch.~7]{resnick}),
or from a combinatorial one (see e.g.\ Karlin and Taylor \cite[Ch.~17]{kt2}). These references discuss many other aspects of fluctuation theory, as do
Chung \cite[\S \S 8.4 \& 8.5]{chung}, Feller \cite{feller2}, Asmussen \cite[Ch.~VIII]{asmussen}, and Tak\'acs \cite{takacs2}. In particular, Chung \cite[pp.~301--302]{chung} gives a direct
proof of Lemma 4.3 closely related to the one presented here; essentially the same proof is in \cite[p.~232]{asmussen}.
\end{remark}

\begin{proof}[Proof of the Spitzer--Widom formula] $\ $\\ \vspace{0cm}
\quad Denote $M_n(\theta):= \max_{0 \leq i \leq n}(S_i \cdot \be_\theta)$ and $m_n(\theta):= \min_{0 \leq i \leq n}(S_i \cdot \be_\theta)$. Note that $M_n(\theta) \geq 0$ and $m_n(\theta) \leq 0$ since $\0 \in \HH_n$.

Applying Fubini's theorem (see Lemma \ref{fubini}) in Cauchy formula (\ref{cauchy_}), we get 
$$\Exp L_n = \int_0^\pi\left( \Exp M_n(\theta) - \Exp m_n(\theta)\right) \ud \theta .$$

Observe that $S_n \cdot \be_{\theta}$ is a one-dimensional random walk on $\R$. Take $T_k = S_k \cdot \be_{\theta}$ in Lemma \ref{kac2}. Then,
$$ \Exp M_n(\theta)= \sum_{k=1}^n \frac{\Exp \left[(S_k \cdot \be_{\theta})^+\right]}{k} \quad \hbox{and} \quad 
 \Exp m_n(\theta)= -\sum_{k=1}^n \frac{\Exp \left[(-S_k \cdot \be_{\theta})^+\right]}{k} ,$$
since $m_n(\theta) = -\max_{0 \leq i \leq n}(- S_i \cdot \be_\theta)$. So, since $x^- = (-x)^+$, 
\begin{align*}
\Exp L_n 
= & \int_0^\pi \sum_{k=1}^n \frac{1}{k} \Exp \left[(S_k \cdot \be_{\theta})^+ + (S_k \cdot \be_{\theta})^- \right] \ud\theta \\
= & \int_0^\pi \sum_{k=1}^n \frac{\Exp \left|S_k \cdot \be_{\theta}\right|}{k} \ud \theta .
\end{align*}
Then, by Fubini's theorem,
\begin{align*}
\Exp L_n 
= & \sum_{k=1}^n \frac{1}{k} \int_0^\pi \Exp \left|S_k \cdot \be_{\theta}\right| \ud \theta \\
= & \sum_{k=1}^n \frac{1}{k} \Exp \int_0^\pi \left|S_k \cdot \be_{\theta}\right| \ud \theta \\
= & 2 \sum_{k=1}^n \frac{\Exp \| S_k \|}{k} . \qedhere
\end{align*} 
\end{proof}

\section{Asymptotics for the expected perimeter length}
\label{sec: asy for per length}

To investigate the first-order properties of $\Exp L_n$, we suggested by the Spitzer-Widom formula (\ref{SW formula}) that the first-order properties of $\Exp \| S_n \|$ need to be studied first. 

\begin{lemma}
\label{ES with drift}
If $\Exp \| Z_1 \| < \infty$, then $ n^{-1}\Exp \| S_n \| \to \|\mu\| $ as $ n \to \infty$.
\end{lemma}
\begin{proof}
The strong law of large numbers for $S_n$ says $\| S_n /n - \Exp Z_1 \| \to 0 \as$ as $n \to \infty$. Then by the triangle inequality,
$$ \| S_n/n \| = \| S_n/n-\Exp Z_1 + \Exp Z_1 \| \leq \| S_n/n-\Exp Z_1 \|+ \| \Exp Z_1\| $$
and
$$ \| \Exp Z_1\| \leq \| \Exp Z_1 - S_n/n \| + \| S_n/n \|. $$
So, $\| S_n \|/n \to \| \Exp Z_1 \| \as$ as $n \to \infty$.

Similarly, let $Y_n = \sum_{i=1}^n \| Z_i \|$, then $Y_n/n \to \Exp \|Z_1 \| \as$ as $n \to \infty$. Also we simply have $\Exp[Y_n/n] = \Exp \| Z_1 \|$ and $0 \leq \| S_n \|/n \leq Y_n/n$. Hence, the result is proved by Pratt's Lemma (see Lemma \ref{pratt's lemma}).
\end{proof}

The following asymptotic result for $\Exp L_n$ was obtained as equation (2.16) by Snyder \& Steele \cite{ss} under
the stronger condition $\Exp ( \| Z_1 \|^2 ) < \infty$; as Lemma \ref{ES with drift} shows, a finite first moment is sufficient.
\begin{proposition}
\label{EL with drift}
Suppose $\Exp \| Z_1 \| < \infty$, then $n^{-1}\Exp L_n \to 2\|\mu\|$, as $n \to \infty$.
\end{proposition}
\begin{proof}
The result is implied by the Spitzer--Widom formula (\ref{SW formula}) and Lemma \ref{convergence of Cesaro mean} with $y_n=n^{-1}\Exp \| S_n\|$, since $y_n \to \|\mu\|$ by Lemma \ref{ES with drift}.
\end{proof}

\begin{remarks}
\label{remarks of EL with drift}
\begin{enumerate}[(i)]
\item
Proposition \ref{EL with drift} says that if $\mu \neq 0$ then $\Exp L_n$ is of order $n$. If $\mu=0$, it says $\Exp L_n = o(n)$. We will show later in Proposition \ref{limit of E L_n} that under mild extra conditions in the $\mu=0$ case, $n^{-1/2} \Exp L_n$ has a limit.
\item
Snyder and Steele\cite[p.\ 1168]{ss} showed that if $\Exp(\| Z_1 \|^2)< \infty$ and $\mu \neq 0$, then in fact $n^{-1}L_n \to 2\|\mu\| \as$ as $n \to \infty$. We give a proof of this in Proposition \ref{LLN for L_n} below.
\end{enumerate}
\end{remarks}

For the zero drift case $\mu=0$, we have the following.
\begin{lemma}
\label{ES 0 drift}
Suppose $\Exp(\| Z_1 \|^2) < \infty$ and $\mu=0$, then $\Exp (\| S_n \|^2)= O(n)$ and $\Exp \| S_n \|= O(n^{1/2})$.
\end{lemma}
\begin{proof}
Consider $\|S_n\|^2$,
\begin{equation} \label{S_n+1}
\| S_{n+1} \|^2 = \| S_n+Z_{n+1} \|^2 = \| S_n \|^2 + 2S_n \cdot Z_{n+1}+ \|Z_{n+1}\|^2 .
\end{equation}
So, $$\Exp(\|S_{n+1}\|^2) - \Exp(\|S_{n}\|^2) = \Exp(\|Z_{1}\|^2) ,$$
since $S_n$ and $Z_{n+1}$ are independent and $Z_{n+1}$ has mean $0$, so $\Exp(S_n \cdot Z_{n+1})=\Exp S_n \cdot \Exp Z_{n+1} = 0$.
Then sum from $n=0$ to $m-1$ to get 
$$\Exp(\|S_m\|^2) - \Exp(\|S_0\|^2)=m \Exp(\|Z_1\|^2).$$
Hence, $\Exp(\|S_n\|^2) =  O(n)$. The last result is given by Jensen's inequality, $\Exp \|S_n\| \leq (\Exp[\|S_n\|^2])^{1/2}$.
\end{proof}
\begin{remark}
Lemma \ref{ES 0 drift} only gives the upper bound for the order of $\Exp \| S_n \|$. Under the mild assumption $\Pr (\|Z_1\| =0) < 1$, $n^{-1/2}\Exp\|S_n\|$ in fact has a positive limit, as we will see in the proof of Proposition \ref{limit of E L_n} below. This extra condition is of course necessary for the positive limit, 
since if $Z_1 \equiv 0$ then $\Exp \| S_n \| \equiv 0$.
\end{remark}

\begin{proposition} \label{upper bound of E L_n}
Suppose $\Exp(\|Z_1\|^2) < \infty$ and $\mu=0$, then $\Exp L_n = O(n^{1/2})$.
\end{proposition}
\begin{proof}
By Lemma \ref{ES 0 drift} and Spitzer--Widom formula (\ref{SW formula}), for some constant C,
$$\Exp L_n \leq 2 \sum_{i=1}^n \frac{C\sqrt{i}}{i} = 2C \sum_{i=1}^n i^{-1/2} = O(n^{1/2}). \qedhere $$
\end{proof}

\begin{lemma} 
\label{lem:walk_moments}
Let $p > 1$. Suppose that $\Exp[ \|Z_1\|^p ] < \infty$.
\begin{itemize}
\item[(i)] For any $e \in \Sp_1$ such that $e \cdot \mu = 0$, $\Exp  [ \max_{0\leq m\leq n} |S_m \cdot e|^p ]  = O(n^{1 \vee (p/2)})$. 
\item [(ii)] Moreover, if $\mu = 0$, then
$\Exp [  \max_{0\leq m\leq n} \| S_m  \|^p ]  = O(n^{1 \vee (p/2)})$.
\item[(iii)] On the other hand, if $\mu \neq 0$, then
$\Exp  [\max_{0\leq m\leq n} |S_m \cdot \hat \mu|^p ]= O ( n^p )$.  
\end{itemize}
\end{lemma}
\begin{proof}
Given that $\mu \cdot e =0$, $S_n \cdot e$ is a martingale, and hence, by convexity,
$| S_n \cdot e|$ is a non-negative submartingale. Then, for $p > 1$,   
\[ \Exp \left[ \max_{0\leq m \leq n} |S_m \cdot e|^p \right] \leq \left( \frac{p}{p-1} \right)^p \Exp \left[ | S_n \cdot e |^p \right]
= O ( n^{ 1 \vee (p/2) } ) ,\]
where the first inequality is Doob's $L^p$ inequality (see Lemma \ref{Doob})
and the second is the Marcinkiewicz--Zygmund inequality (see Lemma \ref{marcinkiewicz}). This gives part (i).
 
Part (ii) follows from part (i): take $\{ e_1, e_2\}$ an orthonormal basis of $\R^2$ and apply (i) with each basis vector.
Then by the triangle inequality 
$$\max_{0\leq m\leq n} \| S_m  \|   \leq \max_{0\leq m\leq n} |S_m \cdot e_1 | + \max_{0\leq m\leq n} |S_m \cdot e_2 |$$
together
with Minkowski's inequality (see Lemma \ref{minkowski}), we have
\begin{align*}
\Exp \left[ \max_{0\leq m\leq n} \| S_m  \|^p \right]
&\leq \Exp \left[\left( \max_{0\leq m\leq n}|S_m \cdot e_1 | + \max_{0\leq m\leq n}|S_m \cdot e_2 | \right)^p \right] \\
&= \left\| \max_{0\leq m\leq n}|S_m \cdot e_1 | + \max_{0\leq m\leq n}|S_m \cdot e_2 | \right\|_p^p \\
&\leq \left( \left\|\max_{0\leq m\leq n}|S_m \cdot e_1 | \right\|_p + \left\|\max_{0\leq m\leq n}|S_m \cdot e_2 | \right\|_p \right)^p \\
&= O(n^{1 \vee (p/2)}).
\end{align*}

Part (iii) follows from the fact that 
$$\max_{0 \leq m \leq n} | S_m \cdot \hat \mu | \leq \sum_{k=1}^n | Z_k \cdot \hat \mu | \leq \sum_{k=1}^n \| Z_k \|$$
and an application of Rosenthal's inequality (see Lemma \ref{rosenthal}) to the latter sum gives
\begin{align*}
\Exp \left[ \max_{0\leq m\leq n} \| S_m \cdot \hat \mu \|^p \right]
&\leq \Exp \left[\left( \sum_{k=1}^n \|Z_k\| \right)^p\, \right] \\
&\leq \max \left\{ 2^p \sum_{k=1}^n \Exp \|Z_k\|^p ,\, 2^{p^2} \left( \sum_{k=1}^n \Exp \|Z_k\| \right)^p \right\} \\
&\leq \max \left\{ O(n) , O(n^p) \right\} \\
&\leq O(n^p). \qedhere
\end{align*} 
\end{proof}

Proposition \ref{upper bound of E L_n} gives the order of $\Exp L_n$. Now we can have the exact limit by the following result,
the statement of which is similar to an example on p.~508 of \cite{sw}.

\begin{proposition} \label{limit of E L_n}
Suppose $\Exp(\|Z_1\|^2) < \infty$ and $\mu=0$. Then, for $Y \sim \NN(\0, \Sigma)$,
 $$\lim_{n \to \infty} n^{-1/2}\Exp L_n = \Exp \cL (\Sigma^{1/2} h_1) = 4 \Exp \|Y\| .$$ 
\end{proposition}
\begin{proof}
The finite point-set case of Cauchy's formula gives
\begin{equation}
\label{eq:walk-cauchy}
L_n = \int_{\Sp_{1}} \max_{0 \leq k \leq n} ( S_k \cdot e)\ud e \leq 2 \pi \max_{0 \leq k \leq n} \| S_k \|.\end{equation}
Then by Lemma \ref{lem:walk_moments}(ii)  we have $\sup_n \Exp [ ( n^{-1/2} L_n )^{2} ] < \infty$.
Hence $n^{-1/2} L_n$ is uniformly integrable, so that Theorem \ref{thm:limit-zero} yields
$\lim_{n \to \infty} n^{-1/2}\Exp L_n = \Exp \cL ( \Sigma^{1/2} h_1 )$.

It remains to show that $\lim_{n \to \infty} n^{-1/2}\Exp L_n = 4 \Exp \| Y \|$. One can use Cauchy's formula to compute
$\Exp \cL ( \Sigma^{1/2} h_1 )$; instead we give a direct random walk argument, following \cite{sw}.
The central limit theorem for $S_n$ implies that  
 $n^{-1/2} \|S_n\| \to \|Y\|$ in distribution. Under the given conditions, $\Exp [ \|S_{n+1}\|^2 ] = \Exp [ \|S_{n}\|^2 ] + \Exp [ \|Z_{n+1} \|^2 ]$,
so that $\Exp [ \| S_n \|^2 ] = O(n)$. It follows that $n^{-1/2} \|S_n\|$ is uniformly integrable,
and hence  
$$\lim_{n \to \infty} n^{-1/2} \Exp \|S_n\| = \Exp \|Y\|.$$

So for any $\eps > 0$, there is some $n_0 \in \N$ such that $\left| k^{-1/2}\Exp \|S_k\| - \Exp\|Y\| \right| < \eps $ for all $k \geq n_0$.
Then by the S--W formula \eqref{SW formula}, we have
\begin{align*}
& \left| \frac{\Exp L_n}{\sqrt{n}} -2\Exp\|Y\| \frac{1}{\sqrt{n}} \sum_{k=1}^n k^{-1/2} \right| \\
&= \frac{2}{\sqrt{n}} \left| \sum_{k=1}^n \left( \frac{\Exp \|S_k\|}{k} - \Exp\|Y\| k^{-1/2} \right) \right| \\
&\leq \frac{2}{\sqrt{n}} \sum_{k=1}^n \left| \frac{\Exp \|S_k\|}{\sqrt{k}} - \Exp \|Y\| \right| k^{-1/2} \\
&= \frac{2}{\sqrt{n}} \left( \sum_{k=1}^{n_0} + \sum_{i = n_0 +1 }^n \right) \left| \frac{\Exp \|S_k\|}{\sqrt{k}} - \Exp\|Y\| \right| k^{-1/2} \\
&\leq \frac{D}{\sqrt{n}} + \frac{2}{\sqrt{n}} \sum_{k = n_0 +1}^n \left| \frac{\Exp \|S_k\|}{\sqrt{k}} - \Exp\|Y\| \right| k^{-1/2} \\
&\leq \frac{D}{\sqrt{n}} + \frac{2 \eps}{\sqrt{n}} \sum_{k = n_0 +1}^n k^{-1/2},
\end{align*}
for some constant $D$ and the $n_0$ mentioned above.

Also notice the fact that $\lim_{n \to \infty} n^{-1/2} \sum_{k=1}^n k^{-1/2} =2$. This can be proved by the monotonicity,
$$2\left[ (n+1)^{1/2}-1 \right] = \int_1^{n+1} x^{-1/2}\,\ud x \leq \sum_{k=1}^n k^{-1/2} \leq \int_0^n x^{-1/2}\,\ud x = 2n^{1/2} .$$

Taking $n \to \infty$ in the displayed inequality gives
$$\limsup_{n \to \infty} \left| \frac{\Exp L_n}{\sqrt{n}} -2\Exp\|Y\| \frac{1}{\sqrt{n}} \sum_{k=1}^n k^{-1/2} \right| \leq 4\eps .$$
Since $\eps > 0$ was arbitrary, it follows that
$$\lim_{n \to \infty} \left| \frac{\Exp L_n}{\sqrt{n}} -2\Exp\|Y\| \frac{1}{\sqrt{n}} \sum_{k=1}^n k^{-1/2} \right| =0 .$$
Therefore,
$$\lim_{n \to \infty} \frac{\Exp L_n}{\sqrt{n}} =  \lim_{n \to \infty}2\Exp\|Y\| \frac{1}{\sqrt{n}} \sum_{k=1}^n k^{-1/2} = 4\Exp\|Y\|. \qedhere$$
\end{proof}

Cauchy's formula applied to the line segment from $0$ to $Y$ with Fubini's theorem implies $2 \Exp \| Y \| = \int_{\Sp_1} \Exp [ ( Y \cdot e )^+ ] \ud e$.
Here $Y \cdot e = e^\tra Y$ is univariate normal
with mean $0$ and variance $e^\tra \Sigma e = \|\Sigma^{1/2} e\|^2$,  so that $\Exp[ ( Y \cdot e)^+ ]$ is  
$\|\Sigma^{1/2} e\|$ times one half of the mean of the square-root of a $\chi_1^2$ random variable. Hence
$$\Exp \| Y \| = ( 8 \pi)^{-1/2} \int_{\Sp_1} \|\Sigma^{1/2} e\|\, \ud e ,$$
which in general may be expressed via a complete elliptic integral of the second kind
in terms of the ratio of the eigenvalues of $\Sigma$.
In the particular case $\Sigma = I$, $\Exp \| Y \| = \sqrt{\pi / 2}$ so
then Proposition \ref{limit of E L_n} implies that
\[
\lim_{n \to \infty} n^{-1/2}\Exp L_n = \sqrt{8 \pi}, \]
 matching the formula 
$\Exp \ell_1 = \sqrt{8 \pi}$ of Letac and Tak\'acs \cite{letac2,takacs} (see Lemma \ref{lem:letac} below). 
We also note the bounds
\begin{equation}
\label{EL-bounds}
 \pi^{-1/2} \sqrt{ \trace \Sigma } \leq \Exp \| Y \| \leq \sqrt{ \trace \Sigma } ;\end{equation}
the upper bound here is from Jensen's inequality and the fact that $\Exp [ \| Y\|^2 ] = \trace \Sigma$.
The lower bound in \eqref{EL-bounds} follows from the inequality  
\[ \Exp \| Y \| \geq \sup_{e \in \Sp_1} \Exp | Y \cdot e |
=  \sqrt{ 2/\pi } \sup_{e \in\Sp_1}  ( \Var [ Y \cdot e ] )^{1/2} \]
together with the fact that 
\[ \sup_{e \in \Sp_1} \Var [ Y \cdot e ] 
=  \sup_{e \in \Sp_1} \|\Sigma^{1/2} e\|^2 
=   \| \Sigma^{1/2} \|^2_{\rm op} = \| \Sigma \|_{\rm op} = \lambda_\Sigma  \geq \frac{1}{2} \trace \Sigma  , 
\]
where $\| \blob  \|_{\rm op}$ is the matrix operator norm and $\lambda_\Sigma $ is the largest
eigenvalue  of $\Sigma$;
in statistical terminology, $\lambda_\Sigma$ is the variance of the first principal component associated with $Y$.

We give a proof of the formula of Letac and Tak\'acs \cite{letac2,takacs}.

\begin{lemma}
\label{lem:letac}
Let $\ell_1 = \cL(h_1)$ (see equation \eqref{eqn:def of Lt At for BM}) 
be the perimeter length of convex hull of a standard Brownian motion on $[0,1]$ in $\R^2$. Then, $\Exp \ell_1 = \sqrt{8 \pi}$.
\end{lemma}
\begin{proof}
Applying Fubini’s theorem (Lemma \ref{fubini}) in Cauchy formula \eqref{cauchy0} for $\ell_1$,
$$\ell_1 = \int_0^{2 \pi} \sup_{t \in [0,1]}(b(t) \cdot \be_{\theta}) \, d\theta ,$$ 
we have
\begin{align*}
\Exp \ell_1 &= \int_0^{2 \pi} \Exp \sup_{t \in [0,1]}(b(t) \cdot \be_{\theta}) \, d\theta \\
&= 2\pi \Exp \sup_{t \in [0,1]} (b(t) \cdot \be_{\theta}), ~\text{where $b(t) \cdot \be_{\theta}$ is a 1 dimensional Brownian motion,} \\
&= 2\pi \Exp \sup_{t \in [0,1]} w(t). 
\end{align*}
Here $w(t)$ is defined as a standard 1-dimensional Brownian motion, which is the same as in Corollary \ref{reflection}. Then we have
\begin{align*}
\Exp \sup_{t \in [0,1]} w(t)
&= \int_0^{\infty} \Pr \left( \sup_{t \in [0,1]} w(t) >r \right)\ud r \\
&= 2 \int_0^{\infty} \Pr \left( w(1) > r \right) \ud r, \text{ by Reflection principle (Corollary \ref{reflection}),} \\
&= 2 \int_0^{\infty} \frac{\ud r}{\sqrt{2 \pi}} \int_r^{\infty} e^{-y^2 /2} \,\ud y \\
&= \sqrt{\frac{2}{\pi}} \int_0^\infty \ud y \int_0^y e^{-y^2 /2}\, \ud r, \text{ by changing orders of integrals,} \\
&= \sqrt{\frac{2}{\pi}}
\end{align*}
Hence, the result follows.
\end{proof}

%% file: chapter5.tex
\pagestyle{myheadings} \markright{\sc Chapter 5}

\chapter{Asymptotics for perimeter length of the convex hull}
\label{chapter5}

\section{Overview}
\label{sec:5.1}

To start this chapter we discuss some simulations. We considered a specific form of random walk with increments 
$Z_i - \Exp [Z_i] = (\cos \Theta_i, \sin \Theta_i)$, where $\Theta_i$ was uniformly distributed on $[0, 2\pi)$,
corresponding to a uniform distribution on a unit circle centred at $\Exp [Z_i] = \mu$. 
We took one example with $\mu = \0$, and two examples with $\mu \neq \0$ of different magnitudes.

For the expected perimeter length, the simulations (see Figure \ref{fig:expper}) are consistent with the 
Spitzer--Widdom--Baxter result (see the argument below \eqref{SW formula}), Proposition \ref{limit of E L_n} and Proposition \ref{LLN for L_n}. 
In the case of $\mu = \0$, the result in Proposition \ref{limit of E L_n} take the form: 
$\lim_{n \to \infty} n^{-1/2}\Exp L_n  = 4 \Exp \|Y\| = 4$. 
In the case of $\mu \neq \0$, the result in Proposition \ref{LLN for L_n} take the form: 
$n^{-1} L_n \toas 2 \|\mu\|=0.4 \text{ or } 0.72 $.
\begin{figure}[h!]
  \centering
	\includegraphics[width=0.31\textwidth]{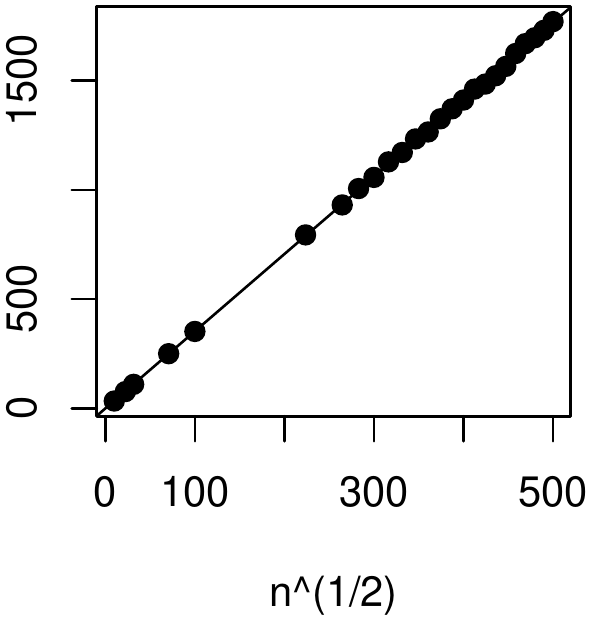} \,
	\includegraphics[width=0.31\textwidth]{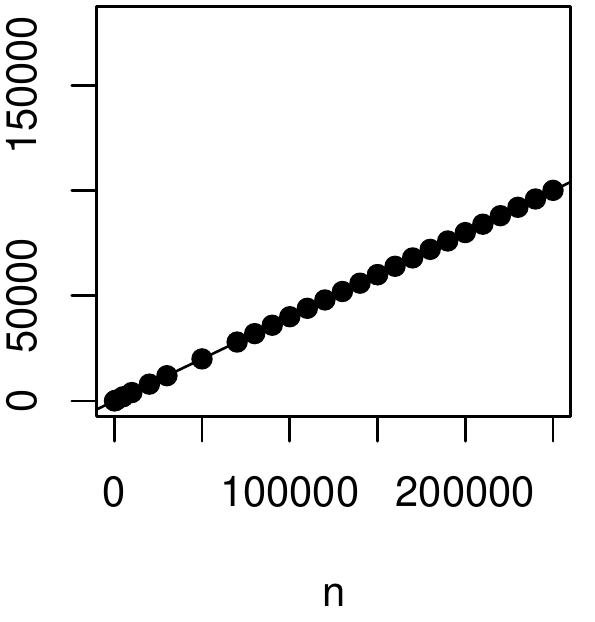} \,
	\includegraphics[width=0.31\textwidth]{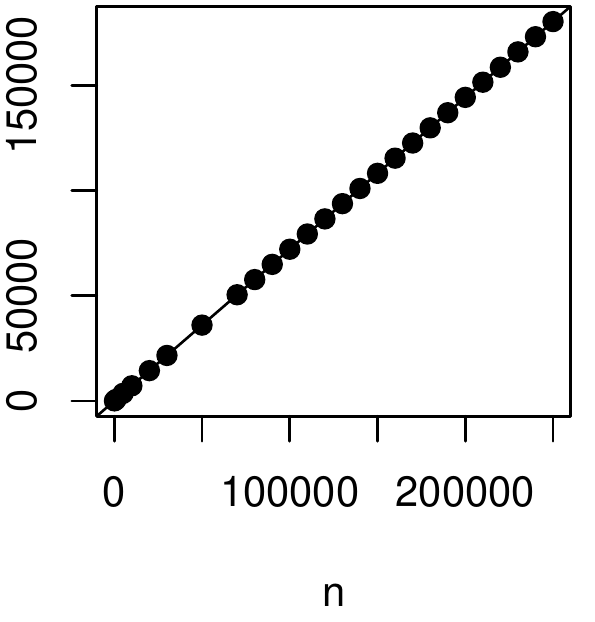} 
  \caption{Plots of $y = \Exp[L_n]$ estimates against  $x =$ (left to right) $n^{1/2}$, $n$, $n$ for
about $25$ values of $n$ in the range $10^2$ to $2.5 \times 10^5$ for 3 examples
with $\|\mu\| =$ (left to right) $0$, $0.2$, $0.36$. Each point is
estimated from $10^3$ repeated simulations. 
Also plotted are straight lines $y = 3.532 x$ (leftmost plot),
$y=0.40x$ (middle plot) and $y= 0.721x$ (rightmost plot).}
\label{fig:expper}
\end{figure}

For the variance of perimeter length with drift, the result in Theorem \ref{thm1} take the form: 
$\lim_{n \to \infty} \Var[L_n] = 4 \Exp[\cos^2 \Theta_1]=2$ and in Theorem \ref{thm2}, 
$(2n)^{-1/2}(L_n - \Exp[L_n])$ converges in distribution to a standard normal distribution. 
The corresponding pictures in Figures \ref{fig:varper} and \ref{fig:normal} show an agreement between the simulations and the theory. 
In the zero drift case, the simulations (the leftmost plot in Figure \ref{fig:varper}) suggest that 
$\lim_{n \to \infty} n^{-1} \Var[L_n]$ exists but Figure \ref{fig:normal} does not appear to be consistent with a normal distribution as a limiting distribution. 

\begin{figure}[h!]
  \centering
	\includegraphics[width=0.31\textwidth]{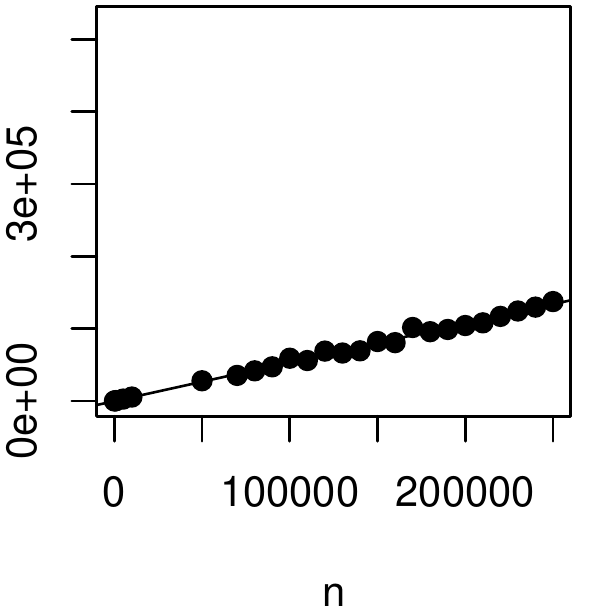} \,
	\includegraphics[width=0.31\textwidth]{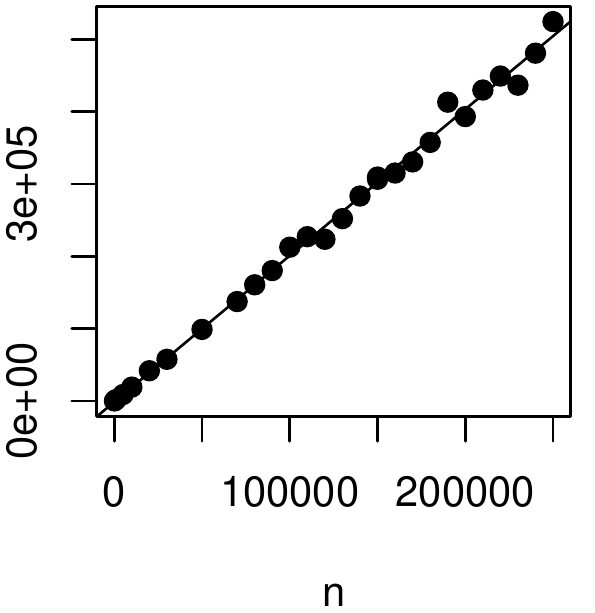} \,
	\includegraphics[width=0.31\textwidth]{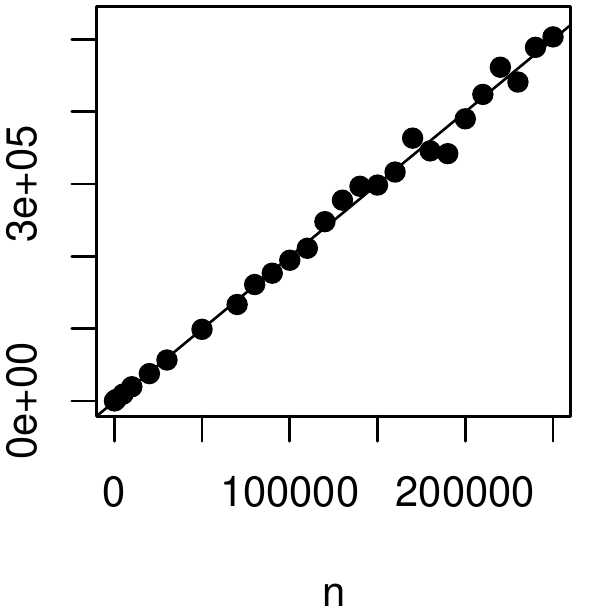}
  \caption{Plots of $y = \Var[L_n]$ estimates against  $x = n$ for 
the three examples described in Figure \ref{fig:expper}. 
Also plotted are straight lines $y = 0.536 x$ (leftmost plot)
and $y=2x$ (other two plots).}
\label{fig:varper}
\end{figure}

\begin{figure}[h!]
\centering
	\includegraphics[width=\textwidth]{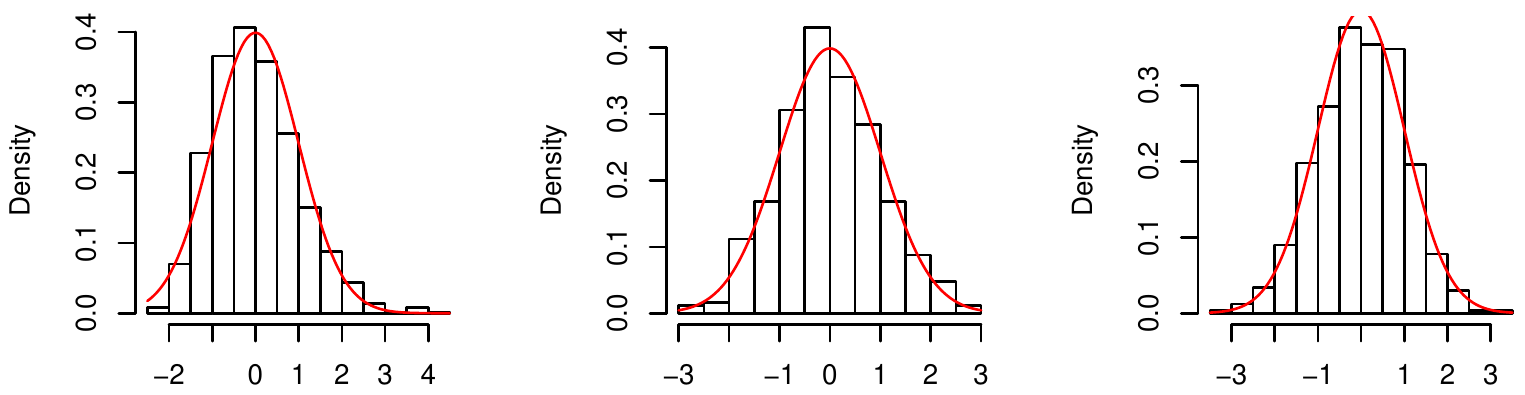}
  \caption{Simulated histogram estimates for the distribution of
  $\tfrac{L_n-\Exp [L_n]}{\sqrt{\Var[L_n]}}$
with  $n=5 \times 10^3$
in the three examples described in Figure \ref{fig:expper}. Each histogram is compiled from $10^3$ samples.}
\label{fig:normal}
\end{figure}

We will show in
Proposition \ref{prop:var-limit-zero u0} that 
$$\text{if } \mu = 0: ~~  \lim_{n \to \infty} n^{-1} \Var L_n = u_0 ( \Sigma ),$$
where $u_0( \blob )$ is finite and positive provided $\sigma^2 < \infty$.
For the constant $u_0(I)$ ($I$ being the identity matrix), Table \ref{table x1} gives numerical evaluation of rigorous bound that we prove in 
Proposition \ref{prop:var_bounds u0} below, plus estimate from simulations. See also Section~\ref{sec:exact evaluation of limiting variances} for an explicit integral expression for $u_0(I)$. 

\begin{table}[!h]
\center
\def\arraystretch{1.4}
\begin{tabular}{c|ccc}
        &   lower bound   & simulation estimate  & upper bound \\
\hline
  $u_0 ( I)$ &  $2.65 \times 10^{-3}$  &  1.08   &  9.87   \\
	\end{tabular}
\caption{The simulation estimate is
 based on $10^5$ instances of a walk of length $n = 10^5$. The final decimal digit in the numerical upper (lower)
bounds has been rounded up (down).}
\label{table x1}
\end{table}

\section{Upper bound for the variance}

Assuming that $\Exp [ \| Z_1 \|^2 ] < \infty$, Snyder and Steele \cite{ss} obtained an upper bound for $\Var [ L_n]$
 using Cauchy's formula together with a version of the Efron--Stein inequality.
Snyder and Steele's result  (Theorem 2.3 of  \cite{ss}) can be expressed as
\begin{equation}
\label{ssup}   n^{-1} \Var [L_n] \leq \frac{\pi^2}{2}
\left( \Exp [ \| Z_1 \|^2 ] - \| \Exp [ Z_1 ] \|^2 \right)
,  ~~~ (n \in \N := \{1,2,\ldots \} ) . \end{equation}

As far as we are aware, there are no  lower bounds  for $\Var [ L_n]$ in the literature.
According to the discussion in \cite[\S 5]{ss}, Snyder and Steele had ``no compelling reason to expect that $O(n)$ is the correct order of magnitude'' in their upper bound for $\Var [L_n]$,
and they speculated that perhaps $\Var [ L_n] = o(n)$ (maybe with a distinction between the cases of zero and non-zero drift).
Our first main result settles this question under minimal conditions, confirming
that (\ref{ssup}) is indeed of the correct order, apart from
in certain degenerate cases, while demonstrating that the constant on the right-hand side of (\ref{ssup}) is not, in general, sharp.

The first step in looking for the variance upper bound is a
 martingale difference argument, based on resampling members of the sequence
$Z_1, \ldots, Z_n$, to get an expression for $\Var [L_n]$ amenable to analysis: see Section \ref{sec:martingales}.
Let $\FF_0$ denote the trivial $\sigma$-algebra, and for $n \in \N$ set $\FF_n := \sigma (Z_1, \ldots, Z_n)$, the $\sigma$-algebra generated
by the first $n$ steps of the random walk. Then $S_n$ is $\FF_n$-measurable, and for $n \in \N$ we can write
$L_n = \Lambda_n ( Z_1, \ldots, Z_n )$ for $\Lambda_n : \R^{2n} \to [0,\infty)$
a measurable function.

Let $Z_1', Z_2',\ldots$ be an independent copy of the sequence $Z_1, Z_2, \ldots$.
Fix $n \in \N$. For $i \in \{1,\ldots, n\}$, we `resample' the $i$th increment, replacing $Z_i$ with $Z_i'$, as follows.
Set
\begin{equation}
\label{resample}
 S_j^{(i)} := \begin{cases} S_j & \textrm{ if } j < i \\
S_j - Z_i + Z_i' & \textrm{ if } j \geq i ;\end{cases} \end{equation}
then $(S_j^{(i)} ; 0 \leq j \leq n)$ is a modification of the random walk $(S_j ; 0 \leq j \leq n)$ that keeps all the components apart from
the $i$th step which is independently resampled. We let $L_n^{(i)}$ denote
the perimeter length of the corresponding convex hull for this modified walk, namely
$\hull ( S_0^{(i)}, \ldots, S_n^{(i)} )$,
 i.e.,
\[ L_n^{(i)} := \Lambda_n (Z_1, \ldots, Z_{i-1}, Z'_i, Z_{i+1}, \ldots, Z_n ) .\]
For $i \in \{1,\ldots, n\}$, define
\begin{equation}
\label{dni}
 D_{n, i} :=  \Exp [ L_n - L_n^{(i)}\mid \FF_{i} ] ;\end{equation}
in other words, $-D_{n,i}$ is the expected change in the perimeter length of the convex hull,
given $\FF_i$, on replacing $Z_i$ by $Z_i'$.
The point of this construction is the following result.

\begin{lemma}
\label{lem1}
Let $n \in \N$. Then (i) $L_n - \Exp [ L_n] = \sum_{i=1}^n D_{n,i}$; and (ii)
$\Var [ L_n ] = \sum_{i=1}^n \Exp [ D_{n,i}^2 ]$, whenever the latter sum is finite.
\end{lemma}
\begin{proof}
Take $W_n = L_n$ in Lemma \ref{resampling}. Then the results follow.
\end{proof}

\begin{remark} \label{remark: upper bound for var Ln}
Lemma \ref{lem1} with the conditional Jensen's inequality gives the bound
$$ \Var[L_n] \leq \sum_{i=1}^n \Exp\left[\left(L_n^{(i)} - L_n \right)^2\right] ,$$ 
which is a factor of $2$ larger than the upper bound obtained from the Efron--Stein inequality: 
$\Var[L_n] \leq 2^{-1} \sum_{i=1}^n \Exp\left[ (L_n^{(i)}-L_n)^2 \right]$ (see equation (2.3) in \cite{ss}).
\end{remark}

Let $\be_\theta = (\cos \theta, \sin \theta)$ be the unit vector in direction  $\theta \in (-\pi, \pi]$.
For $\theta \in [0,\pi]$, define
\[ M_n (\theta) := \max_{0 \leq j \leq n} ( S_j \cdot \be_\theta ) , \textrm{ and }  m_n (\theta) := \min_{0 \leq j \leq n} ( S_j \cdot \be_\theta ) .\]
Note that since $S_0 = \0$, we have $M_n (\theta) \geq 0$ and $m_n (\theta) \leq 0$, a.s.
In the present setting (see equation (\ref{cauchy_})), Cauchy's formula for convex sets yields
\[  L_n =  \int_0^\pi \left(  M_n (\theta) - m_n (\theta) \right) \ud \theta = \int_0^\pi R_n (\theta)  \ud \theta ,\]
where  $R_n (\theta) :=  M_n (\theta) - m_n (\theta) \geq 0$ is the {\em parametrized range function}. Similarly, when the $i$th increment
is resampled,
\[ L_n^{(i)} = \int_0^\pi \left(  M^{(i)}_n (\theta) - m^{(i)}_n (\theta) \right) \ud \theta = \int_0^\pi R^{(i)}_n (\theta)  \ud \theta ,\]
where $R_n^{(i)} (\theta ) =  M^{(i)}_n (\theta) - m^{(i)}_n (\theta)$, defining
\[ M^{(i)}_n (\theta) := \max_{0 \leq j \leq n} ( S^{(i)}_j \cdot \be_\theta ) , \textrm{ and }  m^{(i)}_n (\theta) := \min_{0 \leq j \leq n} ( S^{(i)}_j \cdot \be_\theta ) .\]
Thus to study $D_{n,i} = \Exp [ L_n - L_n^{(i)} \mid \FF_i]$ we will consider
\begin{equation}
\label{cauchy}
 L_n -  L_n^{(i)} = \int_0^\pi \left(  R_n (\theta) -R^{(i)}_n (\theta)  \right) \ud \theta
 = \int_0^\pi \Delta^{(i)}_{n} (\theta) \ud \theta,\end{equation}
where $\Delta^{(i)}_{n} (\theta) :=  R_n (\theta) - R^{(i)}_n (\theta)$.
For $\theta \in [0,\pi]$, let
\[ \ubar J_{n} (\theta) := \argmin_{0 \leq j \leq n } ( S_j \cdot \be_\theta ) , \textrm{ and } \bar J_{n} (\theta) := \argmax_{0 \leq j \leq n } ( S_j \cdot \be_\theta ) ,\]
so $m_n (\theta) = S_{\ubar J_n (\theta)} \cdot \be_\theta$ and
$M_n (\theta) = S_{\bar J_n (\theta)} \cdot \be_\theta$. 
 Similarly, recalling (\ref{resample}), define
\[ \ubar J^{(i)}_{n} (\theta) := \argmin_{0 \leq j \leq n } ( S^{(i)}_j \cdot \be_\theta ) , \textrm{ and } \bar J^{(i)}_{n} (\theta) := \argmax_{0 \leq j \leq n } ( S^{(i)}_j \cdot \be_\theta ) .\]
(Apply the following conventions in the event of ties: $\argmin$ takes the maximum argument
among tied values, and $\argmax$ the minimum.)

We will use the following simple bound repeatedly in the arguments that follow. 
This upper bound for $| \Delta_n ^{(i)} (\theta ) |$ is also given in Lemma 2.1 of \cite{ss}. But we have a different way to prove here.

\begin{lemma}
Almost surely, for any $\theta \in [0,\pi]$ and any $i \in \{ 1,2 ,\ldots , n\}$,
\begin{equation}
\label{deltabound}
 | \Delta_n ^{(i)} (\theta ) | \leq | (Z_i - Z_i') \cdot \be_{\theta}| \leq \| Z_i\| + \| Z_i ' \| . \end{equation}
\end{lemma}
 \begin{proof}
Consider the effect on $S_k \cdot \be_{\theta}$ when $Z_i$ is replaced by $Z_i'$. If $i > k$, then $S_k \cdot \be_{\theta} = S_k^{(i)} \cdot \be_{\theta}$.
If $i \leq k$, then $S_k \cdot \be_{\theta} = S_k^{(i)} \cdot \be_{\theta} + (Z_i - Z_i') \cdot \be_{\theta}$. Hence, for all $i$,
$$ S_k \cdot \be_{\theta} \leq S_k^{(i)} \cdot \be_{\theta} + ((Z_i - Z_i') \cdot \be_{\theta} \vee 0 ) .$$
Therefore,
$$ \max_{1 \leq k \leq n} S_k \cdot \be_{\theta} \leq \max_{1 \leq k \leq n} S_k^{(i)} \cdot \be_{\theta} + ((Z_i - Z_i') \cdot \be_{\theta} \vee 0 ) .$$
Similarly, we have
$$ \min_{1 \leq k \leq n} S_k \cdot \be_{\theta} \geq \min_{1 \leq k \leq n} S_k^{(i)} \cdot \be_{\theta} + ((Z_i - Z_i') \cdot \be_{\theta} \wedge 0 ) .$$
Combining these two inequalities with maximum and minimum, we get
\begin{align*}
R_n(\theta) - R_n^{(i)}(\theta)
& \leq ((Z_i - Z_i') \cdot \be_{\theta} \vee 0) - ((Z_i - Z_i') \cdot \be_{\theta} \wedge 0 ) \\
& = |(Z_i - Z_i') \cdot \be_{\theta}| .
\end{align*}
Also similarly, we can get $R_n^{(i)}(\theta) - R_n(\theta) \leq |(Z_i' - Z_i) \cdot \be_{\theta}|$. Thus, the result follows from the triangle inequality.
\end{proof}

The following is Lemma 2.2 in \cite{ss}.
\begin{lemma}
For all $1 \leq i \leq n$, $$\Exp\left[\left( \int_0^{\pi} \left|(Z_i- Z_i') \cdot \be_{\theta}\right| \ud \theta \right)^2 \right] 
\leq \pi^2 \left(\Exp\|Z_1\|^2 - \|\mu\|^2 \right)= \pi^2 \sigma^2 .$$
\end{lemma}

\begin{proof}
By Cauchy-Schwarz Inequality, we have
$$\Exp\left[\left( \int_0^{\pi} \left|(Z_i- Z_i') \cdot \be_{\theta}\right| \ud \theta \right)^2 \right] \leq \pi \Exp\left(\int_0^{\pi}\left|(Z_i- Z_i') \cdot \be_{\theta}\right|^2 \ud \theta \right) .$$
Then, since $Z_i$, $Z_i'$ are identically and independently distributed,
\begin{align*}
 \Exp \left[ | Z_i \cdot \be_{\theta} - Z_i' \cdot \be_{\theta} |^2 \right]
& = \Exp \left[ (Z_i \cdot \be_{\theta})^2 \right] + \Exp \left[ (Z_i' \cdot \be_{\theta})^2 \right] - 2\Exp \left[(Z_i \cdot \be_{\theta})(Z_i' \cdot \be_{\theta}) \right] \\
& = 2 \Var [Z_1 \cdot \be_{\theta}] \\
& = 2 \left(\spara \cos^2\theta + \sperp \cos^2\theta + 2 \cos \theta \sin \theta \rho_{\mu \mu_\per} \sigma_{\mu} \sigma_{\mu_\per} \right) ,
\end{align*}
where $\rho_{\mu \mu_\per}$ is the covariance of $(Z_1-\mu) \cdot \hat\mu$ and $(Z_1-\mu) \cdot \hat\mu_\per$.
So,
\begin{align*}
\Exp \int_0^{\pi} \left|(Z_i- Z_i') \cdot \be_{\theta}\right|^2 \ud \theta & = 2 \left( \spara \int_0^\pi \cos^2 \theta\,\ud \theta + \sperp \int_0^\pi \sin^2 \theta \,\ud \theta \right) \\
&\quad + 4 \rho_{\mu \mu_\per} \sigma_\mu \sigma_{\mu_\per} \int_0^\pi \cos \theta \sin \theta \,\ud \theta \\
& = \pi(\spara + \sperp) .
\end{align*}
This proves the lemma.
\end{proof}

The next result is a version of Theorem 2.3 in \cite{ss}. But they get better right-hand side by using Efron--Stein inequality
\begin{proposition} \label{upper bound for Var L_n}
Suppose $\Exp(\|Z_1\|^2) < \infty$. Then 
\begin{equation} \label{eq:ss}
\Var(L_n) \leq \frac{\pi^2 \sigma^2}{2} n.
\end{equation}
\end{proposition}

\begin{proof}
By Lemma \ref{lem:efron-stein}, equation (\ref{cauchy}) and (\ref{deltabound}),
\begin{align*}
\Var[L_n]
& \leq \frac{1}{2} \sum_{i=1}^n \Exp\left[\left(\int_0^\pi \Delta^{(i)}_{n} (\theta) \ud \theta\right)^2 \right] \\
& \leq \frac{1}{2} \sum_{i=1}^n \Exp\left[\left(\int_0^\pi \left|(Z_i- Z_i') \cdot \be_{\theta}\right| \ud \theta\right)^2 \right] \\
& \leq \frac{1}{2} \sum_{i=1}^n \pi^2 \sigma^2 \\
& = \frac{n \pi^2 \sigma^2}{2} ,
\end{align*}
since $Z_i$ are independent identically distributed.
\end{proof}

\section{Law of large numbers}

As we mentioned earlier in Remarks \ref{remarks of EL with drift}, Snyder and Steele \cite{ss} has shown the asymptotic behaviour of $L_n /n$. 
They state their law of large numbers only for $\mu \neq 0$ but the case with $\mu = 0$ works equally well. 
Here we give a different proof of the law of large numbers by using the variance bound.

\begin{proposition} \label{LLN for L_n}
If $\Exp (\|Z_1\|^2) < \infty$, then $n^{-1} L_n \to 2 \|\mu\| \as$ as $n \to \infty$.
\end{proposition}
\begin{proof}
We have $n^{-1} \Exp L_n \to 2\|\mu\| $ by Proposition \ref{EL with drift} and
the variance bound $\Var L_n \leq C n$ by Proposition \ref{upper bound for Var L_n}. Chebyshev's inequality says, for any $\eps>0$,
$$\Pr\left(\left| \frac{L_n}{n} - \frac{\Exp L_n}{n} \right|> \eps \right) \leq \frac{\Var (n^{-1} L_n)}{\eps^2} \leq \frac{C}{\eps^2 n}.$$
Take $n = n_k = k^2$, then
$$\sum_{k=1}^{\infty} \Pr \left(\left| \frac{L_{n_k}}{n_k} - \frac{\Exp L_{n_k}}{n_k} \right|> \eps \right) \leq \frac{C}{\eps^2} \sum_{k=1}^{\infty} \frac{1}{k^2} < \infty.$$
So the Borel--Cantelli lemma (see Lemma \ref{borel-cantelli}) implies that $|n_k^{-1} L_{n_k} - n_k^{-1}\Exp L_{n_k} | \to 0 \as$ as $k \to \infty$.
Hence
$$ \left| \frac{L_{n_k}}{n_k} - 2\|\mu\| \right| \leq \left| \frac{L_{n_k}}{n_k} - \frac{\Exp L_{n_k}}{n_k} \right| + \left|\frac{\Exp L_{n_k}}{n_k} -2\|\mu\| \right| \to 0 \as \As k \to 0. $$

For any $n$, let $k=\lfloor \sqrt{n} \rfloor$. Then $n_k \leq n < n_{k+1}$.
Since $L_n$ is non-decreasing in $n$ by (\ref{L_monotone}), we have
$$\frac{L_n}{n} \leq \frac{L_{n_{k+1}}}{n} \leq \frac{L_{n_{k+1}}}{n_{k+1}} \cdot \frac{n_{k+1}}{n} \leq \frac{L_{n_{k+1}}}{n_{k+1}} \cdot \frac{n_{k+1}}{n_k} ,$$
and also
$$\frac{L_n}{n} \geq \frac{L_{n_{k}}}{n} \geq \frac{L_{n_k}}{n_k} \cdot \frac{n_k}{n} \geq \frac{L_{n_k}}{n_k} \cdot \frac{n_k}{n_{k+1}} .$$
Then as $n \to \infty$, $k \to \infty$ so 
$$\frac{L_{n_k}}{n_k} \overset{a.s.}{\to} 2\|\mu\| \quad \mbox{and} \quad \frac{n_k}{n_{k+1}} = \frac{(\lfloor \sqrt{n} \rfloor)^2}{(\lfloor \sqrt{n} \rfloor+1)^2} \to 1 .$$
Therefore $n^{-1} L_n \to 2 \|\mu\| $ a.s.
\end{proof}

Proposition \ref{LLN for L_n} says that if $\Exp[ \| Z_1 \|^2 ] < \infty$ and $\mu =0$, then $n^{-1} L_n \to 0$ a.s. But Proposition
\ref{upper bound of E L_n} says that $\Exp L_n = O ( n^{1/2} )$, so we might expect to be able to improve on this `law of large numbers'. Indeed,
we have the following.

\begin{proposition} \label{2LLN for L_n}
Suppose $\Exp[ \| Z_1 \|^2 ] < \infty$. 
\begin{enumerate}[(i)]
\item
For any $\alpha > 1/2$, as $n \to \infty$,
\[ \frac{L_n - \Exp L_n}{n^{\alpha}  }  \to 0, \text{ in probability}. \]
\item  If, in addition, 
 $\mu =0$, then for any $\alpha > 1/2$, $n^{-\alpha} L_n \to 0$ a.s.\ as $n \to \infty$.
 \end{enumerate}
\end{proposition}
\begin{proof}
Similarly to the proof of Proposition \ref{LLN for L_n}, Chebyshev's inequality gives, for $\eps >0$,
\begin{equation}
\label{eqc1}
 \Pr \left( \frac{| L_n - \Exp L_n |}{n^\alpha} > \eps \right) \leq \frac{C}{\eps^2} n^{1-2\alpha} .\end{equation}
 The right-hand side here tends to $0$ as $n \to \infty$ provided $\alpha > 1/2$, giving (i).

For part (ii), take $n = n_k = 2^k$ in (\ref{eqc1}). Then
\[ \sum_{k=1}^\infty \Pr \left( \frac{| L_{n_k} - \Exp L_{n_k} |}{n_k^\alpha} > \eps \right) < \infty ,\]
provided $\alpha > 1/2$. So 
\[ \lim_{k \to \infty} \frac{| L_{n_k} - \Exp L_{n_k} |}{n_k^\alpha} = 0, \as \]
But 
\[ \lim_{k \to \infty}  \frac{  \Exp L_{n_k}  }{n_k^\alpha} = \lim_{n \to \infty} \frac{  \Exp L_{n}  }{n^\alpha} = 0 ,\]
by Proposition \ref{upper bound of E L_n}, and hence
\[ \lim_{k \to \infty}   \frac{L_{n_k}}{n_k^\alpha} = 0, \as \]
For every positive integer $n$, there exists $k(n) \in \ZP$ for which $2^{k(n)} \leq n < 2^{k(n)+1}$
and $k(n) \to \infty$ as $n \to \infty$. Hence, by (\ref{L_monotone}),
\[ \frac{L_n}{n^\alpha} \leq \frac{L_{2^{k(n)+1}}}{(2^{k(n)})^\alpha} = 2^\alpha \frac{L_{2^{k(n)+1}}}{(2^{k(n)+1})^\alpha} ,\]
which tends to $0$ a.s.\ as $n \to \infty$.
\end{proof}

Moreover, $(L_n - \Exp L_n)n^{-\alpha}$ in Proposition \ref{2LLN for L_n}(i) is also convergent to $0$ almost surely, if we assume $\|Z_1\|$ is upper bounded by some constant. To show this, we need to use Azuma--Hoeffding inequality (see Lemma \ref{azuma-hoeffding}). 

\begin{lemma} \label{p(L_n-EL_n)}
Assume $\|Z_1\| \leq B$ a.s. for some constant $B$. Then, for any $t>0$,
$$\Pr\left(| L_n - \Exp L_n | > t \right) \leq 2 \exp\left( - \frac{t^2}{8 \pi^2 B^2 n}\right). $$
\end{lemma}

\begin{proof}
Let $D_{n,i}=\Exp[L_n - L_n^{(i)} | \FF_i]$, where $\FF_0$ denote the trivial $\sigma$-algebra, and for $i \in \N$, $\FF_i = \sigma(Z_1,\dots,Z_i)$ is the $\sigma$-algebra generated by the first $n$ steps of the random walk. So $D_{n,i}$ is $\FF_i$-measurable. 
Since $L_n^{(i)}$ is independent of $Z_i$, $$\Exp [L_n^{(i)} | \FF_i] = \Exp [L_n^{(i)} | \FF_{i-1}] = \Exp [L_n | \FF_{i-1}],$$ so that
$D_{n,i}=\Exp [L_n | \FF_{i}] - \Exp [L_n | \FF_{i-1}]$. Hence, $\Exp [D_{n,i} | \FF_{i-1}]=0$.

By using equation (\ref{deltabound}) and our assumption that $\|Z_1\| \leq B$ a.s., we can deduce an upper bound for $|D_{n,i}|$ as follows.
$$ |D_{n,i}| \leq \Exp \left[ \int_0^{\pi} |\Delta_n^{(i)}(\theta)| \ud \theta \Big| \FF_{i}\right] \leq \pi (\|Z_i\| + \|Z_i'\|) \leq 2\pi B .$$
Hence, the result follows Lemma \ref{azuma-hoeffding} with $d_{\infty}=2 \pi B$.
\end{proof}

\begin{proposition}
Suppose $\| Z_1 \| \leq B$ for some constant $B$. Then for any $\alpha > 1/2$,
$$\frac{L_n - \Exp L_n}{n^{\alpha}} \to 0 \as $$
\end{proposition}

\begin{proof}
The result follows Lemma \ref{p(L_n-EL_n)} by using Borel--Cantelli Lemma (see Lemma \ref{borel-cantelli}).
\end{proof}

\section{Central limit theorem for the non-zero drift case}
\label{sec:CLT for drift}

\subsection{Control of extrema}
\label{sec:extrema}

For the remainder of this section, without loss of generality, we 
suppose that $\Exp [ Z_1 ] = \mu \be_{\pi/2}$
with $\mu \in (0,\infty)$.
Observe that $( S_j \cdot \be_\theta ; 0 \leq j \leq n)$ is a one-dimensional random walk: indeed,
$S_j \cdot \be_\theta = \sum_{k=1}^j Z_k \cdot \be_\theta$. The mean drift of this one-dimensional random walk is
\begin{equation}
 \label{mutheta}
 \Exp [ Z_1 \cdot \be_\theta ] = \Exp [ Z_1] \cdot \be_\theta = \mu \sin \theta   .
\end{equation}

Note that the drift $\mu \sin \theta$ is positive if $\theta \in (0, \pi)$.
This crucial fact gives us control over the behaviour of the extrema such as $M_n  (\theta)$ and $m_n (\theta)$
that contribute to (\ref{cauchy}), and this will allow us to estimate the conditional expectation of the final term in (\ref{cauchy})
(see Lemma \ref{estimate} below).

For  $\gamma \in (0,1/2)$ and  $\delta \in (0,\pi/2)$
(two constants that will be chosen to be suitably small later in our arguments), we denote by $E_{n,i} (\delta, \gamma)$ the event that the following   occur:
\begin{itemize}
\item for all $\theta \in [\delta, \pi - \delta ]$,
$\ubar J_{n} (\theta) < \gamma  n$ and $\bar J_{n} (\theta) > (1 - \gamma) n$;
\item for all $\theta \in [\delta, \pi - \delta ]$,
$\ubar J^{(i)}_{n} (\theta) < \gamma n$ and $\bar J^{(i)}_{n} (\theta) > (1 - \gamma) n$.
\end{itemize}
We write $E^\rc_{n,i} (\delta, \gamma)$ for the complement of $E_{n,i} (\delta, \gamma)$. The  idea is that $E_{n,i} (\delta, \gamma)$ will occur with high probability, and on this event
we have good control over $\Delta^{(i)}_{n} (\theta)$.
The next result  formalizes  these assertions. For $\gamma \in (0,1/2)$, define $I_{n, \gamma} := \{1,\ldots, n\} \cap [ \gamma n , (1-\gamma) n ]$.

\begin{lemma}
\label{En}
For any $\gamma \in (0,1/2)$ and any $\delta \in (0,\pi/2)$, the following hold.
\begin{itemize}
\item[(i)]
If $i \in I_{n,\gamma}$, then, a.s., for any $\theta \in [\delta, \pi - \delta ]$,
\begin{equation}
\label{change}
  \Delta_n^{(i)} (\theta) \1 ( E_{n,i}(\delta,\gamma) )
=   ( Z_i - Z'_i ) \cdot \be_\theta  \1 ( E_{n,i}(\delta,\gamma) )  .
\end{equation}
\item[(ii)] If  $\Exp \| Z_1 \| < \infty$ and $\| \Exp [ Z_1]\| \neq 0$, then
  $\min_{1 \leq i \leq n} \Pr [ E_{n,i} (\delta, \gamma) ] \to 1$ as $n \to \infty$.
\end{itemize}
\end{lemma}
\begin{proof}
First we prove part (i).
Suppose that $i \in I_{n,\gamma}$, so $\gamma n \leq i \leq (1-\gamma) n$.
Suppose that $\theta \in [\delta, \pi-\delta ]$.
Then on $E_{n,i} (\delta,\gamma)$, we have $\ubar J_n (\theta) < i < \bar J_n (\theta)$ and
$\ubar J_n^{(i)} (\theta) < i < \bar J_n^{(i)} (\theta)$.
Then from (\ref{resample}) it follows that in fact
$\ubar J_n (\theta) =\ubar J_n^{(i)} (\theta)$
and $\bar J_n (\theta) =\bar J_n^{(i)} (\theta)$. Hence
$m_n (\theta) = m_n^{(i)} (\theta)$ and
 $$M_n^{(i)} (\theta) = S^{(i)}_{\bar J_n (\theta)} \cdot \be_\theta = M_n (\theta) + (Z_i'-Z_i) \cdot \be_\theta, \text{ by (\ref{resample}).}$$
Equation (\ref{change}) follows.

Next we prove part (ii). Suppose that $\mu = \| \Exp [ Z_1 ] \| >0$.
Since $\Exp   \| Z_1 \| < \infty$, the strong law of large
numbers implies that $\| n^{-1} S_n - \Exp [ Z_1 ] \| \to 0$, a.s.,
as $n \to \infty$. In other words, for any $\eps_1 >0$, there
exists $N:= N (\eps_1)$ such that $\Pr [ N < \infty ] =1$
and
$ \| n^{-1} S_n - \Exp [Z_1] \| < \eps_1$ for all $n \geq N$.
In particular, for $n \geq N$, by (\ref{mutheta}),
\begin{equation}
 \label{project}
\left| n^{-1} S_n \cdot \be_\theta - \mu \sin \theta \right|
=
\left| n^{-1} S_n \cdot \be_\theta - \Exp [ Z_1] \cdot \be_\theta \right|
\leq \left\|  n^{-1} S_n - \Exp [Z_1] \right\| < \eps_1 ,\end{equation}
for all $\theta \in [0, 2\pi)$.

Take $\eps_1 < \mu \sin \delta$.
If $n \geq N$, then, by (\ref{project}),
\[ S_n \cdot \be_\theta > ( \mu \sin \theta - \eps_1 ) n \geq ( \mu \sin \delta - \eps_1 ) n,\]
provided $\theta \in [\delta, \pi-\delta]$. By choice
of $\eps_1$, the last term in the previous display is strictly
positive. Hence, for $n \geq N$, for any $\theta \in [\delta, \pi-\delta]$,
$S_n \cdot \be_\theta >0$. But, $S_0 \cdot \be_\theta =0$. So
$\ubar J_n (\theta) < N$ for all $\theta \in [\delta, \pi-\delta]$, and
\[ \Pr \left[  \cap_{\theta \in [\delta, \pi-\delta]}
\{ \ubar J_n (\theta) < \gamma n \} \right]
\geq \Pr [ N < \gamma n ] \to 1 ,\]
as $n \to \infty$, since $N < \infty$ a.s.

Now,
\begin{equation}
 \label{max}
 \max_{0 \leq j \leq (1-\gamma) n } S_j \cdot \be_\theta \leq
  \max \left\{ \max_{0 \leq j \leq N} S_j \cdot \be_\theta , \max_{N \leq j \leq (1-\gamma) n }
S_j \cdot \be_\theta \right\} .\end{equation}
For the final term on the right-hand side of (\ref{max}), (\ref{project}) implies that
\[ \max_{N \leq j \leq (1-\gamma) n }
S_j \cdot \be_\theta
\leq \max_{0 \leq j \leq (1-\gamma) n } ( \mu \sin \theta + \eps_1 ) j
\leq ( \mu \sin \theta + \eps_1 ) (1-\gamma ) n .\]
 On the other hand, if $n \geq N$, then (\ref{project}) implies that
$S_n \cdot \be_\theta \geq ( \mu \sin \theta - \eps_1 ) n$.
Here 
$$\mu \sin \theta - \eps_1 \geq ( \mu \sin \theta + \eps_1 )(1-\gamma)
\text{ if } \eps_1 < \frac{\gamma \mu \sin \theta}{2-\gamma}.$$
Now we choose $\eps_1 < \frac{\gamma \mu \sin \delta}{2}$.
Then, for any $\theta \in [\delta, \pi-\delta]$, we have that,
for $n \geq N$,
\[ S_n \cdot \be_\theta >  \max_{N \leq j \leq (1-\gamma) n }
S_j \cdot \be_\theta  .\]
Hence, by (\ref{max}),
\begin{align*}
 \Pr \left[  \cap_{\theta \in [\delta, \pi-\delta]}
\{ \bar J_n (\theta) > (1-\gamma) n \} \right]
  \geq \Pr \left[
 \cap_{\theta \in [\delta, \pi-\delta]} \left\{
S_n \cdot \be_\theta >  \max_{0 \leq j \leq (1-\gamma) n } S_j \cdot \be_\theta
\right\} \right] \\
  \geq
\Pr \left[ N \leq n ,  \, \cap_{\theta \in [\delta, \pi-\delta]} \left\{
S_n \cdot \be_\theta >  \max_{0 \leq j \leq N } S_j \cdot \be_\theta
\right\}   \right] . \end{align*}
 Also, for $n \geq N$, $S_n \cdot \be_\theta > (1- \frac{\gamma}{2}) \mu n \sin \delta$,
 so we obtain
\begin{align*} \Pr \left[  \cap_{\theta \in [\delta, \pi-\delta]}
\{ \bar J_n (\theta) > (1-\gamma) n \} \right]
  \geq
\Pr \left[ N \leq n ,  \, \max_{0 \leq j \leq N} \|  S_j \| \leq \left(1- \frac{\gamma}{2}\right) \mu n \sin \delta   \right] ,
\end{align*}
using the fact that $\max_{0 \leq j \leq N}  S_j \cdot \be_\theta
\leq \max_{0 \leq j \leq N} \|  S_j \|$ for all $\theta$.

Now, as $n \to \infty$,
 $\Pr [ N > n ] \to 0$, and
\[ \Pr  \left[ \max_{0 \leq j \leq N} \| S_j \|  > \left(1- \frac{\gamma}{2}\right) \mu n \sin \delta   \right] \to 0,\]
since $N < \infty$ a.s.
So we conclude that
\[  \Pr \left[  \cap_{\theta \in [\delta, \pi-\delta]}
\{ \ubar J_n (\theta) < \gamma n , \, \bar J_n (\theta) > (1-\gamma) n \} \right] \to 1,\]
as $n \to \infty$, and the same result holds for $\ubar J_n^{(i)} (\theta)$
and $\bar J_n^{(i)} (\theta)$, uniformly in $i \in \{1,\ldots,n\}$,
 since resampling $Z_i$ does not change the distribution of the trajectory.
\end{proof}

\subsection{Approximation for the martingale differences}

The following result is a key component to our proof. Recall that $D_{n,i} = \Exp [ L_n - L_n^{(i)} \mid \FF_i ]$.

\begin{lemma}
\label{estimate}
Suppose that $\Exp \| Z_1 \| < \infty$, $\gamma \in (0,1/2)$, and  $\delta \in (0, \pi/2)$.
For any $i \in I_{n,\gamma}$,
\begin{align}
\label{eq2}
  \left| D_{n,i} - \frac{2 ( Z_i - \Exp [ Z_1] ) \cdot \Exp [ Z_1]}{\| \Exp [ Z_1] \|}
\right|
& \leq 
4 \delta \| Z_i \| + 4 \delta \Exp  \| Z_1\| + 3 \pi \| Z_i\| \Pr [ E_{n,i}^\rc(\delta, \gamma) \mid \FF_i ] \nonumber \\
&{} \quad {} + 3 \pi \Exp [ \| Z_i'\| \1 (  E_{n,i}^\rc(\delta, \gamma) ) \mid \FF_i  ]
, \as
\end{align}
\end{lemma}
\begin{proof}
Taking (conditional) expectations
in (\ref{cauchy}), we obtain
\begin{equation}
 \label{eq4}
 D_{n,i}
 = \int_0^\pi \Exp [ \Delta_n^{(i)} (\theta) \1 (E_{n,i} (\delta, \gamma)) \mid \FF_i ] \ud \theta + \int_0^\pi \Exp [ \Delta_n^{(i)} (\theta) \1 (E_{n,i} ^\rc (\delta,\gamma)) \mid \FF_i ] \ud \theta .\end{equation}
For the second term on the right-hand side of (\ref{eq4}), we have
\begin{align}
\label{eq5}
 \left| \int_0^\pi \Exp [ \Delta_n^{(i)} (\theta) \1 (E_{n,i} ^\rc (\delta,\gamma)) \mid \FF_i ] \ud \theta \right| &
\leq \int_0^\pi \Exp [ | \Delta_n^{(i)} (\theta) | \1 (E_{n,i}^\rc (\delta,\gamma) ) \mid \FF_i ] \ud \theta . \end{align}
 Applying the bound (\ref{deltabound}), we obtain
\begin{align}
\label{eq6}
\int_0^\pi \Exp [ | \Delta_n^{(i)} (\theta) | \1 (E_{n,i}^\rc (\delta,\gamma) ) \mid \FF_i ] \ud \theta
  \leq
\pi \Exp [ ( \| Z_i\| + \| Z_i'\|  )  \1 (E_{n,i}^\rc (\delta,\gamma) ) \mid \FF_i ] \nonumber\\
  = \pi   \| Z_i \| \Pr [ E_{n,i}^\rc (\delta,\gamma)  \mid \FF_i ]
   +  \pi \Exp [ \| Z_i' \| \1 ( E_{n,i}^\rc (\delta,\gamma) ) \mid \FF_i ] ,\end{align}
since $Z_i$ is $\FF_i$-measurable with $\Exp   \| Z_i \|   < \infty$.

We decompose the   first integral on the right-hand side of (\ref{eq4}) as $I_1 + I_2 + I_3$,
where
\begin{align*} I_1 & := \int_0^{\delta}  \Exp [ \Delta_n^{(i)} (\theta) \1 (E_{n,i} (\delta,\gamma)) \mid \FF_i ] \ud \theta , \\
I_2 & := \int_{\delta}^{\pi-\delta}  \Exp [ \Delta_n^{(i)} (\theta) \1 (E_{n,i} (\delta,\gamma)) \mid \FF_i ] \ud \theta , \\
I_3 & := \int_{\pi-\delta}^{\pi}  \Exp [ \Delta_n^{(i)} (\theta) \1 (E_{n,i} (\delta,\gamma)) \mid \FF_i ] \ud \theta .\end{align*}
First we deal with $I_1$ and $I_3$.
We have
\begin{align*} | I_1 |    \leq \int_0^{\delta} \Exp  [ | \Delta^{(i)}_{n} (\theta) | \mid \FF_i ] \ud \theta   \leq 
 \delta \Exp [  \| Z_i \| + \| Z_i'\|  \mid \FF_i ] ,\as,   \end{align*}
by another application of (\ref{deltabound}).
Here $\Exp [ \| Z_i \| \mid \FF_i ] = \| Z_i \|$, since $Z_i$ is $\FF_i$-measurable,
and, since $Z_i'$ is independent of $\FF_i$,
 $\Exp [ \| Z_i'\| \mid \FF_i ] = \Exp \| Z_i'\| = \Exp \| Z_1 \|$.
A similar argument applies to $I_3$, so that
\begin{equation}
\label{eq7}
|I_1+I_3| \leq 2 \delta \| Z_i \| + 2 \delta \Exp \| Z_1 \| , \as \end{equation}

We now consider $I_2$.
From (\ref{change}), since $i \in I_{n,\gamma}$,
 we have
\begin{align*} I_2 & = \int_{\delta}^{\pi - \delta} \Exp [ (Z_i - Z'_i) \cdot \be_\theta \1 (E_{n,i} (\delta,\gamma)) \mid \FF_i ] \ud \theta
  \\
& = \int_{\delta}^{\pi - \delta} \Exp [ (Z_i - Z'_i) \cdot \be_\theta \mid \FF_i ] \ud \theta
- \int_{\delta}^{\pi - \delta} \Exp [ (Z_i - Z'_i) \cdot \be_\theta \1 (E^{\rm c}_{n,i} (\delta,\gamma)) \mid \FF_i ] \ud \theta
.\end{align*}
 Here, by the triangle inequality,
\begin{align}
&  {} \phantom{=} {}  \left| \int_{\delta}^{\pi - \delta} \Exp [ (Z_i - Z'_i) \cdot \be_\theta \1 (E^{\rm c}_{n,i} (\delta,\gamma)) \mid \FF_i ] \ud \theta \right| \nonumber\\
 & \leq \int_{0}^{\pi} \Exp [ ( \| Z_i\| + \| Z'_i\|) \1 (E^{\rm c}_{n,i} (\delta,\gamma)) \mid \FF_i ] \ud \theta \nonumber\\
\label{eq8}
 &  =
\pi   \| Z_i \| \Pr [ E_{n,i}^\rc (\delta,\gamma)  \mid \FF_i ]
   +  \pi \Exp [ \| Z_i' \| \1 ( E_{n,i}^\rc (\delta,\gamma) ) \mid \FF_i ]
,  \end{align}
similarly to (\ref{eq6}).
Finally, similarly to (\ref{eq7}),
\begin{align}     \left| \int_{\delta}^{\pi - \delta} \Exp [ (Z_i - Z'_i) \cdot \be_\theta  \mid \FF_i ] \ud \theta
- \int_{0}^{\pi} \Exp [ (Z_i - Z'_i) \cdot \be_\theta  \mid \FF_i ] \ud \theta  \right| \nonumber\\
  \leq 2 \delta \Exp [ \| Z_i\|  + \| Z'_i \| \mid \FF_i ] 
\label{eq15}
   = 2 \delta \left( \| Z_i \| + \Exp \| Z_1 \| \right) .     \end{align}

We combine (\ref{eq4}) with (\ref{eq5}) and the bounds in (\ref{eq6})--(\ref{eq15})
to give
\begin{align}
\label{eq29}
  \left| D_{n,i} - \int_{0}^{\pi} \Exp [ (Z_i - Z'_i) \cdot \be_\theta  \mid \FF_i ] \ud \theta \right|
& \leq 
4 \delta \| Z_i \| + 4 \delta \Exp  \| Z_1\| + 3 \pi \| Z_i\| \Pr [ E_{n,i}^\rc(\delta, \gamma) \mid \FF_i ] \nonumber \\
&{} \quad {} + 3 \pi \Exp [ \| Z_i'\| \1 (  E_{n,i}^\rc(\delta, \gamma) ) \mid \FF_i  ]
, \as
\end{align}
To complete the proof of the lemma, we compute the integral on the left-hand side
of (\ref{eq29}). 
First note that
$\Exp [ (Z_i - Z'_i) \cdot \be_\theta  \mid \FF_i ] = (  Z_i - \Exp [ Z'_i]  ) \cdot \be_\theta$, since
$Z_i$ is $\FF_i$-measurable and $Z_i'$ is independent of $\FF_i$,
so that
\[ \int_{0}^{\pi} \Exp [ (Z_i - Z'_i) \cdot \be_\theta  \mid \FF_i ] \ud \theta = \int_{0}^{\pi} (  Z_i -\Exp [ Z_i] ) \cdot \be_\theta \ud \theta.\]
To evaluate the last integral, it is convenient to introduce the notation $Z_i - \Exp [ Z_i] = R_i \be_{\Theta_i}$
where $R_i = \| Z_i - \Exp [ Z_i ] \| \geq 0$ and $\Theta_i \in [0, 2\pi )$.
Then
\begin{align*} \int_{0}^{\pi} (  Z_i -\Exp [ Z_i] ) \cdot \be_\theta \ud \theta &
= \int_0^\pi   R_i \be_{\Theta_i} \cdot \be_\theta   \ud \theta
= R_i \int_0^\pi \cos (\theta - \Theta_i) \ud \theta  \\
& =  2 R_i \sin \Theta_i = 2 R_i \be_{\Theta_i} \cdot \be_{\pi/2}.\end{align*}
Now (\ref{eq2}) follows from (\ref{eq29}), and the proof is complete.
\end{proof}

\subsection{Proofs for the central limit theorem}

For ease of notation, we write $Y_i := 2 \| \Exp [ Z_1] \|^{-1} ( Z_i - \Exp [ Z_1] ) \cdot \Exp [ Z_1]$, and
define
\[ W_{n,i} :=   D_{n,i} - Y_i .\]
The upper bound for $|W_{n,i}|$ in
Lemma \ref{estimate} together with Lemma \ref{En}(ii)  will enable us to prove the following result, which will be the basis of our proof of Theorem \ref{thm0}.

\begin{lemma}
\label{wni}
Suppose that $\Exp [ \| Z_1\|^2] <\infty$ and $\| \Exp [Z_1]\| \neq 0$. Then
\[ \lim_{n \to \infty} n^{-1} \sum_{i=1}^n \Exp [ W_{n,i}^2 ] = 0 . \]
\end{lemma}
\begin{proof}
Fix $\eps >0$. We take $\gamma \in (0,1/2)$ and $\delta \in (0,\pi/2)$, to be specified later. We divide the sum of interest into two parts, namely $i \in I_{n, \gamma}$ and $i \notin I_{n,\gamma}$.
 Now from (\ref{cauchy}) with (\ref{deltabound}) we have
$| L_n^{(i)} - L_n | \leq \pi ( \| Z_i \| + \| Z_i' \| )$, a.s.,
so that
\[ | D_{n,i} | \leq \pi \Exp [ \| Z_i \| + \| Z_i' \| \mid \FF_i ]
=  \pi ( \| Z_i \| + \Exp \| Z_i \| ) .\]
It then follows from  the triangle inequality that
\[ | W_{n,i} | \leq | D_{n,i} | + 2 \| Z_i - \Exp [ Z_i ] \|
\leq ( \pi + 2) ( \| Z_i \| + \Exp \| Z_i \| ) .\]
So provided $\Exp [ \| Z_1 \|^2 ] < \infty$, we have
$\Exp [ W_{n,i}^2 ] \leq C_0$ for all $n$ and all $i$, for some
constant $C_0 < \infty$, depending only on the distribution of $Z_1$.
Hence
\[ \frac{1}{n} \sum_{i \notin I_{n,\gamma} } \Exp [ W_{n,i}^2 ]
\leq \frac{1}{n} 2 \gamma n C_0 = 2 \gamma C_0 ,\]
using the fact that there are at most $2\gamma n$ terms in the sum.
From now on, choose $\gamma >0$ small enough so that $2 \gamma C_0 < \eps$.

Now consider $i \in I_{n, \gamma}$.
For such $i$, (\ref{eq2}) shows that, for some constant $C_1 < \infty$,
\begin{align}
\label{eq30}
 | W_{n,i} | & \leq C_1 (1 + \| Z_i \| )   \delta + C_1 \| Z_i \| \Pr [ E_{n,i}^\rc (\delta,\gamma) \mid \FF_i ]   \nonumber\\
& \qquad \qquad {} + C_1 \Exp [ \| Z_i'\| \1 (E_{n,i}^\rc (\delta,\gamma)) \mid \FF_i ]  ,   \as \end{align}
Here, for any $B_1 \in (0,\infty)$, a.s.,
\begin{align*}
 \Exp [ \| Z_i'\| \1 (E_{n,i}^\rc (\delta,\gamma)) \mid \FF_i ]
& \leq \Exp [ \| Z_i'\| \1 \{ \| Z_i' \| > B_1 \} \mid \FF_i ]
+ B_1 \Pr [ E_{n,i}^\rc (\delta,\gamma) \mid \FF_i ] \\
& = \Exp [ \| Z_i'\| \1 \{ \| Z_i' \| > B_1 \} ] + B_1 \Pr [ E_{n,i}^\rc (\delta,\gamma) \mid \FF_i ] ,
\end{align*}
since $Z_i'$ is independent of $\FF_i$.
Here, since $\Exp \| Z_i'\| = \Exp \|Z_1 \| < \infty$,
the dominated convergence
theorem (see Lemma \ref{dominated convergence}) implies that $\Exp [ \| Z_i'\| \1 \{ \| Z_i' \| > B_1 \} ]  \to 0$
as $B_1 \to \infty$.
So we can choose $B_1 = B_1 (\delta)$ large enough so that 
\[  \Exp [ \| Z_i'\| \1 (E_{n,i}^\rc (\delta,\gamma)) \mid \FF_i ]
\leq \delta +  B_1 \Pr [ E_{n,i}^\rc (\delta,\gamma) \mid \FF_i ], \as \]
Combining this with (\ref{eq30}) we see that there is  a constant $C_2 < \infty$ for which
\[ | W_{n,i} | \leq C_2 (1 + \| Z_i \| )  \left(
 \delta + B_1 \Pr [ E_{n,i}^\rc (\delta,\gamma) \mid \FF_i ] \right)  , \as 
\]
Hence
\begin{align*}
W_{n,i}^2 & \leq  C_2^2 (1 + \| Z_i \| )^2 \left( \delta^2 + 2B_1 \delta \Pr [ E_{n,i}^\rc (\delta,\gamma) \mid \FF_i ] +
B_1^2 \Pr [ E_{n,i}^\rc (\delta,\gamma) \mid \FF_i ]^2 \right) \\
& \leq  C_3^2 (1 + \| Z_i \| )^2 \left(  \delta + B_1^2 \Pr [ E_{n,i}^\rc (\delta,\gamma) \mid \FF_i ] \right) ,\end{align*}
for some constant $C_3 < \infty$,
using the facts that $\delta < \pi/2 < 2$ and $\Pr [ E_{n,i}^\rc (\delta,\gamma) \mid \FF_i ] \leq 1$.
Taking expectations we get
\[ \Exp [ W_{n,i}^2 ] \leq C_3^2 \delta \Exp [  (1 + \| Z_i \| )^2 ] + C_3^2 B_1^2 \Exp
\left[ (1 + \| Z_i \| )^2 \Pr [ E_{n,i}^\rc (\delta,\gamma) \mid \FF_i ] \right] .\]
Provided $\Exp [ \|Z_1 \|^2 ] < \infty$, 
there is a constant $C_4 < \infty$ such that the first term
on the right-hand side of the last display is bounded
by $C_4 \delta$. Now fix $\delta >0$ small
enough so that $C_4 \delta < \eps$; this choice also
fixes $B_1$. Then 
\begin{equation}
\label{eq32}
 \Exp [ W_{n,i}^2 ] \leq \eps  + C_3^2 B_1^2 \Exp
\left[ (1 + \| Z_i \| )^2 \Pr [ E_{n,i}^\rc (\delta,\gamma) \mid \FF_i ] \right] .\end{equation}
For the final term in (\ref{eq32}), observe that, for any $B_2 \in (0,\infty)$, a.s.,
\begin{align}
\label{eq31}
(1 + \| Z_i \|)^2 \Pr [ E_{n,i}^\rc (\delta,\gamma) \mid \FF_i ]
& \leq (1+B_2)^2 \Pr [ E_{n,i}^\rc (\delta ,\gamma) \mid \FF_i ] \nonumber\\
& {} \quad {} + (1 + \| Z_i \|)^2 \1 \{ \| Z_i \| > B_2 \} .\end{align}
Here
$\Exp [ (1 + \| Z_i \|)^2 \1 \{ \| Z_i \| > B_2 \} ] \to 0$
as $B_2 \to \infty$, provided $\Exp [ \| Z_1 \|^2] <\infty$,
by the dominated convergence theorem.
Hence, since $\delta$ and $B_1$ are fixed,
we can choose $B_2 = B_2(\eps) \in (0,\infty)$ such that
$$C_3^2 B_1^2 \Exp \left[ (1 + \| Z_i \|)^2 \1 \{ \| Z_i \| > B_2 \} \right] < \eps .$$
Then taking expectations in (\ref{eq31}) we obtain from (\ref{eq32}) that
\[ \Exp [ W_{n,i}^2 ] \leq   2 \eps + C_3^2 B_1^2  (1+B_2)^2 \Pr [ E_{n,i}^{\rm c } (\delta,\gamma) ]
  .\]

Now choose $n_0$ such that
$C_3^2 B_1^2  (1+B_2)^2 \Pr [ E_{n,i}^{\rm c } (\delta,\gamma) ] < \eps$ for all $n \geq n_0$,
which we may do by Lemma \ref{En}(ii). So for the given $\eps>0$ and  $\gamma \in (0,1/2)$,
we can choose
$n_0$ such that for all $i \in I_{n,\gamma}$ and all $n \geq n_0$,
$ \Exp [ W_{n,i}^2 ] \leq 3 \eps$. Hence
\[ \frac{1}{n} \sum_{i \in I_{n,\gamma}}  \Exp [ W_{n,i}^2 ]  \leq 3 \eps ,\]
for all $n \geq n_0$.

Combining the estimates for $i \in I_{n,\gamma}$ and $i \notin I_{n,\gamma}$, we see that
\[ \frac{1}{n} \sum_{i=1}^n  \Exp [ W_{n,i}^2 ]
\leq 2 \gamma C_0 + 3 \eps \leq 4 \eps ,\]
for all $n \geq n_0$. Since $\eps>0$ was arbitrary, the result follows.
\end{proof}

Now we can claim and prove our main theorems.

\begin{theorem}
\label{thm0}
Suppose that $\Exp [  \| Z_1 \|^2]   < \infty$ and
$\| \Exp [ Z_1 ] \|  \neq 0$. Then,
as $n \to \infty$,
\[ n^{-1/2} \left| L_n - \Exp [ L_n] -   \sum_{i=1}^n \frac{ 2 (Z_i - \Exp [ Z_1] ) \cdot \Exp [ Z_1]}{\| \Exp [Z_1] \|} \right| \to 0 , \ \textrm{in} \ L^2 .\]
\end{theorem}
 
\begin{proof}
First note that
\[ \Exp [ W_{n,i} \mid \FF_{i-1} ] = \Exp [ D_{n,i} \mid \FF_{i-1} ] - \Exp [ Y_i \mid \FF_{i-1} ]
 = 0 - \Exp [ Y_i],
\]
since $D_{n,i}$ is a martingale difference sequence and $Y_i$ is independent of $\FF_{i-1}$.
Here, by definition, $\Exp [ Y_i] =0$, and so $W_{n,i}$ is also a martingale difference
sequence. Therefore, by orthogonality,
$$n^{-1}\Exp\left[ \left( \sum_{i=1}^{n}W_{n,i} \right)^2 \right] =
n^{-1} \sum_{i=1}^{n} \Exp\left[ W_{n,i}^2 \right]  \to 0 \As n \to \infty, \text{ by Lemma \ref{wni}.} $$
In other words, $n^{-1/2} \sum_{i=1}^{n}W_{n,i} \to 0$ in $L^2$,
which, with Lemma \ref{lem1}(i), implies the statement in the theorem.
\end{proof}

\begin{theorem}
\label{thm1}
Suppose that $\Exp [  \| Z_1 \|^2]   < \infty$ and
$\| \Exp [ Z_1 ] \|  \neq 0$. Then
\begin{equation}
\label{vlim}
\lim_{n \to \infty} n^{-1} \Var [ L_n ] = \frac{4 \Exp [ ( (Z_1 - \Exp [Z_1] ) \cdot \Exp [ Z_1] )^2 ]}{\| \Exp [ Z_1] \|^2} = 4 \spara. \end{equation}
\end{theorem}

\begin{remarks} \label{rmk:degenerate}
\begin{enumerate}[(i)] 
\item The assumptions  $\Exp [  \| Z_1 \|^2]   < \infty$ and
$\| \Exp [ Z_1 ] \|  \neq 0$ ensure $4 \spara < \infty$.
\item To compare the limit result (\ref{vlim}) with Snyder and Steele's upper bound (\ref{ssup}), observe that
\[ 4 \spara = 4 \left( \frac{\Exp [ (Z_1 \cdot \Exp [ Z_1] )^2 ] - \| \Exp[ Z_1 ] \|^4 }{\| \Exp [ Z_1] \|^2 } \right)
\leq 4 \left(  \Exp[ \| Z_1 \|^2 ] - \| \Exp [ Z_1 ] \|^2 \right) .\]
\item The limit $4 \spara$ is zero if and only if $(Z_1 - \Exp [ Z_1] ) \cdot \Exp [ Z_1] =0$
with probability 1, i.e., if $Z_1 - \Exp [ Z_1]$ is always orthogonal to $\Exp[Z_1]$.
In such a degenerate case,
(\ref{vlim}) says that $\Var[L_n]=o(n)$. This is the case, for example, if $Z_1$ takes values $(1,1)$ and $(1,-1)$ each with probability $1/2$. Note that
the Snyder--Steele bound (\ref{ssup}) applied in this example says
only that $\Var [L_n] \leq (\pi^2/2) n$, which is not the correct order. Here, the two-dimensional trajectory can be viewed as a space-time trajectory of a \emph{one-dimensional} simple symmetric random walk. We conjecture that in fact $\Var [L_n] = O (\log n)$. Steele \cite{steele}
obtains variance results for the \emph{number of faces} of the convex hull
of one-dimensional simple random walk, and comments that such results for $L_n$ seem ``far out of reach'' \cite[p.\ 242]{steele}.
\end{enumerate}
\end{remarks}

\begin{proof}
Write
\begin{equation}
 \label{xizeta}
\xi_n = \frac{L_n - \Exp[L_n]}{\sqrt {  n} }; ~~\textrm{and} ~~
\zeta_n = \frac {1}{\sqrt{  n }} \sum_{i=1}^n Y_i,
~\textrm{where} ~ Y_i = \frac{2 (Z_i - \Exp [ Z_1 ] ) \cdot \Exp [ Z_1]}{\| \Exp [ Z_1 ] \|}.
\end{equation}
Then Theorem \ref{thm0} shows that $| \xi_n - \zeta_n | \to 0$ in $L^2$ as $n \to \infty$. Also,
with $4 \spara$ as given by (\ref{vlim}), $\Exp [ \zeta_n^2 ] = 4 \spara$. Then a computation shows that
\[    n^{-1} \Var [ L_n ] = \Exp [ \xi_n^2 ] = \Exp [ (\xi_n - \zeta_n)^2] + \Exp [ \zeta_n^2 ] +2 \Exp [ (\xi_n -\zeta_n) \zeta_n ] .\]
Here, by the $L^2$ convergence, $\Exp [ (\xi_n - \zeta_n)^2] \to 0$ and, by the Cauchy--Schwarz inequality (see Lemma \ref{Cauchy-Schwarz ineq.}),
$$\left| \Exp [ (\xi_n -\zeta_n) \zeta_n ] \right| \leq \left( \Exp [ (\xi_n - \zeta_n)^2] \Exp [ \zeta_n^2] \right)^{1/2} \to 0 \text{ as well.} $$
So $\Exp [ \xi_n^2] \to 4 \spara$ as $n \to \infty$.
\end{proof}

In the case where $\Exp [  \| Z_1 \|^2]   < \infty$ and
$\| \Exp  [ Z_1 ] \| = \mu > 0$, Snyder and Steele deduce from their bound (\ref{ssup}) a strong law of large numbers for $L_n$,
namely
$\lim_{n \to \infty} n^{-1} L_n = 2 \mu$, a.s.\ (see \cite[p.\ 1168]{ss}).
Given this and the variance asymptotics of Theorem \ref{thm1},
it is natural to ask whether there is an accompanying central limit theorem.
Our next result gives a positive answer in the non-degenerate case, again with essentially minimal assumptions.

In the proof of Theorem \ref{thm2} we will use two facts about convergence in distribution that we now
recall (see Lemma \ref{slutsky}).
First, if sequences of random variables
$\xi_n$ and $\zeta_n$ are such that $\zeta_n \to \zeta$ in distribution for some random variable $\zeta$ and $|\xi_n -\zeta_n| \to 0$ in probability,
then $\xi_n \to \zeta$ in distribution (this is \emph{Slutsky's theorem}). Second,
if $\zeta_n \to \zeta$ in distribution and $\alpha_n \to \alpha$ in probability, then $\alpha_n \zeta_n \to \alpha \zeta$ in distribution.

\begin{theorem}
\label{thm2}
Suppose that $\Exp [  \| Z_1 \|^2]   < \infty$, $\| \Exp [ Z_1 ] \|  \neq 0$ and $\spara > 0$. Then for any $x \in \R$,
\begin{equation}
\label{clt}
 \lim_{n \to \infty} \Pr \bigg[ \frac{L_n - \Exp [ L_n ] }{\sqrt{\Var [ L_n ]} } \leq x \bigg]
=
\lim_{n \to \infty} \Pr \bigg[ \frac{L_n - \Exp [ L_n ] }{\sqrt{4 \spara n} } \leq x \bigg]
 = \Phi (x) ,
  \end{equation}
  where $\Phi$ is the standard normal distribution function.
\end{theorem}

\begin{proof}
Use the notation for $\xi_n$ and $\zeta_n$ as given by (\ref{xizeta}).
Then, by Theorem \ref{thm0}, $| \xi_n - \zeta_n | \to 0$ in $L^2$,
and hence
in probability.

In the sum $\zeta_n$, the $Y_i$ are i.i.d.\ random variables
with mean $0$ and variance $\Exp [Y_i^2 ] = 4 \spara$. Hence the classical
central limit theorem (see e.g.\ \cite[p.\ 93]{durrett})
shows that $\zeta_n$ converges in distribution to a   normal
random variable with mean $0$ and variance $4 \spara$.
  Slutsky's theorem then implies that $\xi_n$ has the same distributional limit. Hence, for any $x \in \R$,
\[ \lim_{n\to \infty} \Pr \left[ \frac{ \xi_n}{\sqrt{ 4 \spara }}  \leq x \right] = \lim_{n\to \infty}  \Pr \left[ \frac{L_n - \Exp [ L_n]}{\sqrt{ 4 \spara n}} \leq x \right] = \Phi (x) ,\]
where $\Phi$ is the standard normal distribution function.
Moreover,
\[ \Pr \left[ \frac{L_n - \Exp [ L_n]}{\sqrt{ \Var [ L_n ]}} \leq x \right]
 = \Pr \left[ \frac{\xi_n \alpha_n}{\sqrt { 4 \spara }} \leq x \right] ,
\]
where $\alpha _n = \sqrt {\frac{4 \spara n}{\Var [L_n] }} \to 1$
by Theorem \ref{thm1}. Thus we verify the limit statements in (\ref{clt}).
\end{proof}

\section{Asymptotics for the zero drift case}

Recall that $h_1$ is defined in \eqref{h_t} and $\Sigma$ is a covariance matrix (see Section \ref{sec:Brownian-hulls}), which is positive semidefinite and symmetric.
Let
\begin{equation}
\label{eq:var_constants}
 u_0 ( \Sigma ) := \Var \cL ( \Sigma^{1/2} h_1 ) ,
\end{equation}
we have the following results.

\begin{proposition}
\label{prop:var-limit-zero u0}
 Suppose that \eqref{ass:moments} holds for some $p > 2$, 
and $\mu =0$. Then
\[ \lim_{n \to \infty} n^{-1} \Var L_n = u_0 ( \Sigma ).\]
\end{proposition}
\begin{proof}
From \eqref{eq:walk-cauchy} and Lemma \ref{lem:walk_moments}(ii), for $p > 2$ we have $\sup_n \Exp [ ( n^{-1} L_n ^2 )^{p/2} ] < \infty$.
Hence $n^{-1} L_n^2$ is uniformly integrable, and we deduce convergence of $n^{-1} \Var L_n$ in Corollary \ref{cor:zero-limits}.
\end{proof}

The next result gives bounds on $u_0(\Sigma)$ defined in \eqref{eq:var_constants}.

\begin{proposition}
\label{prop:var_bounds u0}
\begin{equation}
  \frac{263}{1080} \pi^{-3/2}  \re^{-144/25} 
\trace \Sigma
 \leq  u_0 ( \Sigma)   \leq \frac{\pi^2}{2} \trace \Sigma . \label{eq:u-bounds} 
\end{equation}
In addition, if $\Sigma = I$ we have the following sharper form of the lower bound:
\[ \Var \ell_1 = u_0 (  I ) \geq \frac{2}{5} \left(1-\frac{8}{25\pi} \right) \re^{-25 \pi /16 } > 0 .\]
\end{proposition}

For the proof of this result, we rely on a few facts about one-dimensional Brownian motion,
including the bound  (see e.g.\ equation (2.1) of  \cite{jp}), valid for all $r>0$,
\begin{equation}
\label{eq:brown_norm}
\Pr \left[ \sup_{ 0 \leq s \leq 1} | w (s) | \leq r \right] \geq 
 \frac{4}{\pi} \left( \re^{-\pi^2/(8r^2)} - \frac{1}{3} \re^{-9\pi^2/(8r^2)} \right) .
\end{equation}
We let $\Phi$ denote the distribution function of a standard normal random variable; we will also need
the standard Gaussian tail bound (see e.g.~\cite[p.~12]{durrett})
\begin{equation}
\label{eq:gauss-bound}
1 - \Phi( x) = \frac{1}{\sqrt{2 \pi}} \int_x^\infty \re^{-y^2/2} \ud y \geq \frac{1}{x \sqrt{2 \pi}}   \left(1 - \frac{1}{x^2} \right) \re^{-x^2/2} , ~~ \text{for } x > 0 .
\end{equation}
We also note  that for $e \in \Sp_1$ the diffusion
$e \cdot (\Sigma^{1/2} b)$ is one-dimensional
 Brownian motion with
variance parameter $e^\tra \Sigma e$.

The idea behind the variance lower bounds is elementary. For a  random variable $X$ with mean $\Exp X$,
we have, for any $\theta \geq 0$, 
$$\Var X = \Exp \left[ (X- \Exp X )^2 \right] \geq \theta^2 \Pr \left[ | X - \Exp X | \geq \theta \right].$$
If $\Exp X  \geq 0$,
taking $\theta = \alpha  \Exp X $ for $\alpha >0$, we obtain
\begin{equation}
\label{eq:var-bound}
 \Var X \geq \alpha^2 ( \Exp X )^2 \big( \Pr [ X \leq (1-\alpha)   \Exp X   ] + \Pr [ X \geq (1 + \alpha)   \Exp X   ] \big)
 ,\end{equation}
and  our lower bounds  use whichever of the
 latter two probabilities is most convenient.

\begin{proof}[Proof of Proposition \ref{prop:var_bounds u0}.]
We start with the upper bounds. Snyder and Steele's bound  \eqref{eq:ss}  with
the statement for $\Var L_n$ in
Proposition \ref{prop:var-limit-zero u0} gives the upper bound in \eqref{eq:u-bounds}.

We now move on to the lower bounds.
Let $e_\Sigma \in \Sp_1$ denote an eigenvector of $\Sigma$ corresponding to the principal eigenvalue $\lambda_\Sigma$.
Then since $\Sigma^{1/2} h_1$ contains the line segment from $0$ to any (other) point in $\Sigma^{1/2} h_1$, we have from monotonicity of $\cL$ that
\[ \cL ( \Sigma^{1/2} h_1 )  \geq 2 \sup_{0 \leq s \leq 1 } \| \Sigma^{1/2} b(s) \| \geq 2 \sup_{0 \leq s \leq 1 } \left( e_\Sigma \cdot ( \Sigma^{1/2} b (s) ) \right) .\]
Here $e_\Sigma \cdot ( \Sigma^{1/2} b )$ has the same distribution as $\lambda_\Sigma^{1/2} w$. Hence, for $\alpha > 0$,
\begin{align*} \Pr \left[ \cL ( \Sigma^{1/2} h_1) \geq (1+ \alpha) \Exp \cL  ( \Sigma^{1/2} h_1) \right] & 
\geq 
\Pr \left[ \sup_{0 \leq s \leq 1} w(s) \geq \frac{1+\alpha}{2} \lambda_\Sigma^{-1/2} \Exp \cL  ( \Sigma^{1/2} h_1) \right] \\
& \geq 
\Pr \left[ \sup_{0 \leq s \leq 1} w(s) \geq 2 (1+\alpha) \sqrt{2}    \right] ,\end{align*}
using the fact that $\lambda_\Sigma \geq \frac{1}{2} \trace \Sigma$ and the upper bound
in \eqref{EL-bounds}. 
Applying \eqref{eq:var-bound} to $X = \cL  ( \Sigma^{1/2} h_1) \geq 0$ gives, for $\alpha >0$,
\begin{align*}
\Var \cL (\Sigma^{1/2} h_1 ) & \geq   \alpha^2  ( \Exp \cL  ( \Sigma^{1/2} h_1) )^2 \Pr \left[ \sup_{0 \leq s \leq 1} w(s) \geq 2 (1+\alpha) \sqrt{2}    \right] \\
& \geq \frac{32}{\pi}
\alpha^2 \left( \trace \Sigma \right) \left( 1 - \Phi ( 2 (1+\alpha) \sqrt{2}  ) \right)
,\end{align*}
using the lower bound in \eqref{EL-bounds} and the fact that
$\Pr [ \sup_{0 \leq s \leq 1 } w(s) \geq r ] = 2 \Pr [ w(1) \geq r ] = 2 ( 1 -\Phi (r))$ for $r >0$, which is
a consequence of the reflection principle. Numerical curve sketching suggests that $\alpha = 1/5$ is close to optimal; this choice of $\alpha$ gives,
using \eqref{eq:gauss-bound},
\[ \Var \cL (\Sigma^{1/2} h_1 ) \geq \frac{32}{25 \pi} \left( \trace \Sigma \right) \left( 1 - \Phi (  12 \sqrt{2} /5 ) \right)
\geq  \frac{263}{1080} \pi^{-3/2} \left( \trace \Sigma \right) \exp \left\{ -\frac{144}{25} \right\} ,\]
which is the lower bound in \eqref{eq:u-bounds}.
We get a sharper result when $\Sigma = I$ and $\cL ( h_1 ) = \ell_1$, since we know $\Exp \ell_1 = \sqrt{ 8 \pi}$ explicitly.
Then, similarly to above, we get 
\[ \Var \ell_1 \geq 8 \pi \alpha^2 \Pr 
 \left[ \sup_{0 \leq s \leq 1} w(s) \geq (1+\alpha) \sqrt{2 \pi}  \right] , \text{ for } \alpha >0, \]
which at $\alpha = 1/4$ yields the stated lower bound.
\end{proof}

%% file: chapter6.tex
\pagestyle{myheadings} \markright{\sc Chapter 6}

\chapter{Results on area of the convex hull}
\label{chapter6}

\section{Overview}

The aims of the present chapter are to provide first and second-order information for $A_n$ in both the cases $\mu=0$ and $\mu \neq 0$.
We start by some simulations. We considered the same form of random walk as in Section \ref{sec:5.1}.

For the expected area, the simulations (see Figure \ref{fig:exparea}) are consistent with Theorem \ref{prop:EA-zero} and Theorem \ref{prop:EA-drift}. 
In the case of $\mu = \0$, Theorem \ref{prop:EA-zero} implies: 
$\lim_{n \to \infty} n^{-1} \Exp A_n = \frac{\pi}{2} \sqrt{\det \Sigma} = 0.785$. 
In the case of $\mu \neq 0$, Theorem \ref{prop:EA-drift} takes the form: 
$\lim_{n \to \infty} n^{-3/2} \Exp A_n = \frac{1}{3} \| \mu \|  \sqrt{2\pi \sperp} = 0.236$ or $0.425$. 

\begin{figure}[h!]
  \centering
	\includegraphics[width=0.32\textwidth]{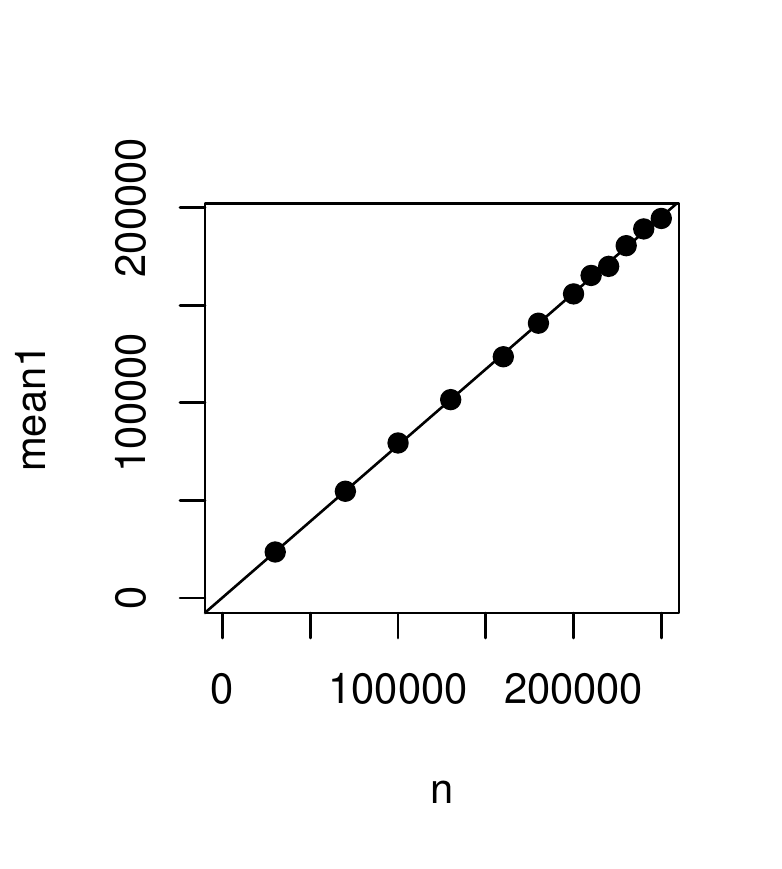} 
	\includegraphics[width=0.32\textwidth]{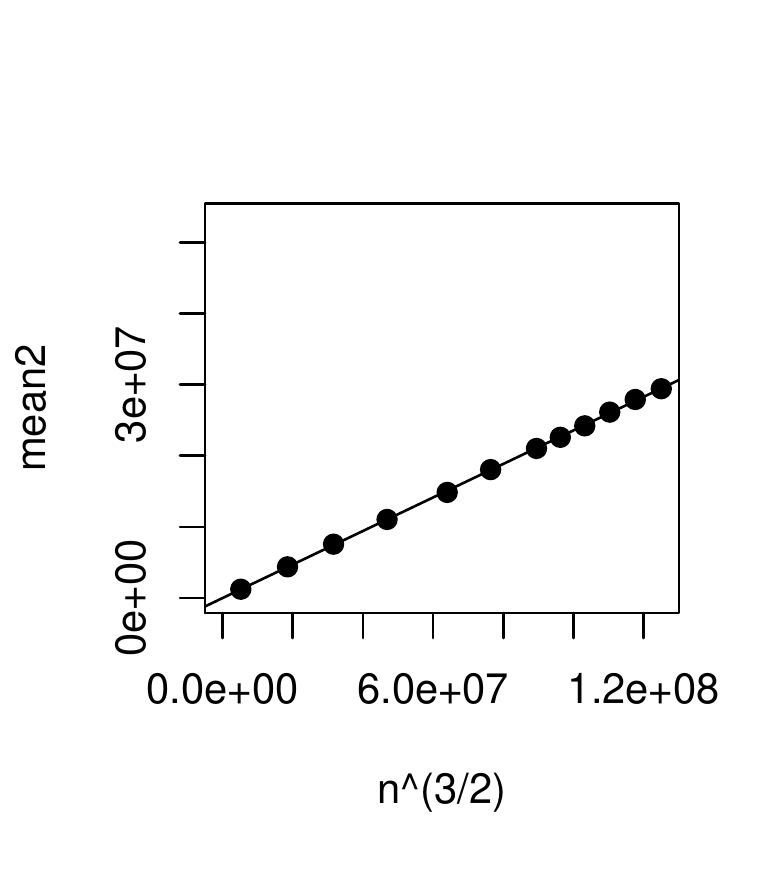} 
	\includegraphics[width=0.32\textwidth]{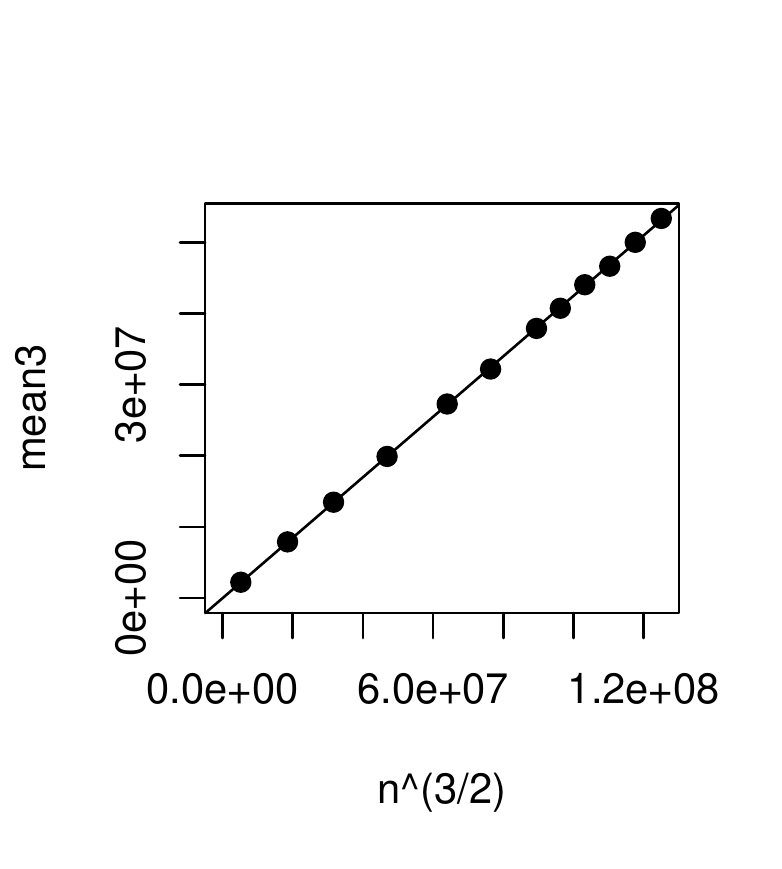}
  \caption{Plots of $y = \Exp[A_n]$ estimates against  $x =$ (left to right) $n$, $n^{3/2}$, $n^{3/2}$ for
about $25$ values of $n$ in the range $10^2$ to $2.5 \times 10^5$ for 3 examples
with $\|\mu\| =$ (left to right) $0$, $0.4$, $0.72$. Each point is
estimated from $10^3$ repeated simulations. 
Also plotted are straight lines $y = 0.781 x$ (leftmost plot),
$y=0.236x$ (middle plot) and $y= 0.425x$ (rightmost plot).} \label{fig:exparea}
\end{figure}

For the variance of area, Proposition \ref{prop:var-limit-zero} and \ref{prop:var-limit-drift v+} show that the 
limits for variance exist in both zero and non-zero drift cases. 
For example,  we will show that
\begin{alignat}{1}
\label{eq:two_vars10}
\text{if } \mu \neq 0: ~~ &   \lim_{n \to \infty} n^{-3} \Var A_n   = v_+ \| \mu \|^2 \sperp ; \nonumber\\
\text{if } \mu = 0: ~~  &    \lim_{n \to \infty} n^{-2} \Var A_n   = v_0 \det \Sigma  
 ,\end{alignat}
where $v_0$ and $v_+$ are finite and positive, and these quantities are in fact variances associated with convex hulls of Brownian scaling limits for the walk.
These scaling limits provide the basis of the analysis in this chapter; the methods are necessarily quite different
from those in \cite{wx}. 
For the constants  $v_0$ and $v_+$, Table~\ref{table2} gives 
numerical evaluations of rigorous bounds that we prove in Proposition~\ref{prop:var_bounds v0 v+} below, plus estimates from simulations. 
\begin{table}[!h]
\center
\def\arraystretch{1.4}
\begin{tabular}{c|ccc}
        &   lower bound   & simulation estimate  & upper bound \\
\hline
  $v_0$      &  $8.15  \times 10^{-7}$ &  0.30   &  5.22   \\
  $v_+$      &  $1.44  \times 10^{-6}$ &  0.019  &  2.08   
	\end{tabular}
\caption{Each of the simulation estimates is
 based on $10^5$ instances of a walk of length $n = 10^5$. The final decimal digit in each of the numerical upper (lower)
bounds has been rounded up (down).}
\label{table2}
\end{table}
The variance limits we deduced in the simulations 
(see Figure \ref{fig:vararea}) are indeed lie in the variance bounds given by Proposition \ref{prop:var_bounds v0 v+}.

\begin{figure}[h!]
  \centering
	\includegraphics[width=0.31\textwidth]{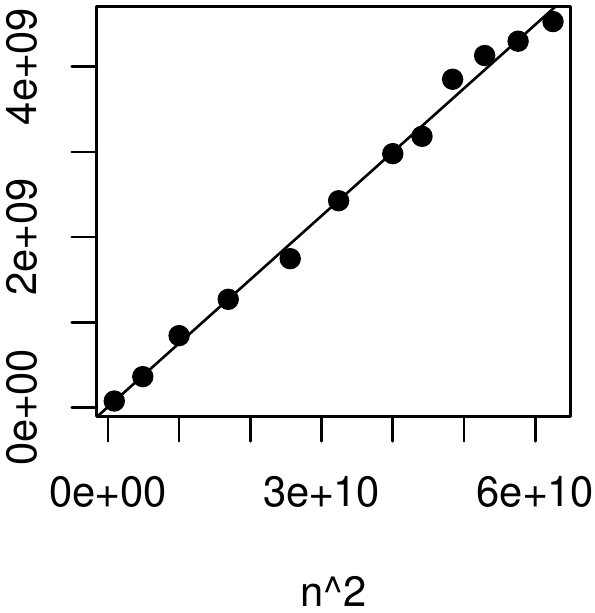} \,
	\includegraphics[width=0.31\textwidth]{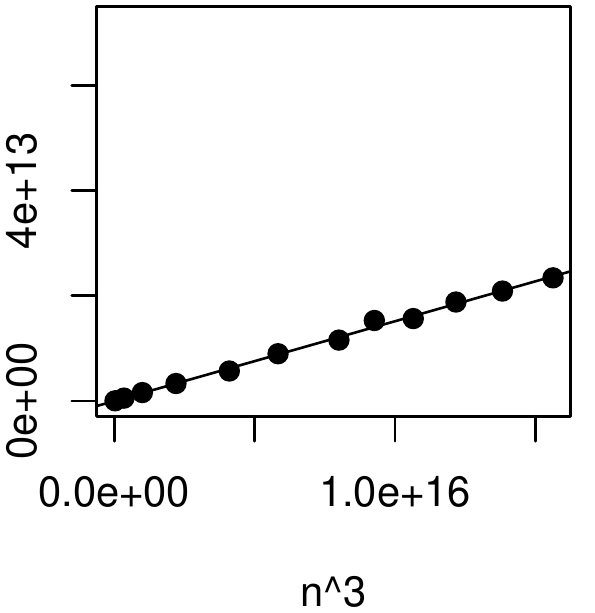} \,
	\includegraphics[width=0.31\textwidth]{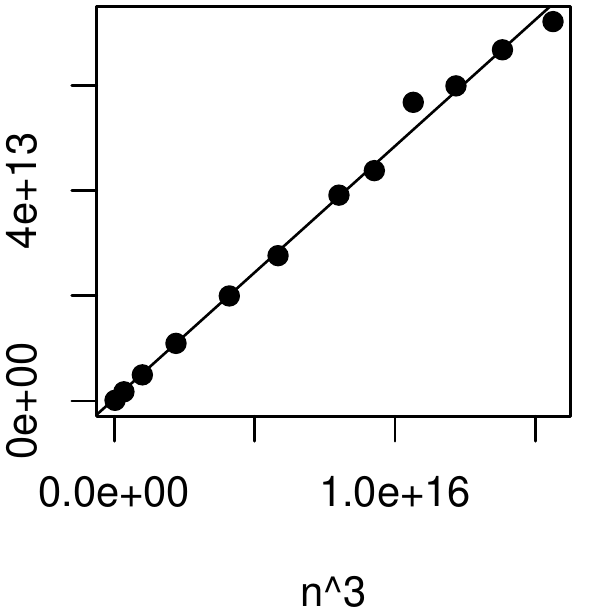}
  \caption{Plots of $y = \Var[A_n]$ estimates against $x =$ (left to right) $n^2$, $n^{3}$, $n^{3}$ for 
the three examples described in Figure \ref{fig:exparea}. 
Also plotted are straight lines $y = 0.0748 x$ (leftmost plot),
$y=0.00152x$ (middle plot) and $y= 0.00480x$ (rightmost plot).} \label{fig:vararea}
\end{figure}

\section{Upper bound for the expected value and variance for the area}


\begin{proposition} 
\label{lem:A_moments}
Let $p \geq 1$. Suppose that $\Exp [ \| Z_1 \|^{2p} ] < \infty$. 
\begin{itemize}
\item[(i)] We have $\Exp [ A_n^p ] = O ( n^{3p/2} )$. Suppose in addition $\Exp(\|Z_1 \|^{4p}) < \infty$, then $\Var (A_n^p)=O(n^{3p})$.
\item[(ii)] Moreover, if $\mu =0$ we have $\Exp [ A_n^p ] = O (n^{p} )$. Suppose in addition $\Exp(\|Z_1 \|^{4p}) < \infty$, then $\Var (A_n^p)=O(n^{2p})$.
\end{itemize}
\end{proposition}

\begin{proof}
For part (i),  it suffices to suppose $\mu \neq 0$. Then, bounding the convex hull by a rectangle,
\begin{align*}
A_n
& \leq 
\left(\max_{0\leq m \leq n} S_m \cdot \hat \mu - \min_{0\leq m \leq n} S_m \cdot \hat \mu \right) \left(\max_{0\leq m \leq n} S_m \cdot \hat \mu_\perp - \min_{0\leq m \leq n} S_m \cdot \hat \mu_\perp \right)\\
& \leq 
4 \left(\max_{0\leq m \leq n} |S_m \cdot \hat \mu| \right) \left(\max_{0\leq m \leq n} |S_m \cdot \hat \mu_\perp| \right) .
\end{align*}
Hence, by the Cauchy--Schwarz inequality, we have
$$ \Exp [ A_n^p ] \leq 4^p \left( \Exp \left[ \max_{0\leq m \leq n} |S_m \cdot \hat \mu|^{2p} \right] \right)^{1/2} 
\left(\Exp \left[ \max_{0\leq m \leq n} |S_m \cdot \hat \mu_\perp|^{2p} \right] \right)^{1/2} .$$
Now an application of Proposition \ref{lem:walk_moments}(i) and (iii) gives $\Exp [ A_n^p ] = O ( n^{3p/2} )$.

Suppose in addition $\Exp(\|Z_1 \|^{4p}) < \infty$. By the same process as above, we have
$$A_n^{2p} \leq 4^{2p} \left(\max_{0\leq m \leq n} |S_m \cdot \hat \mu|^{2p} \right) \left(\max_{0\leq m \leq n} |S_m \cdot \hat \mu_\perp|^{2p} \right) ,$$
and $ \Exp (A_n^{2p}) = O(n^{3p})$.
Hence, $\Var (A_n^p) = \Exp \left(A_n^{2p}\right) - \left(\Exp A_n^p \right)^{2} = O(n^{3p})$.

For part (ii), $\mu = 0$.
Since the convex $\text{hull}(S_0, \dots, S_n)$ is contained in the disk of radius $\max_{0 \leq m \leq n} \|S_m\|$ and centre $0$, $A_n^p \leq \pi^p (\max_{0 \leq m \leq n} \|S_m\|^{2p})$ a.s. 
Proposition \ref{lem:walk_moments}(ii) then yields $\Exp [ A_n^p ] = O (n^{p})$.

Suppose in addition $\Exp(\|Z_1 \|^{4p}) < \infty$. By the same process as above, we have $\Exp [ A_n^{2p} ] = O (n^{2p})$. Therefore, $\Var (A_n^p)=O(n^{2p})$.
\end{proof}

\begin{remark}
 We will show below in Theorem \ref{prop:EA-drift} $n^{-3/2}\Exp A_n$ has a limit in the non-zero drift case and, in Proposition \ref{prop:EA-zero}, $n^{-1} \Exp A_n$ has a limit in the zero drift case.
\end{remark}

\section{Asymptotics for the expected area}

Let $T(\bu,\bv)$ ($\bu,\bv \in \R^2$) be the area of a triangle with sides of $\bu,\bv$ and $\bu + \bv$. Then, 
$$T(\bu,\bv)= \frac{1}{2} \sqrt{\|\bu\|^2 \|\bv\|^2 - (\bu \cdot \bv)^2} .$$
For $\alpha, \beta >0$, $T(\alpha \bu, \beta \bv) = \alpha \beta T(\bu,\bv)$.

\begin{lemma} \label{CLT for ET}
Suppose $\Exp (\|Z_1\|^2)< \infty$, $\Exp Z_1 =\0$ and $\Exp (Z_1^{T} Z_1) = \Sigma$. Then as $m \to \infty$ and $(k-m) \to \infty$,
$$\frac{\Exp T(S_m, S_k-S_m)}{\sqrt{m(k-m)}} \to \Exp T(Y_1,Y_2) ,$$
where $Y_1$, $Y_2$ are iid. rvs. $Y_1, Y_2 \sim \NN(\0, \Sigma)$.
\end{lemma}

\begin{proof}
By Central Limit Theorem in $\R^2$ (see \cite{durrett}), $n^{-1/2} S_n \overset{d.}{\to} \NN(\0, \Sigma)$.
Since $S_m$ and $S_k - S_m$ are independent, as $m$ and $k-m \to \infty$, 
$$\left( \frac{S_m}{\sqrt{m}}, \frac{S_k-S_m}{\sqrt{k-m}} \right) \overset{d.}{\to} T(Y_1,Y_2) .$$
Using the fact $T$ is continuous, 
$$\frac{T(S_m, S_k-S_m)}{\sqrt{m(k-m)}} = T\left( \frac{S_m}{\sqrt{m}}, \frac{S_k-S_m}{\sqrt{k-m}} \right) \overset{d.}{\to} T(Y_1,Y_2) .$$

Also, by Lemma \ref{ES 0 drift},
\begin{align*}
\Exp \left(\left[ \frac{\Exp T(S_m, S_k-S_m)}{\sqrt{m(k-m)}} \right]^2 \right) 
& \leq \frac{\Exp(\|S_m\|^2 \|S_k-S_m\|^2)}{m(k-m)} \\
& \leq \frac{\Exp \|S_m\|^2}{m} \cdot\frac{\Exp \|S_k-S_m\|^2}{k-m} < \infty.
\end{align*}
That means $m^{-1/2}(k-m)^{-1/2}T(S_m,S_k - S_m)$ is uniformly integrable over $(m, k)$ with $m \geq 1$, $k \geq m+1$. So the result follows.
\end{proof}

We state the following result without proof. It is a higher dimensional analogue of S--W formula \eqref{SW formula}. See Barndorff--Nielson and Baxter \cite{baxter} for the proof.

\begin{lemma}[Barndorff Nielsen \& Baxter]
\begin{equation} \label{E A_n formula}
\Exp(A_n)= \sum_{k=2}^n \sum_{m=1}^{k-1} \frac{\Exp\big[ T(S_m,S_k-S_m)\big]}{m(k-m)} .
\end{equation}
\end{lemma}

\begin{lemma} \label{single sum limit}
$$ \lim_{k \to \infty} \sum_{m=1}^{k-1} \frac{1}{m^{1/2}(k-m)^{1/2}} = \pi .$$
\end{lemma}
\begin{proof}
Let $f(m,k)=m^{-1/2}(k-m)^{-1/2}$. For any $\delta \in (0,1)$, we have $f(m,k) \leq f(m-\delta,k)$ if $m \leq k/2$ and $f(m,k) \geq f(m-\delta,k)$ if $m \geq k/2$.
Consider the sum as two parts, 
$$\sum_{m=1}^{k-1} f(m,k) = \left( \sum_{m=1}^{\lfloor k/2 \rfloor} + \sum_{m=\lfloor k/2 \rfloor +1}^{k-1} \right) f(m,k) .$$
Then,
\begin{align*}
 \sum_{m=1}^{k-1} f(m,k)
& \geq \int_{1}^{\lfloor k/2 \rfloor} f(m,k)\,\ud m + \int_{\lfloor k/2 \rfloor +1}^{k-1} f(m-1,k)\,\ud m, \\
& \qquad\qquad \text{ by letting } u=\frac{m}{k} \text{ and } v=\frac{m-1}{k}, \\
& = \int_{1/k}^{\lfloor \frac{k}{2} \rfloor /k} \frac{1}{\sqrt{u(1-u)}}\,\ud u + \int_{\lfloor \frac{k}{2} \rfloor /k}^{1-\frac{2}{k}} \frac{1}{\sqrt{v(1-v)}}\,\ud v \\
& = \int_{1/k}^{1-2/k} \frac{1}{\sqrt{u(1-u)}}\,\ud u .
\end{align*}
Also,
\begin{align*}
 \sum_{m=1}^{k-1} f(m,k)
& \leq \int_{1}^{\lfloor k/2 \rfloor} f(m-1,k)\,\ud m + \int_{\lfloor k/2 \rfloor +1}^{k-1} f(m,k)\,\ud m \\
& = \int_{0}^{\lfloor \frac{k}{2} \rfloor /k - 1/k} \frac{1}{\sqrt{u(1-u)}}\,\ud u + \int_{\lfloor \frac{k}{2} \rfloor /k +1/k}^{1-\frac{1}{k}} \frac{1}{\sqrt{v(1-v)}}\,\ud v \\
& \leq \int_{0}^{1-1/k} \frac{1}{\sqrt{u(1-u)}}\,\ud u .
\end{align*}
Therefore, 
\begin{equation}
\label{eq:pi-sum}
\lim_{k \to \infty}\sum_{m=1}^{k-1} f(m,k) = \int_0^1 [u(1-u)]^{-1/2}\,\ud u = B\left(\frac{1}{2},\frac{1}{2}\right) = \Gamma\left(\frac{1}{2}\right)^2 = \pi ,\end{equation}
where $B(\blob, \blob)$ is the Beta function and $\Gamma(\blob)$ is the Gamma function.
\end{proof}

\begin{lemma} \label{double sum limit}
$$\lim_{n \to \infty} \frac{1}{n} \sum_{k=2}^n \sum_{m=1}^{k-1} \frac{1}{m^{1/2}(k-m)^{1/2}} = \pi .$$
\end{lemma}
\begin{proof}
The result follows from Lemma \ref{convergence of Cesaro mean} and Lemma \ref{single sum limit}.
\end{proof}

\begin{proposition} \label{lim E A/n with 0 drift}
Suppose $\Exp(\|Z_1\|^2)<\infty$ and $\mu = 0$. Then, $$ \lim_{n \to \infty}\frac{\Exp A_n}{n} = \pi \Exp T(Y_1, Y_2) ,$$ where $Y_1$, $Y_2$ are iid. rvs. 
$Y_1, Y_2 \sim \NN(\0, \Sigma)$ and $\Sigma = \Exp (Z_1^T Z_1)$.
\end{proposition}
\begin{proof}
In (\ref{E A_n formula}), denote $ g(k,m):= m^{-1/2}(k-m)^{-1/2} \Exp\big[ T(S_m,S_k-S_m)\big]$. Then, 
\begin{equation} 
\label{eq:bnb}
\Exp A_n= \sum_{k=2}^n \sum_{m=1}^{k-1} \frac{g(k, m)}{m^{1/2}(k-m)^{1/2}} . \end{equation}
and by Lemma \ref{CLT for ET},
\begin{equation}
\label{eq:triangle_mean}
\lim_{m \to \infty,\ k-m \to \infty} g(k, m)=\Exp T(Y_1,Y_2):= \lambda . \end{equation}
So, for every $\eps >0$, there exists $m_0 \in \Z_+$ such that for any $m \geq m_0$ and $k-m \geq m_0$ we have $|g(k,m)-\lambda|\leq \eps$.

For the upper bound of $ \Exp A_n$, Separate the inner sum as 
\begin{align*}
\Exp A_n 
& = \left(\sum_{k=2}^{m_0} + \sum_{k=m_0 +1}^n \right) \sum_{m=1}^{k-1}\ \frac{g(k, m)}{m^{1/2}(k-m)^{1/2}} \\
& = \sum_{k=m_0 +1}^n \sum_{m=1}^{k-1}\frac{g(k, m)}{m^{1/2}(k-m)^{1/2}} + O(1) \\
&= \sum_{k=m_0 +1}^n \left(\sum_{m=1}^{m_0} + \sum_{m=k-m_0}^{k-1} +\sum_{m=m_0+1}^{k-m_0 -1} \right) \frac{g(k, m)}{m^{1/2}(k-m)^{1/2}} + O(1) ,
\end{align*}
where
\begin{align} \label{star1}
& \sum_{k=m_0 +1}^n \left(\sum_{m=1}^{m_0} + \sum_{m=k-m_0}^{k-1} \right) \frac{g(k, m)}{m^{1/2}(k-m)^{1/2}} \notag \\ 
& \leq m_0 \sum_{k=m_0 +1}^n \frac{\max_{1 \leq m \leq m_0} g(k, m)}{(k-m_0)^{1/2}} + m_0 \sum_{k=m_0 +1}^n \frac{\max_{k-m_0 \leq m \leq k} g(k, m)}{(k-m_0)^{1/2}} \notag \\ 
& \leq \lambda' \sum_{k=m_0 +1}^n \frac{2 m_0}{(k-m_0)^{1/2}}, \quad \hbox{since}\ \max_{1 \leq k,m \leq n} g(k,m) < \infty, \notag \\ 
& \leq O(n^{1/2}) ,
\end{align}
where $\lambda'$ is some constant, 
and
$$ \sum_{k=m_0 +1}^n \sum_{m=m_0+1}^{k-m_0 -1} \frac{g(k, m)}{m^{1/2}(k-m)^{1/2}} 
\leq  (\lambda+\eps) \sum_{k=2}^n \sum_{m=1}^{k-1} \frac{1}{m^{1/2}(k-m)^{1/2}} .$$
By Lemma \ref{double sum limit},
$$\limsup_{n \to \infty} \frac{1}{n} \sum_{k=m_0 +1}^n \sum_{m=m_0+1}^{k-m_0 -1} \frac{g(k, m)}{m^{1/2}(k-m)^{1/2}} \leq (\lambda + \eps)\pi .$$
Hence, $\limsup_{n \to \infty} n^{-1} \Exp A_n \leq (\lambda + \eps) \pi$ by \eqref{star1}.
So $\limsup_{n \to \infty} n^{-1} \Exp A_n \leq \lambda \pi$, since $\eps >0$ was arbitrary.

For the lower bound
\begin{align*} 
\Exp A_n 
& \geq  \sum_{k=2}^n \sum_{m=m_0}^{k-m_0} \frac{g(k,m)}{m^{1/2}(k-m)^{1/2}} \\
& \geq (\lambda-\eps) \sum_{k=2}^n \sum_{m=m_0}^{k-m_0} \frac{1}{m^{1/2}(k-m)^{1/2}} \\
& \geq (\lambda-\eps) \sum_{k=2}^n \left(\sum_{m=1}^{k-1} - \sum_{m=1}^{m_0 -1} -\sum_{m=k-m_0+1}^{k-1} \right) \frac{1}{m^{1/2}(k-m)^{1/2}} \\
& \geq (\lambda-\eps) \sum_{k=2}^n \sum_{m=1}^{k-1} \frac{1}{m^{1/2}(k-m)^{1/2}} - (\lambda-\eps)\sum_{k=2}^n \frac{2(m_0 -1)}{(k-1)^{1/2}} .
\end{align*}
By Lemma \ref{double sum limit}, $\liminf_{n \to \infty} n^{-1} \Exp A_n \geq (\lambda - \eps)\pi$.
Therefore $\liminf_{n \to \infty} n^{-1} \Exp A_n \geq \lambda \pi$, since $\eps >0$ was arbitrary. Then the result follows.
\end{proof}

\begin{lemma} 
\label{lemma:EA-zero}
If  $Y_1$, $Y_2$ are iid. rvs. $Y_1, Y_2 \sim \NN(\0, \Sigma)$ and $\Sigma = \Exp (Z_1^T Z_1)$
Then, 
\[ \Exp T(Y_1,Y_2) = \frac{1}{2} \sqrt{ \det \Sigma } . \]
\end{lemma}

\begin{proof}

With $\Sigma = (\Sigma^{1/2})^2$, we have that $(Y_1, Y_2)$ is equal in distribution to $(\Sigma^{1/2} W_1, \Sigma^{1/2} W_2)$
where $W_1$ and $W_2$ are independent $\cN (0, I)$ random vectors. Since $\Sigma^{1/2}$ acts as a linear transformation on $\R^2$
with Jacobian $\sqrt{ \det \Sigma}$,
\[ \Exp T(Y_1,Y_2) = \Exp T (\Sigma^{1/2} W_1, \Sigma^{1/2} W_2) = \sqrt{ \det \Sigma } \Exp T (W_1, W_2 ) .\]
 Here 
$$\Exp T (W_1, W_2 ) = \frac{1}{2} \Exp [ \| W_1 \| \| W_2 \| \sin \Theta ],$$
where the minimum angle $\Theta$ between $W_1$ and $W_2$ is uniform on $[0, \pi]$, and $(\| W_1\|, \|W_2\|, \Theta)$
are independent. Hence  
$$\Exp T (W_1, W_2 ) = \frac{1}{2} ( \Exp \| W_1 \| )^2 ( \Exp \sin \Theta ) = \frac12,$$
using the fact that $\Exp \sin \Theta = 2/\pi$ and $\| W_1 \|$ is the square-root of a $\chi_2^2$ random variable, so $\Exp \| W_1 \| = \sqrt{ \pi/2}$ and the result follows.
\end{proof}

\begin{theorem} 
\label{prop:EA-zero}
 Suppose that $\Exp \|Z_1\|^2 < \infty$ and $\mu=0$. 
Then, 
\[ \lim_{n \to \infty} n^{-1}\Exp A_n = \frac{\pi}{2} \sqrt{ \det \Sigma } . \]
\end{theorem}
\begin{proof}
The result follows from Proposition \ref{lim E A/n with 0 drift} combining with Lemma \ref{lemma:EA-zero}.
\end{proof}

\begin{theorem}
\label{prop:EA-drift}
 Suppose that \eqref{ass:moments} holds for some $p > 2$, $\mu \neq 0$,
and $\sperp >0$.
 Then 
\[ \lim_{n \to \infty} n^{-3/2} \Exp A_n = \| \mu \|  ( \sperp)^{1/2} \Exp \tilde a_1 = \frac{1}{3} \| \mu \|  \sqrt{2\pi \sperp} . \]
In particular,  $\Exp \tilde a_1 = \frac{1}{3} \sqrt{2 \pi}$.
\end{theorem}
\begin{proof}
Recall that $\tilde a_1 = \cA(\tilde h_1)$ is the convex hull area of the space-time diagram of one-dimensional Brownian motion run for unit time.

Given  $\Exp [ \|Z_1\|^p ] < \infty$ for some $p >2$, Proposition \ref{lem:A_moments}(i) shows that
$\Exp [ A_n^{p/2} ] = O ( n^{3p/4} )$, so that 
$\Exp [ ( n^{-3/2} A_n )^{p/2} ]$
is uniformly bounded. Hence $ n^{-3/2} A_n$ is uniformly integrable,
so Corollary \ref{cor:A-limit-drift} implies that 
\begin{equation}
\label{eq:EA-scaling-drift}
 \lim_{n \to \infty} n^{-3/2} \Exp A_n = \| \mu \|  ( \sperp)^{1/2} \Exp \tilde a_1.
\end{equation}

In light of \eqref{eq:EA-scaling-drift},    it 
remains to identify $\Exp \tilde a_1= \frac{1}{3} \sqrt{2 \pi}$. It does not seem straightforward to work directly with the Brownian limit;
it turns out again to be
simpler to work with a suitable random walk.
We choose a walk that is particularly convenient for
computations.

Let $\xi \sim \cN (0,1)$ be a standard normal random variable,
and take $Z$ to be distributed as $Z = ( 1, \xi )$ in Cartesian coordinates. Then $S_n = ( n , \sum_{k=1}^n \xi_k )$ is
the space-time diagram of the symmetric random walk on $\R$ generated by i.i.d.\ copies $\xi_1, \xi_2, \ldots$ of $\xi$.

For $Z = (1, \xi)$, $\mu = (1,0)$ and $\sigma^2 = \sperp = \Exp [ \xi^2 ] = 1$. Thus 
by \eqref{eq:EA-scaling-drift}, to complete the proof of Theorem \ref{prop:EA-drift}
it suffices to show  that for this walk
$\lim_{n\to \infty} n^{-3/2} \Exp A_n = \frac{1}{3} \sqrt{2 \pi}$.
If $u, v \in \R^2$ have Cartesian components $u = (u_1, u_2)$ and $v = (v_1, v_2)$, then we
may write $T ( u, v) = \frac{1}{2} | u_1 v_2 - v_1 u_2 |$. Hence
\begin{align*}
T ( S_m, S_k - S_m ) & = \frac{1}{2} \left| ( k-m) \sum_{j=1}^m \xi_j - m \sum_{j=m+1}^k \xi_j \right| .\end{align*}
By properties of the normal distribution, the right-hand side of the last display has the same distribution as $ \frac{1}{2} | \xi  \sqrt{ k m (k-m)}  |$.
Hence
\[ \frac{\Exp T ( S_m, S_k - S_m )}{\sqrt{ m (k-m) } } = \frac{1}{2} \Exp |  \xi  \sqrt{k} | = \frac{1}{2} \sqrt{ 2 k / \pi } ,\]
using the fact that $| \xi |$ is distributed as the square-root of a $\chi_1^2$ random variable, so $\Exp | \xi | = \sqrt{ 2 / \pi }$.
Hence, by \eqref{eq:bnb}, this random walk enjoys the exact formula
\begin{align*} \Exp A_n & = \frac{1}{\sqrt{2\pi}} \sum_{k=2}^n \sum_{m=1}^{k-1} \frac{\sqrt{k}}{\sqrt{ m (k-m) }} . \end{align*}
Then from \eqref{eq:pi-sum} we obtain
$\Exp A_n \sim   \sqrt{\pi /2} \sum_{k=2}^n k^{1/2}$, which gives the result.
\end{proof}
\begin{remark}
The idea used in the proof of Theorem~\ref{prop:EA-drift},
first establishing the existence of a limit for a class of models
and then choosing a particular model for which the limit can be conveniently
evaluated, goes back at least to Kac; see~\cite[p.\,293]{kac}.
\end{remark}

\section{Law of large numbers for the area}

\begin{proposition} \label{LLN for A_n with 0 drift}
Suppose $\Exp (\|Z_1\|^4) < \infty$ and $\|\Exp Z_1 \| =0$. Then for any $\alpha >1$, $n^{-\alpha} A_n \to 0$ a.s. as $n \to \infty$.
\end{proposition}

\begin{proof}
By Chebyshev's inequality for $A_n$,
$$ \Pr \left(\frac{|A_n - \Exp A_n|}{n^{\alpha}} \geq \eps \right) = \Pr(|A_n - \Exp A_n| \geq \eps n^{\alpha}) \leq \frac{\Var(A_n)}{\eps^2 n^{2 \alpha}} .$$
Since $\Var(A_n) = O(n^2)$ by Proposition \ref{lem:A_moments}(ii), for any $\alpha >1$, as $n \to \infty$ we have
$$\Pr \left(\frac{|A_n - \Exp A_n|}{n^{\alpha}} \geq \eps \right) = O(n^{2-2\alpha}) .$$
So $n^{-\alpha} (A_n - \Exp A_n) \to 0$ in probability.

Take $n=n_k=2^k$ for $k \in \N$, we have
$$\Pr \left(\frac{|A_{n_k} - \Exp A_{n_k}|}{n_k^{\alpha}} \geq \eps \right) = O(n_k^{2-2\alpha}) = O(4^{k(1-\alpha)}) .$$
So for any $\eps > 0$,
$$\sum_{k=1}^{\infty} \Pr \left(\frac{|A_{n_k} - \Exp A_{n_k}|}{n_k^{\alpha}}\geq \eps \right) < \infty .$$
By Borel--Cantelli Lemma (Lemma \ref{borel-cantelli}), as $k \to \infty$
$$\frac{A_{n_k} - \Exp A_{n_k}}{n_k^{\alpha}} \to 0 ~\as$$
By Proposition \ref{lem:A_moments}(ii), $n_k^{-\alpha} \Exp A_{n_k} \to 0$ as $n \to \infty$, we get 
$$\frac{A_{n_k}}{n_k^{\alpha}} \to 0 ~ \as \As k \to \infty.$$

For any $n \in \N$, there exists $k(n) \in N$ such that $2^{k(n)} \leq n < 2^{k(n)+1}$. By monotonicity of $A_n$,
$$2^{-\alpha} \frac{A_{n_{k(n)}}}{n_{k(n)}^{\alpha}} = \frac{A_{2^{k(n)}}}{(2^{k(n)+1})^{\alpha}} \leq \frac{A_n}{n^{\alpha}} \leq \frac{A_{2^{k(n)+1}}}{(2^{k(n)})^{\alpha}} = 2^{\alpha} \frac{A_{n_{k(n)+1}}}{n_{k(n)+1}^{\alpha}} .$$
The result follows by the Squeezing Theorem.
\end{proof}

\begin{proposition}
Suppose $\Exp (\|Z_1\|^4) < \infty$. Then, for any $\alpha > 3/2$, $n^{-\alpha} A_n \to 0$ a.s. as $n \to \infty$. 
\end{proposition}

\begin{proof}
By Chebyshev's inequality for $A_n$,
$$ \Pr \left(\frac{|A_n - \Exp A_n|}{n^{\alpha}} \geq \eps \right) = \Pr(|A_n - \Exp A_n| \geq \eps n^{\alpha}) \leq \frac{\Var(A_n)}{\eps^2 n^{2 \alpha}} .$$
Since $\Var(A_n) = O(n^3)$ by Proposition \ref{lem:A_moments}(i), for any $\alpha >3/2$, as $n \to \infty$ we have
$$\Pr \left(\frac{|A_n - \Exp A_n|}{n^{\alpha}} \geq \eps \right) = O(n^{3-2\alpha}) .$$
So $n^{-\alpha} (A_n - \Exp A_n) \to 0$ in probability.

Take $n=n_k=2^k$ for $k \in \N$, we have
$$\Pr \left(\frac{|A_{n_k} - \Exp A_{n_k}|}{n_k^{\alpha}} \geq \eps \right) = O(n_k^{3-2\alpha}) = O(4^{k(3/2-\alpha)}) .$$
So for any $\eps > 0$,
$$\sum_{k=1}^{\infty} \Pr \left(\frac{|A_{n_k} - \Exp A_{n_k}|}{n_k^{\alpha}}\geq \eps \right) < \infty .$$
By Borel--Cantelli Lemma (Lemma \ref{borel-cantelli}), as $k \to \infty$
$$\frac{A_{n_k} - \Exp A_{n_k}}{n_k^{\alpha}} \to 0 ~\as$$
By Proposition \ref{lem:A_moments}(i), $n_k^{-\alpha} \Exp A_{n_k} \to 0$ as $n \to \infty$, we get 
$$\frac{A_{n_k}}{n_k^{\alpha}} \to 0 ~ \as \As k \to \infty.$$

For any $n \in \N$, there exists $k(n) \in N$ such that $2^{k(n)} \leq n < 2^{k(n)+1}$. By monotonicity of $A_n$,
$$2^{-\alpha} \frac{A_{n_{k(n)}}}{n_{k(n)}^{\alpha}} = \frac{A_{2^{k(n)}}}{(2^{k(n)+1})^{\alpha}} \leq \frac{A_n}{n^{\alpha}} \leq \frac{A_{2^{k(n)+1}}}{(2^{k(n)})^{\alpha}} = 2^{\alpha} \frac{A_{n_{k(n)+1}}}{n_{k(n)+1}^{\alpha}} .$$
The result follows by the Squeezing Theorem.
\end{proof}

\section{Asymptotics for the variance}

Recall that Proposition \ref{prop:var-limit-zero u0} shows $\lim_{n \to \infty} n^{-1} \Var L_n = u_0 ( \Sigma )$. 
In this section, we will show that
\begin{alignat}{2}
\label{eq:three_vars}
\text{if } \mu \neq 0: ~~ &  \lim_{n \to \infty} n^{-3} \Var A_n   = v_+ \| \mu \|^2 \sperp ; \nonumber\\
\text{if } \mu = 0: ~~  &   \lim_{n \to \infty} n^{-2} \Var A_n   = v_0 \det \Sigma  
 .\end{alignat}
The quantities $ v_0$ and $v_+$ in \eqref{eq:three_vars} are finite and positive, 
as is $u_0( \blob )$ provided $\sigma^2 \in (0,\infty)$,
and these quantities are in fact variances associated with convex hulls of Brownian scaling limits for the walk. 

\begin{proposition}
\label{prop:var-limit-zero}
 Suppose that \eqref{ass:moments} holds for some $p >4$, 
and $\mu =0$. Then
\[ \lim_{n \to \infty} n^{-2} \Var A_n = v_0 \det \Sigma.\]
\end{proposition}
\begin{proof}
Lemma \ref{lem:A_moments}(ii) shows that
$\Exp [ A_n^{2(p/4)} ] = O ( n^{p/2} )$, so that 
$\Exp [ ( n^{-2} A^2_n )^{p/4} ]$
is uniformly bounded. Hence $ n^{-2} A_n^2$ is uniformly integrable,
and we deduce convergence of $n^{-2} \Var A_n$ in Corollary \ref{cor:zero-limits}.
\end{proof}

For the case with drift, we have the following variance result.

\begin{proposition}
\label{prop:var-limit-drift v+}
Suppose that \eqref{ass:moments} holds for some $p > 4$ and $\mu \neq 0$.
Then
\[ \lim_{n \to \infty} n^{-3} \Var A_n = v_+ \| \mu \|^2 \sperp.\]
\end{proposition}
\begin{proof}
Given  $\Exp [ \|Z_1\|^p ] < \infty$ for some $p >4$, Lemma \ref{lem:A_moments}(i) shows that
$\Exp [ A_n^{2(p/4)} ] = O ( n^{3p/4} )$, so that 
$\Exp [ ( n^{-3} A^2_n )^{p/4} ]$
is uniformly bounded. Hence $ n^{-3} A_n^2$ is uniformly integrable,
so Corollary \ref{cor:A-limit-drift} yields the result.
\end{proof}

\section{Variance bounds}
 
\begin{proposition}
\label{prop:var_bounds v0 v+}
We have $u_0 (\Sigma) =0$ if and only if $\trace \Sigma =0$.
The following inequalities for the quantities defined at \eqref{eq:var_constants} hold.
\begin{alignat}{1}
0 < \frac{4}{49} \left( \re^{- 7\pi^2 / 12} 
- 
\frac{1}{3} \re^{-21 \pi^2 / 4}
\right)^2  &{} \leq  v_0   \leq  16 (\log 2)^2 - \frac{\pi^2}{4}; \label{eq:v0-bounds} \\
0 < \frac{2}{225} \left( 
\re^{-25 \pi/9}
-\frac{1}{3} 
\re^{-25 \pi}
\right) &{} \leq  v_+   \leq 4 \log 2 - \frac{2 \pi}{9}. \label{eq:v1-bounds}
\end{alignat}
\end{proposition}

\begin{proof}
Bounding $\tilde a_1$ by the area of a rectangle, we have
\begin{equation}
\label{eq:a1-upper}
\tilde a_1 \leq r_1 \leq 2 \sup_{0 \leq s \leq 1} | w (s) |, \as , \end{equation}
where $r_1 := \sup_{0 \leq s \leq 1} w (s) - \inf_{0 \leq s \leq 1} w (s)$. 
A result of Feller \cite{feller} states that $\Exp [ r_1^2 ] = 4 \log 2$.
So by the first inequality in \eqref{eq:a1-upper}, we have $\Exp [\tilde a_1^2] \leq 4 \log 2$, 
 and by Theorem \ref{prop:EA-drift}
we have $\Exp \tilde a_1 = \frac{1}{3} \sqrt{ 2 \pi}$; the upper bound in \eqref{eq:v1-bounds} follows.

Similarly, for any orthonormal basis $\{ e_1, e_2\}$ of $\R^2$, we bound $a_1$ by a rectangle
\[ a_1 \leq \left( \sup_{0 \leq s \leq 1} e_1 \cdot b(s) - \inf_{0 \leq s \leq 1 } e_1 \cdot b(s) \right) 
   \left( \sup_{0 \leq s \leq 1} e_2 \cdot b(s) - \inf_{0 \leq s \leq 1 } e_2 \cdot b(s) \right) ,\]
and the two (orthogonal) components are independent, so $\Exp [ a_1^2 ] \leq ( \Exp [ r_1^2 ] )^2 = 16 (\log 2)^2$,
which with the fact that $\Exp a_1 = \frac{\pi}{2}$ gives the upper bound in \eqref{eq:v0-bounds}.

We now move on to the lower bounds. Tractable upper bounds for $a_1$ and $\tilde a_1$ are easier to come by than
lower bounds, and thus we obtain a lower bound on the variance by showing
the appropriate area has positive probability of being smaller than the corresponding mean.

Consider $a_1$; note $\Exp a_1 = \pi/2$ \cite{elbachir}. Since, for any orthonormal basis $\{e_1, e_2\}$ of $\R^2$,
\[ a_1 \leq \pi \sup_{0 \leq s \leq 1} \| b (s) \|^2 \leq \pi \sup_{0 \leq s \leq 1 } |  e_1 \cdot b(s) |^2 +  \pi \sup_{0 \leq s \leq 1 } |  e_2\cdot b(s) |^2 ,\]
 using the fact that $e_1 \cdot b$ and $e_2 \cdot b$ are independent one-dimensional Brownian motions,
\[ \Pr [ a_1 \leq r ] \geq \Pr \left[   \sup_{0 \leq s \leq 1 } | w (s)   |^2 \leq \frac{r}{2\pi}  \right]^2  , ~ \text{for} ~ r >0 .\]
We apply \eqref{eq:var-bound} with $X = a_1$ and $\alpha \in (0,1)$, and set $r = (1-\alpha) \frac{\pi}{2}$ to obtain
\begin{align*}
\Var\, a_1 & \geq \alpha^2 \frac{\pi^2}{4}  \Pr \left[   \sup_{0 \leq s \leq 1 } |  w (s) | \leq \frac{\sqrt{1-\alpha}}{2}  \right]^2 \\
& \geq 4 \alpha^2 \left(  \exp \left\{-\frac{\pi^2}{2(1-\alpha) } \right\} 
- \frac{1}{3}  \exp \left\{-\frac{9 \pi^2}{2(1- \alpha) } \right\}  \right)^2 ,\end{align*}
by \eqref{eq:brown_norm}.
Taking $\alpha = 1/7$ is close to optimal, and gives the lower bound in \eqref{eq:v0-bounds}.

For $\tilde a_1$, we apply \eqref{eq:var-bound} with $X = \tilde a_1$ and $\alpha \in (0,1)$.
Using the fact
that $\Exp \tilde a_1 = \frac{1}{3} \sqrt{2 \pi}$ (from Theorem \ref{prop:EA-drift}) and the weaker of the two bounds in \eqref{eq:a1-upper},
we obtain
\begin{align*}
\Var\, \tilde a_1 & \geq \alpha^2 \frac{2 \pi}{9} \Pr \left[ \sup_{0 \leq s \leq 1} |  w (s) | \leq \frac{ (1-\alpha) \sqrt{2 \pi} }{6} \right]  \\
& \geq  \frac{8}{9} \alpha^2  \left( 
\exp \left\{-\frac{9\pi}{4 (1-\alpha)^2 } \right\} 
- \frac{1}{3}  \exp \left\{-\frac{81\pi}{4 (1-\alpha)^2 } \right\}   \right) ,\end{align*}
by \eqref{eq:brown_norm}.
Taking $\alpha = 1/10$ is close to optimal, and gives the lower bound in \eqref{eq:v1-bounds}.
\end{proof}

\begin{remark}
The main interest of the lower bounds in Proposition \ref{prop:var_bounds v0 v+} is that they are \emph{positive};
they are certainly not sharp. The bounds can surely be improved. We note just the following idea.
A lower bound for $\tilde a_1$ can be obtained by conditioning on $\theta := \sup \{ s \in [0,1] : w(s) =0\}$
and using the fact that the maximum of $w$ up to time $\theta$ is distributed as the maximum of a scaled Brownian bridge;
combining this with the previous argument improves the lower bound on $v_+$ to $2.09 \times 10^{-6}$.\\
\end{remark}

%% file: chapter7.tex
\pagestyle{myheadings} \markright{\sc Chapter 7}

\chapter{Conclusions and open problems}
\label{chapter7}

\section{Summary of the limit theorems}

We summarize in general the asymptotic behaviour of the expectation and variance of $L_n$ and $A_n$ as the following table.

\begin{table}[!h]
\center
\def\arraystretch{1.4}
\begin{tabular}{cc|ccc}
 & &  limit exists for $\Exp$ & limit exists for $\Var$ & limit law \\
\hline
\multirow{2}{*}{ $\mu = 0$ } 
 &  $L_n$ & $n^{-1/2} \Exp L_n$$^\mathsection$  & $n^{-1} \Var L_n$  & non-Gaussian  \\
  & $A_n$ & $n^{-1} \Exp A_n$$^{\mathparagraph}$  & $n^{-2} \Var A_n$  & non-Gaussian  \\
\hline 
 \multirow{2}{*}{ $\mu \neq 0$ }  
 &  $L_n$ & $n^{-1} \Exp L_n$$^\mathsection$$^\dagger$  & $n^{-1} \Var L_n$$^\ddagger$ &  Gaussian$^\ddagger$ \\
  & $A_n$ & $n^{-3/2} \Exp A_n$ & $n^{-3} \Var A_n $ & non-Gaussian
\end{tabular}
\caption{Results originate from: $\mathsection$\!\cite{sw}; $\dagger$\!\cite{ss}; $\ddagger$\!\cite{wx};  $\mathparagraph$\!\cite{bnb} (in part);
the rest are new.
The limit laws exclude degenerate cases when associated variances vanish.}
\label{table1}
\end{table}

Table \ref{table x3} collets the lower and upper bounds and simulation estimates for the constants defined at equation \eqref{eq:var_constants} and equation \eqref{eq:three_vars}.

\begin{table}[!h]
\center
\def\arraystretch{1.4}
\begin{tabular}{c|ccc}
        &   lower bound   & simulation estimate  & upper bound \\
\hline
  $u_0 ( I)$ &  $2.65 \times 10^{-3}$  &  1.08   &  9.87   \\
  $v_0$      &  $8.15  \times 10^{-7}$ &  0.30   &  5.22   \\
  $v_+$      &  $1.44  \times 10^{-6}$ &  0.019  &  2.08   
	\end{tabular}
\caption{Each of the simulation estimates is
 based on $10^5$ instances of a walk of length $n = 10^5$. The final decimal digit in each of the numerical upper (lower)
bounds has been rounded up (down).}
\label{table x3}
\end{table}

Claussen et al. \cite{claussen} give some numerical estimations that $\Var\, l_1 \approx 1.075$ and $\Var\, a_1 \approx 0.31$, which is a good agreement with our limit estimations $1.08$ and $0.30$.

\section{Exact evaluation of limiting variances}
\label{sec:exact evaluation of limiting variances}

It would, of course, be of interest to evaluate any of $u_0$, $v_0$, or $v_+$ exactly. 
In general this looks hard. The paper \cite{rs} provides a key component to a possible approach to evaluating $u_0$. 
By Cauchy's formula and Fubini's theorem, 
\[ \Exp [ \ell_1^2 ] = \int_{\Sp_1} \int_{\Sp_1} 
\Exp \left[ \left( \sup_{0 \leq s \leq 1} ( e_1 \cdot b (s)   ) \right)
\left( \sup_{0 \leq t \leq 1} ( e_2 \cdot b (t)   ) \right) \right]
\ud e_1 \ud e_2 .\]
Here, the two standard one-dimensional Brownian motions $e_1 \cdot b$ and $e_2 \cdot b$
have correlation determined by the cosine of the angle $\phi$ between them, i.e.,
\[ \Exp \left[  ( e_1 \cdot b (s)   )  ( e_2 \cdot b (t)   )  \right]
= ( s \wedge t ) \, e_1 \cdot e_2
= ( s \wedge t ) \cos   \phi   .\]
The result of Rogers and Shepp \cite{rs} then shows that
\[ \Exp \left[ \left( \sup_{0 \leq s \leq 1} ( e_1 \cdot b (s)   ) \right)
\left( \sup_{0 \leq t \leq 1} ( e_2 \cdot b (t)   ) \right) \right]
= c ( \cos   \phi  ) ,\]
where the function $c$ is given explicitly in \cite{rs}.
Using this result, we obtain
\[ \Exp [ \ell_1^2 ] = 4 \pi \int_{-\pi/2}^{\pi/2} c ( \sin \theta ) \ud \theta 
= 4 \pi \int_{-\pi/2}^{\pi/2}  \ud \theta \int_0^\infty \ud u  \cos \theta
\frac{ \cosh (u \theta  )}
{ \sinh ( u \pi /2 ) } \tanh \left( \frac{ (2 \theta + \pi) u }{4} \right)  .
\]
We have not been able to deal with this integral analytically, but numerical
integration gives $\Exp [ \ell_1^2 ] \approx 26.1677$, which with the fact that $\Exp \ell_1 = \sqrt{ 8 \pi}$
gives
$u_0(I) = \Var \ell_1 \approx 1.0350$,   in reasonable agreement with the
simulation estimate in Table~\ref{table2}.
 
Another possible approach to evaluating $u_0$ is suggested by a remarkable computation of Goldman \cite{goldman} for the analogue of $u_0(I)= \Var \ell_1$ for the planar \emph{Brownian bridge}. Specifically, if
$b'_t$ is the standard Brownian bridge in $\R^2$ with $b'_0
= b'_1 = 0$, and $\ell'_1 = \cL ( \hull b' [0,1] )$ the perimeter length of its convex hull, 
 \cite[Th\'eor\`eme 7]{goldman} states that 
\[ \Var   \ell'_1 {} = {} 
\frac{\pi^2}{6} \left( 2 \pi \int_0^\pi \frac{ \sin \theta}{\theta} \ud \theta -   2 - 3 \pi \right) \approx 0.34755 . \]

\section{Open problems}

\subsection{Degenerate case for $L_n$ when $\mu \neq 0$ and $\sigma_\mu^2 =0$}
\label{sec:degenerate case}

Recall Remark \ref{rmk:degenerate}(iii) for Theorem \ref{thm1}. For example, consider
$$ Z_1 = \begin{cases}\phantom{2} 
(1,1), & \text{with probability } 1/2; \\
(1,-1), & \text{with probability } 1/2. 
\end{cases} $$

Then the $\spara$ in Theorem \ref{thm1} is zero and our results on the second-order properties 
of $L_n$ in Chapter \ref{chapter5} can not be applied in this degenerate case.
See Figure \ref{fig:degenerate} for an example of random walk in this case.

\begin{figure}[h!]
  \centering
	\includegraphics[width=\textwidth]{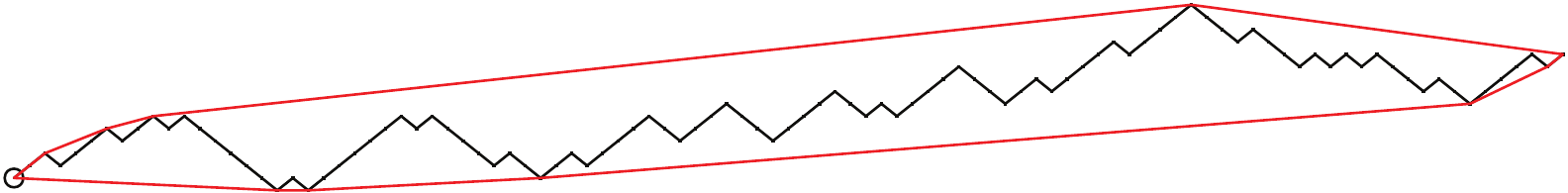} 
  \caption{Example of the degenerate case with $n=100$.} \label{fig:degenerate}
\end{figure}

For this example, we conjecture $\frac{\Var L_n}{\log n} \to \text{constant}$, based on some simulations. See Figure \ref{fig:deg sim1} below.

\begin{figure}[h!]
  \centering
	\includegraphics[width=0.6\textwidth]{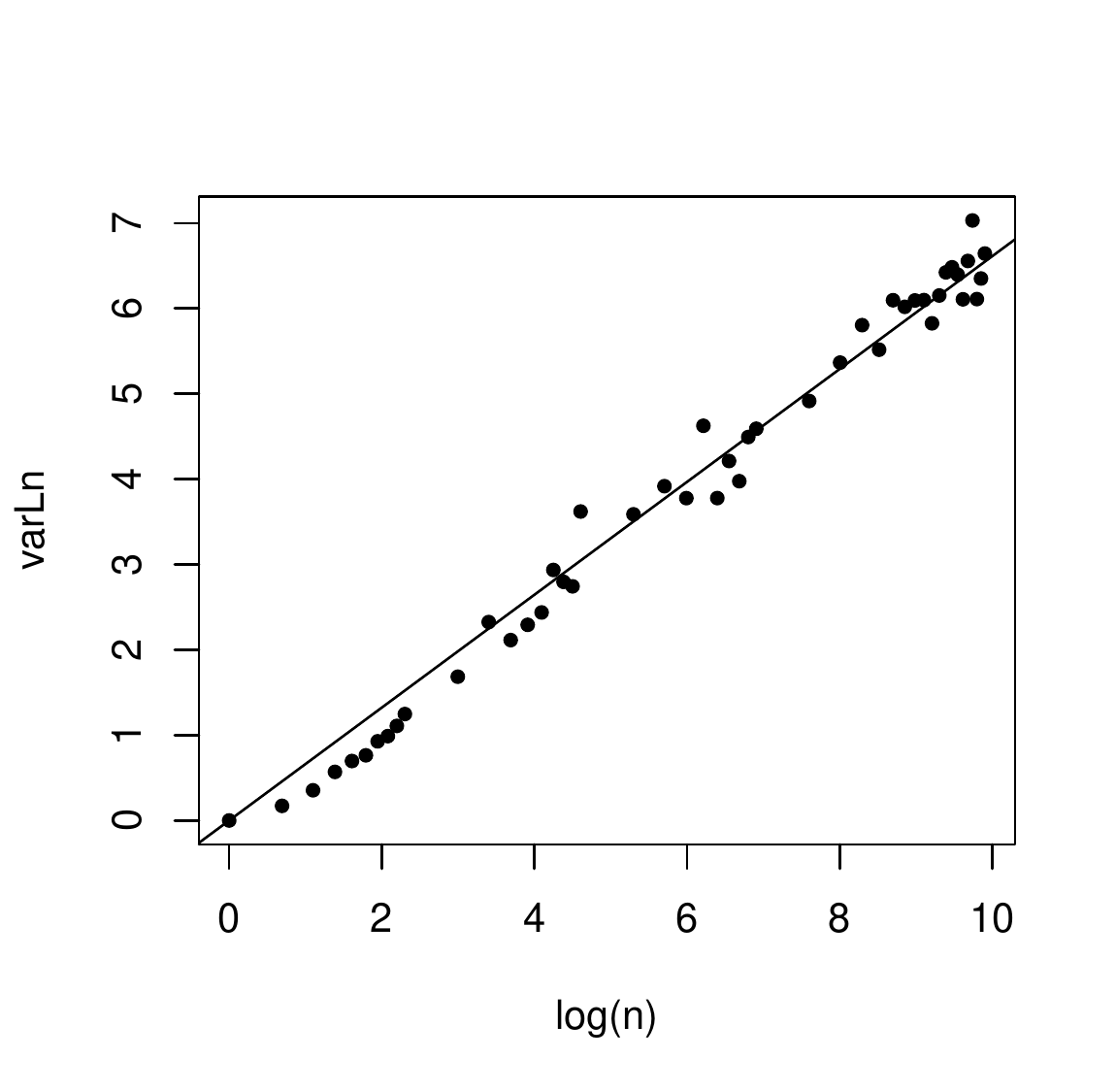} 
  \caption{Simulation for the degenerate case $\Var L_n =0.6612 \log(n)$.} \label{fig:deg sim1}
\end{figure}

A second open question is whether in this case $\frac{L_n - \Exp L_n}{\sqrt{\Var L_n}}$ has a distributional limit. If so, is that limit normal? 
We conjecture that there is a limit, but it is not normal (see Figure \ref{fig:deg normal}).

\begin{figure}[h!]
  \centering
	\includegraphics[width=0.49\textwidth]{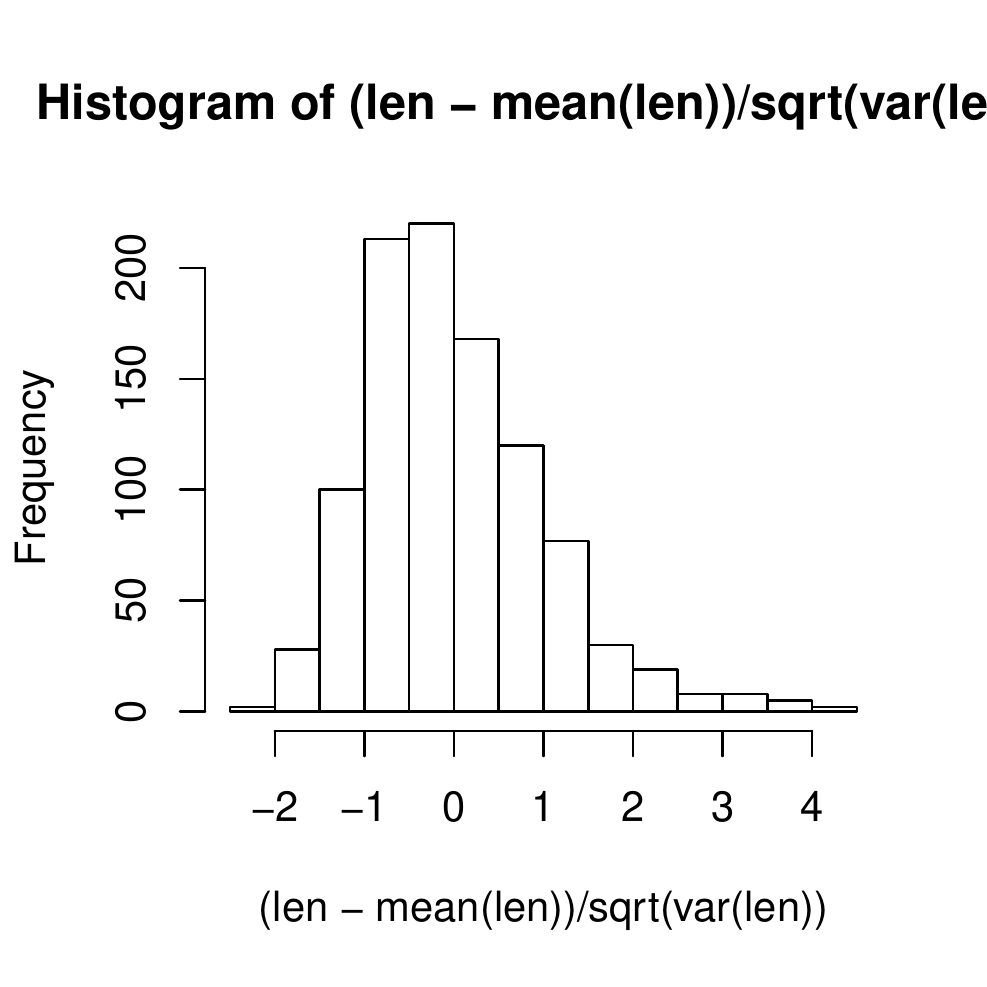} 
	\includegraphics[width=0.49\textwidth]{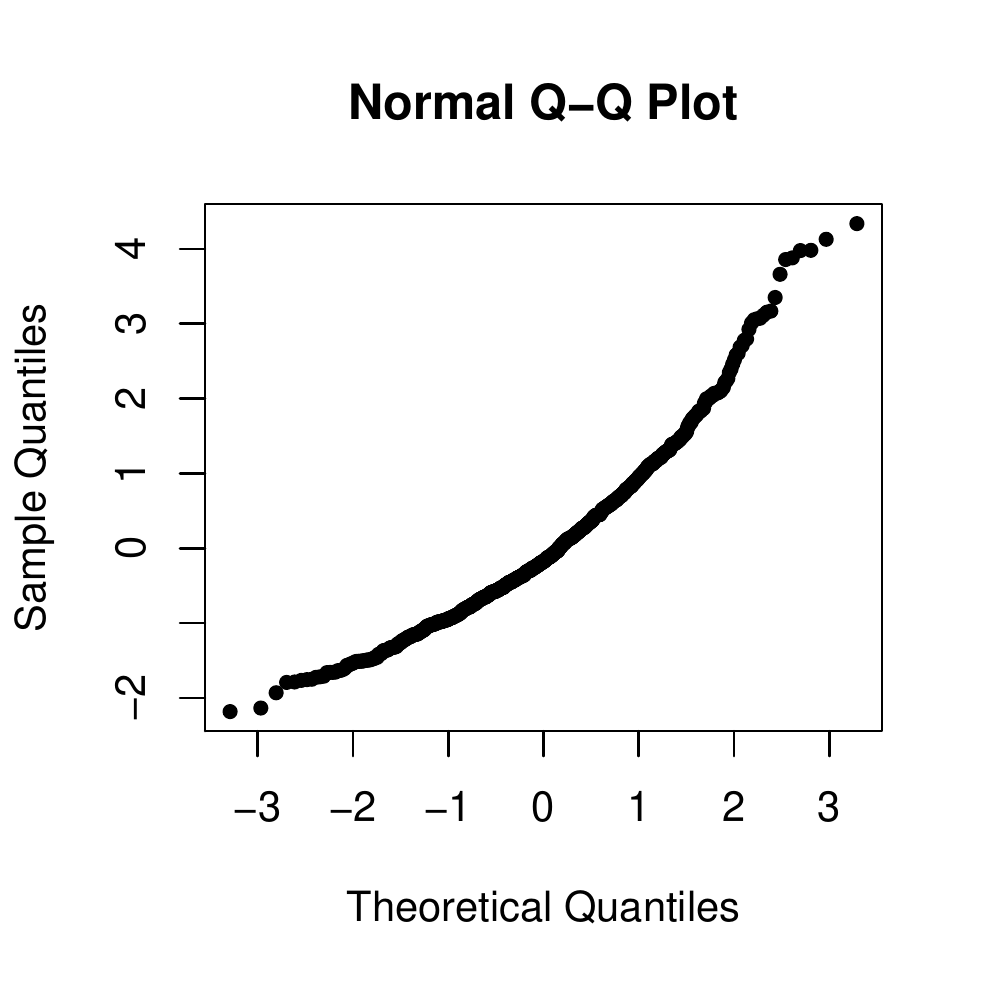} 
  \caption{Simulations for the degenerate case.} \label{fig:deg normal}
\end{figure}

\subsection{Heavy-tailed increments}

All main results from previous chapters are based on the assumption \ref{ass:moments} for $p=2$, that the second moments of increments are finite. 
But what happens in the heavy-tail problems, in which $\Exp( \|Z_1\|^2) = \infty$? We give two simulation examples.

\subsection{Centre-of-mass process}

We can associate to a random walk trajectory $S_0, S_1, S_2, \ldots$ its \emph{centre-of-mass}
process $G_0, G_1, G_2, \ldots$ defined by $G_0 := S_0 = 0$ and for $n \geq 1$ by
$G_n = \frac{1}{n} \sum_{k=1}^n S_k$.
By convexity, the convex hull of $\{ G_0, G_1, \ldots, G_n \}$ is contained in the convex hull of $\{ S_0, S_1, \ldots, S_n\}$.
What can one say about its perimeter length or area?
Note that one may express $G_n$ as a weighted sum of the increments of the walk as 
\[ G_n = \sum_{k=1}^n \left( \frac{n-k+1}{n} \right) Z_k .\]
Then, for example, we expect that the method of Section~\ref{sec:CLT for drift}
carries through to this case; this is one direction for future work.

\subsection{Higher dimensions}

Most of the analysis of $L_n$ in this thesis is restricted to $d=2$ because we rely on the Cauchy formula for planar convex sets. In higher dimensions,
the analogues of $L_n$ and $A_n$ are the \emph{intrinsic volumes} of the convex body. Analogues of Cauchy's formula are available,
but these seem more difficult to use as the basis for analysis.

However, the scaling limit theories in Chapter~\ref{chapter3} may have some relatively straightforward corollaries in higher dimensions. 
So, some analogous results for $A_n$ in Chapter~\ref{chapter6} may not be so difficult to figure out.

%% file: thesis.bbl
\begin{thebibliography}{99}

\bibitem{andersen} E.\ S.\ Andersen, On the fluctuations of sums of random variables II, 
{\em Math.\ Scand.}\ {\bf 2} (1954) 195--223.

\bibitem{asmussen} S.\ Asmussen, 
\emph{Applied Probability and Queues}, 2nd ed., Springer-Verlag, New York, 2003.

\bibitem{bnb} O.\ Barndorff--Nielsen and G.\ Baxter,
Combinatorial lemmas in higher dimensions,
\emph{Trans.\ Amer.\ Math.\ Soc.}\ {\bf 108} (1963) 313--325.

\bibitem{barnett1} V.\ Barnett, The ordering of multivariate data, {\em J.\ Roy.\ Statist.\ Soc.\ Ser.\ A}\ {\bf 139} (1976) no.3, 318--355.

\bibitem{barnett2} V.\ Barnett, Outliers and order statistics, {\em Comm.\ Statist.\ Theory\ Methods}\ {\bf 17} (1988) no.7, 2109--2118.

\bibitem{barnett-lewis} V.\ Barnett and T.\ Lewis, \emph{Outliers in statistical data}, 3rd ed., Wiley, Chichester-New York-Brisbane, 1994.

\bibitem{blvc} F.\ Bartumeus, M.\ G.\ E.\ Da Luz, G.\ M.\ Viswanathan and J.\ Catalan, Animal search strategies: a quantitative random-walk analysis, 
{\em Ecology}\ {\bf 86} no.11, (2005) 3078--3087.

\bibitem{bass} R.\ F.\ Bass, Markov processes and convex minorants, \emph{Seminaire de Probabilities}, LNM (1982) 1059.

\bibitem{baxter} G.\ Baxter, A combinatorial lemma for complex numbers,
{\em Ann.\ Math.\ Statist.}\ {\bf 32} (1961) 901--904.

\bibitem{bill} P.\ Billingsley, \emph{Convergence of Probability Measures}, 2nd ed., Wiley, New York, 1999.

\bibitem{burdzy} K.\ Burdzy, Brownian paths and cones, \emph{Ann.\ Prob.}\ {\bf 13} no.3, (1985) 1006--1010.

\bibitem{claussen} G.\ Claussen, A.\ K.\ Hartmann, and S.\ N.\ Majumdar, Convex hulls of random walks: Large-deviation properties,
\emph{Phys.\ Rev.}\ {\bf 91} (2015) 052104.

\bibitem{chm} M.\ Cranston, P.\ Hsu, and P.\ March, Smoothness of the convex hull of planar Brownian motion,
\emph{Ann.\ Probab.}\ {\bf 17} (1989) 144--150.

\bibitem{chung} K.\ L.\ Chung, \emph{A Course in Probability Theory}, 3rd ed., Academic Press, San Diego, 2001.

\bibitem{chung-fuchs} K.\ L.\ Chung and W.\ H.\ J.\ Fuchs, On the distribution of values of sums of random variables,
  \emph{Mem.\ Amer.\ Math.\ Soc.}\ {\bf 6} (1951).

\bibitem{codling} E.\ A.\ Codling, M.\ J.\ Plank and S.\ Benhamou,
Random walk models in biology,
\emph{J. R. Soc. Interface} {\bf 5} (2008) 813--834.

\bibitem{durrett} R.\ Durrett, \emph{Probability: Theory and Examples},
Wadsworth \& Brooks/Cole, Pacific Grove, CA, 1991.


\bibitem{efron} B.\ Efron, The convex hull of a random set of points, {\em Biometrika}\ {\bf 52} no. 3/4, (1965) 331--343.

\bibitem{elbachir} M.\ El Bachir, \emph{L'enveloppe convex du mouvement Brownien}, Ph.D. thesis, Universit\'e Toulouse III---Paul Sabatier, 1983.

\bibitem{eldan} R.\ Eldan, Volumetric properties of the convex hull of an n-dimensional Brownian motion, \emph{Electron.\ J.\ Prob.}\ 
{\bf 19} no.45, (2014) 1--34.

\bibitem{evans} S.\ N.\ Evans, On the {H}ausdorff dimension of {B}rownian cone points, \emph{Math.\ Proc.\ Camb.\ Philos.\ Soc.}\ {\bf 98} (1985) 343--353. 

\bibitem{feller} W.\ Feller, The asymptotic distribution of the range of sums of independent random variables,
\emph{Ann.\ Math.\ Statist.}\ {\bf 22} (1951) 427--432.

\bibitem{feller2} W.\ Feller, 
\emph{An Introduction to Probability Theory and its Applications. Vol. II.},
2nd ed., Wiley, New York, 1971.

\bibitem{geffroy} J.\ Geffroy, Localisation asymptotique du poly\`edre d'appui d'un \'echantillon laplacien \`a $k$ dimensions, \emph{Publ.\ Inst.\ Stat.\ Univ.\ Paris}\ {\bf 10} (1961) 213--228. 

\bibitem{gph} L.\ Giuggioli, J.\ R.\ Potts and S.\ Harris, Animal interactions and the emergence of territoriality, {\em PLoS.\ Comput.\ Biol.}\ {\bf 7}(3) (2011) e1002008.\ doi:10.1371/journal.pcbi.1002008

\bibitem{glendinning} R.\ H.\ Glendinning, The convex hull of a dependent vector-valued process, {\em J.\ Statist.\ Comput.\ Simul.}\ {\bf 38} (1991) 219--237.

\bibitem{goldman} A.\ Goldman, Le spectre de certaines mosa\"iques poissoniennes du plan et l'enveloppe convex du pont brownien,
\emph{Probab.\ Theory Relat.\ Fields} {\bf 105} (1996) 57--83.

\bibitem{green} P.\ J.\ Green, Peeling bivariate data, pp.~3--19 in {\em Interpreting Multivariate Data}, V.\ Barnett (ed.), Wiley, 1981.

\bibitem{gruber} P.\ M.\ Gruber, \emph{Convex and Discrete Geometry}, Springer, Berlin, 2007.

\bibitem{gut} A.\ Gut, \emph{Probability: A Graduate Course}, Springer, Uppsala, 2005.

\bibitem{hug} D.\ Hug, Random polytopes, Chapter 7 in \emph{Stochastic Geometry, Spatial Statistics and Random Fields}, Springer, 2013.

\bibitem{hughes} B.\ Hughes, \emph{Random Walks and Random Environments, Vol. I.}, Oxford, 1995.

\bibitem{jp} N.\ C.\ Jain and W.\ E.\ Pruitt, The other law of the iterated logarithm,
\emph{Ann.\ Probab.}\ {\bf 3} (1975) 1046--1049.

\bibitem{kac} M.\ Kac, 
Toeplitz matrices, translation kernels and a related problem in probability theory,
{\em Duke\ Math.\ J.}\ {\bf 21} (1954) 501--509.

\bibitem{kallenberg} O.\ Kallenberg, \emph{Foundations of Modern Probability}, 2nd ed., Springer, New York, 2002.

\bibitem{klm} J.\ Kampf, G.\ Last, and I.\ Molchanov, On the convex hull of symmetric stable processes,
\emph{Proc.\ Amer.\ Math.\ Soc.}\ {\bf 140} (2012) 2527--2535.

\bibitem{kt2} S.\ Karlin, and H.\ M.\ Taylor, 
\emph{A Second Course in Stochastic Processes}, Academic Press, 
New York, 1981.

\bibitem{letac} G.\ Letac,
An explicit calculation of the mean of the perimeter of the convex hull of a plane random walk,
{\em J.\ Theor.\ Prob.}\ {\bf 6} (1993) 385--387.

\bibitem{letac2} G.\ Letac, Advanced problem 6230, \emph{Amer.\ Math.\ Monthly} {\bf 85} (1978) 686.

\bibitem{levybook} P.\ L\'evy, \emph{Processus Stochastiques et Mouvement Brownien},
Gauthier-Villars, Paris, 1948.

\bibitem{mardia} K.\ V.\ Mardia, J.\ T.\ Kent and J.\ M.\ Bibby, \emph{Multivariate Analysis}, Academic Press, London, 1979.

\bibitem{mcr} S.\ N.\ Majumdar, A.\ Comtet, and J.\ Randon-Furling,
Random convex hulls and extreme value statistics,
{\em J.\ Stat.\ Phys.}\ {\bf 138} (2010) 955--1009.

\bibitem{mohr} C.\ O.\ Mohr, Table of equivalent populations of north American small mammals, {\em Amer. Midland Naturalist}\ {\bf 37} no.1, (1947) 223--249.

\bibitem{morters} P.\ M\"orters and Y.\ Peres, \emph{Brownian Motion}, Cambridge, 2010.

\bibitem{nevzorov} V.\ B.\ Nevzorov, \emph{Records: Mathematical Theory}, Amer. Math. Soc., 2001.

\bibitem{penrose} M.\ Penrose, \emph{Random Geometric Graphs}, Oxford, 2003.


\bibitem{pitman-ross} J.\ Pitman and N.\ Ross, The greatest convex minorant of Brownian motion, meander, and bridge,
{\em Probab.\ Theory Relat.\ Fields}\ {\bf 153} (2012) 771--807.


\bibitem{polya} G.\ P\'olya, \"Uber eine Aufgabe der Wahrscheinlichkeitsrechnung betreffend die Irrfahrt im Strassennetz,
  \emph{Math.\ Ann.}\ {\bf 84} (1921) 149-–160.

\bibitem{reitzner} .\ Reitzner, pp.~45--76 in \emph{New Perspectives
in Stochastic Geometry}, W.S.~Kendall \& I.~Molchanov (eds.), OUP, 2010.

\bibitem{renyi-sulanke} A.\ R\'enyi and R.\ Sulanke, \"Uber die konvexe h\"ulle von $n$ zuf\"allig gew\"ahlten punkten, {\em Z.\ Wahrscheinlichkeitstheorie}\ {\bf 2} (1963) 75--84.

\bibitem{resnick} S. Resnick, 
\emph{Adventures in Stochastic Processes}, Birkh\"auser, Boston, 1992.

\bibitem{rs} L.C.G.\ Rogers and L.\ Shepp, 
The correlation of the maxima of correlated {B}rownian motions,
\emph{J.\ Appl.\ Probab.}\ {\bf 43} (2006) 880--883. 

\bibitem{rudin} W.\ Rudin, \emph{Principles of Mathematical Analysis}, 3rd ed., McGraw-Hill, 1976.

\bibitem{schneider-weil} R.\ Schneider and W.\ Weil, Classical stochastic geometry, pp.~1--42 in \emph{New Perspectives
in Stochastic Geometry}, W.S.~Kendall \& I.~Molchanov (eds.), OUP, 2010.

\bibitem{sinai} Ya.\ G.\ Sinai, Convex hulls of random processes, \emph{Amer.\ Math.\ Soc.\ Transl.}\ {\bf 186} (1998).

\bibitem{smouse} P.E. Smouse, S. Focardi, P.R. Moorcroft, J.G. Kie, J.D. Forester and J.M. Morales,
Stochastic modelling of animal movement,
\emph{Phil. Trans. R. Soc. B} {\bf 365} (2010) 2201--2211.

\bibitem{ss} T.\ L.\ Snyder and J.\ M.\ Steele, Convex hulls of random walks,
{\em Proc.\ Amer.\ Math.\ Soc.}\ {\bf 117} (1993) 1165--1173.

\bibitem{sw} F.\ Spitzer and H.\ Widom, The circumference of a convex polygon,
{\em Proc.\ Amer.\ Math.\ Soc.}\ {\bf 12} (1961) 506--509.


\bibitem{steele} J.\ M.\ Steele, The {B}ohnenblust--{S}pitzer algorithm and its applications,
{\em J.\ Comput.\ Appl.\ Math.}\ {\bf 142} (2002) 235--249.

\bibitem{steele2} J.\ M.\ Steele, \emph{Probability Theory and Combinatorial Optimization}, Soc. for Industrial and Applied Math., 1997.

\bibitem{takacs} L.\ Tak\'acs, Expected perimeter length, \emph{Amer.\ Math.\ Monthly} {\bf 87} (1980) 142.

\bibitem{takacs2} L.\ Tak\'acs, 
\emph{Combinatorial Methods in the Theory of Stochastic Processes},
 Wiley, New York, 1967.

\bibitem{wx} A.R.\ Wade and C.\ Xu, Convex hulls of planar random walks with drift, {\em Proc.\ Amer.\ Math.\ Soc.}\ {\bf 143} (2015) 433--445.

\bibitem{wx2} A.R.\ Wade and C.\ Xu, Convex hulls of planar random walks and their scaling limits, {\em Stoc.\ Proc.\ and\ their\ Appl.}\ {\bf 125} (2015) 4300--4320.

\bibitem{worton1} B.\ J.\ Worton, A review of models of home range for animal movement, {\em Ecol.\ Modelling}\ {\bf 38} (1987) 277--298.

\bibitem{worton2} B.\ J.\ Worton, A convex hull-based estimator of home-range size, {\em Biometrics}\ {\bf 51} no.4, (1995) 1206--1215.

\bibitem{yukich} J.\ E.\ Yukich, \emph{Probability Theory of Classical Euclidean Optimization Problems}, Springer, 1998.


\end{thebibliography}
